\documentclass[a4paper,10pt]{article}
\usepackage[vlined]{algorithm2e}
\usepackage[symbol]{footmisc}

\SetFuncSty{textsc}
\SetKwFunction{frun}{Run}
\SetKwHangingKw{arun}{\frun}
\SetKwFunction{fset}{Set}
\SetKwHangingKw{aset}{\fset}
\SetKwFunction{fselect}{Select}
\SetKwHangingKw{aselect}{\fselect}
\SetKwFunction{fcompute}{Compute}
\SetKwHangingKw{acompute}{\fcompute}
\SetKwFunction{fsolve}{Solve}
\SetKwHangingKw{asolve}{\fsolve}
\SetKwFunction{festimate}{Estimate}
\SetKwHangingKw{aestimate}{\festimate}
\SetKwFunction{fmark}{Mark}
\SetKwHangingKw{amark}{\fmark}
\SetKwFunction{frefine}{Refine}
\SetKwHangingKw{arefine}{\frefine}
\SetKwFunction{compute}{Compute}
\SetKwFunction{set}{Set}

{\algorithm}%
{\endalgorithm}% Figure/float layout

\title{Parameter dependent finite element analysis for  ferronematics solutions
}
\author{\stepcounter{footnote} \stepcounter{footnote} \stepcounter{footnote} Ruma Rani Maity\footnote{
		Department of Mathematics, Indian Institute of Technology Bombay, Powai, Mumbai 400076, India. Email. ruma@math.iitb.ac.in} $\;\;$
	Apala Majumdar* \footnote{ Department of Mathematics And Statistics, University of Strathclyde, 16 Richmond St, Glasgow G1 1XQ, United Kingdom. Visiting Professor,  Indian Institute of Technology Bombay, Powai, Mumbai 400076, India. Email. apala.majumdar@strath.ac.uk}        
	$\;\;$ Neela Nataraj\footnote{Department of Mathematics, Indian Institute of Technology Bombay, Powai, Mumbai 400076, India. Email. neela@math.iitb.ac.in}
}

\usepackage{amsmath,amsthm,amssymb,enumerate, amsbsy}
\usepackage[T1]{fontenc}
\usepackage{gensymb}
\usepackage[sc]{mathpazo}
\usepackage{multirow} 
\usepackage{newtxtext,newtxmath}
\usepackage{subfig}
\usepackage{graphicx}
\usepackage{epstopdf}
\usepackage{hyperref}
\usepackage{cancel}
\usepackage[margin=3cm]{geometry}
\usepackage{tikz,pstricks}

\usepackage{epsfig}
 \usepackage{pst-grad} % For gradients
 \usepackage{pst-plot} % For axes
 \usepackage[space]{grffile} % For spaces in paths
 \usepackage{etoolbox} % For spaces in paths
 \makeatletter % For spaces in paths
 \patchcmd\Gread@eps{\@inputcheck#1 }{\@inputcheck"#1"\relax}{}{}
\usepackage[titletoc]{appendix}
\usepackage{caption}
\usetikzlibrary{shapes,calc}
\usepackage{verbatim}
\usepackage{mathrsfs}
\usepackage{accents}
\usepackage[utf8]{inputenc}
\usepackage{float}

\restylefloat{table}

\usetikzlibrary{decorations.pathmorphing}
\usetikzlibrary{decorations.pathreplacing}
\usetikzlibrary{positioning}
\usetikzlibrary{shapes}
\usetikzlibrary{arrows}
\usetikzlibrary{patterns}
\usetikzlibrary{fadings}
\usetikzlibrary{plotmarks}
\usetikzlibrary{calc}
\usetikzlibrary{intersections}
\tikzstyle{every picture}+=[font=\footnotesize]
\usepackage{paralist}
\usepackage{bbm}
\usepackage{latexsym}           %benoetigt fuer \Box am Beweisende
\usepackage{enumerate}
\usepackage{enumitem}

\setlist{noitemsep, topsep=0.8ex, partopsep=0pt%, parsep=0pt, itemsep=0pt
	, leftmargin=3em}
\setlist[1]{labelindent=\parindent}

\newlist{axioms}{enumerate}{1}
\setlist[axioms]{font=\bfseries}

\newlist{alphenum}{enumerate}{1}
\setlist[alphenum]{label=\textbf{(\alph*)}, leftmargin=4em}

\newlist{alphienum}{enumerate}{1}
\setlist[alphienum]{label=\textit{(\alph*)}}

\newlist{romanenum}{enumerate}{1}
\setlist[romanenum]{label=\textit{(\roman*)}}

\newlist{romaninenum}{enumerate*}{1}
\setlist[romaninenum]{label=\textit{(\roman*)}}
\usepackage[vlined]{algorithm2e}
\SetKwIF{If}{ElseIf}{Else}{if}{}{else if}{else}{endif}
\SetKwFor{For}{for}{}{endfor}
\usepackage[noabbrev, capitalise]{cleveref}
\usepackage{tabu}
\crefname{equation}{\unskip}{\unskip}
\creflabelformat{equation}{#2(#1)#3}
\newtheorem{thm}{Theorem}[section]

\newtheorem{lem}[thm]{Lemma}

\newtheorem{prop}[thm]{Proposition}

\theoremstyle{definition}

\newtheorem{defn}{Definition}[section]

\theoremstyle{remark}

\newtheorem{rem}[thm]{Remark}

\numberwithin{equation}{section}

\usepackage[many]{tcolorbox}
\usetikzlibrary{shadows}

\newtcolorbox{shadedbox}{
	drop shadow southeast,
	breakable,
	enhanced jigsaw,
	colback=white,
}

\usepackage{listings}
\usepackage{xcolor}
\definecolor{codegreen}{rgb}{0,0.6,0}
\definecolor{codegray}{rgb}{0.5,0.5,0.5}
\definecolor{codepurple}{rgb}{0.58,0,0.82}
\definecolor{backcolour}{rgb}{0.95,0.95,0.92}
\lstdefinestyle{mystyle}{
	backgroundcolor=\color{backcolour},   
	commentstyle=\color{codegreen},
	keywordstyle=\color{magenta},
	numberstyle=\tiny\color{codegray},
	stringstyle=\color{codepurple},
	basicstyle=\ttfamily\footnotesize,
	breakatwhitespace=false,         
	breaklines=true,                 
	captionpos=b,                    
	keepspaces=true,                 
	numbers=left,                    
	numbersep=5pt,                  
	showspaces=false,                
	showstringspaces=false,
	showtabs=false,                  
	tabsize=2
}

\newcommand{\bTheta}{{\Theta}}

\newcommand{\Qvec}{\mathbf{Q}}
\newcommand{\Mvec}{\mathbf{M}}

\newcommand{\e}{\mathbf{e}}

\newcommand{\V}{\mathbf{V}}
\newcommand{\X}{\mathbf{X}}
\newcommand{\h}{\mathbf{H}}

\newcommand{\abs}[1]{\left\lvert#1\right\rvert}
\newcommand{\dx}{\,{\rm dx}}

\newcommand{\ds}{\,{\rm ds}}

\newcommand{\norm}[1]{{\vert\kern-0.25ex\vert #1 
		\vert\kern-0.25ex \vert}}
\newcommand{\vertiii}[1]{{\vert\kern-0.25ex\vert\kern-0.25ex\vert #1 
		\vert\kern-0.25ex \vert\kern-0.25ex \vert}}
\newcommand{\vertiiih}[1]{{\vert\kern-0.25ex \vert\kern-0.25ex \vert #1 
		\vert\kern-0.25ex \vert\kern-0.25ex \vert}_{h}}
\newcommand{\verti}[1]{{\vert\kern-0.25ex \vert\kern-0.25ex \vert #1 
		\vert\kern-0.25ex \vert\kern-0.25ex \vert}_{1}}	
	\newcommand{\vertiiii}[1]{{\vert\kern-0.25ex\vert\kern-0.25ex\vert #1 
			\vert\kern-0.25ex \vert\kern-0.25ex \vert}}
		
	\newcommand{\vertii}[1]{{ \vert\kern-0.25ex \vert\kern-0.25ex \vert #1 
		 \vert\kern-0.25ex \vert\kern-0.25ex \vert}_{1}}	
	\newcommand{\vertiiiih}[1]{{ \vert\kern-0.25ex \vert\kern-0.25ex \vert #1 
			 \vert\kern-0.25ex \vert\kern-0.25ex \vert}_{h}}
\newcommand{\vertiiidg}[1]{{\vert\kern-0.25ex \vert\kern-0.25ex \vert #1 
		\vert\kern-0.25ex \vert\kern-0.25ex\vert}_{h}}

\newcommand{\dual}[1]{\langle #1 \rangle}

\theoremstyle{plain}

\begin{document}
		\maketitle
		\begin{abstract}
			
	\medskip
	
	\noindent	This paper focuses on the analysis of a free energy functional, that models a dilute suspension of magnetic nanoparticles in a two-dimensional nematic well.  %study of global minimizers of the energy functional that models a dilute suspension of magnetic nano-particles in a nematic filled micron sized wells in two dimensional geometry. 
	The {\it first part} of the article is devoted to the asymptotic analysis of global energy minimizers in the limit of vanishing elastic constant, $\ell \rightarrow  0$ where the re-scaled elastic constant $\ell$ is inversely proportional to the domain area. %discusses the limit of small rescaled elastic constant, $\ell \propto \frac{1}{\text{domain length}^2},$ 
	The first results concern the strong $H^1$-convergence and a $\ell$-independent $H^2$-bound for the global minimizers on smooth bounded 2D domains, with smooth boundary and topologically trivial Dirichlet conditions. %$\mathbf{g}$ is smooth on the domain boundary $\partial \Omega$ with deg$(\mathbf{g}, \partial \Omega)=0.$ 
	The {\it second part} focuses on the discrete approximation of regular solutions of the corresponding non-linear system of partial differential equations with cubic non-linearity and non-homogeneous Dirichlet boundary conditions. We establish (i) the existence and local uniqueness of the discrete solutions using fixed point argument, (ii) a best approximation result in energy norm, (iii) error estimates in the energy and $L^2$ norms with $\ell $- discretization parameter dependency  for the conforming finite element method.  Finally, the theoretical results are complemented by numerical experiments on the discrete solution profiles, the numerical convergence rates that corroborates the theoretical estimates, followed by plots that illustrate the dependence of the discretization parameter on $\ell$. 	\end{abstract}
		
		\section{Int\(  \)roduction}
		Nematic liquid crystals (NLCs) are classical examples of partially ordered materials that combine fluidity with the directional order of crystalline solids \cite{dg}. NLCs have long-range orientational order i.e. there are distinguished material directions, referred to as \emph{nematic directors} such that the NLC properties are different in different directions e.g. along the directors. The anisotropic or direction-dependent NLC response to incident light and electric fields make them the working material of choice for the multi-billion dollar liquid crystal display (LCD) industry \cite{scalia_review}. However, NLC devices largely rely on their dielectric anisotropy i.e. the NLC response to external electric fields depends on whether the electric field is parallel or non-parallel to the nematic directors \cite{dg}. NLC devices rarely use external magnetic fields since the magnetic anisotropy is typically much smaller than the NLC dielectric anisotropy, so that NLC responses to external magnetic fields are very weak \cite{Brochard_DE_Gennes_1970}. In the 1970's, Brochard and de Gennes \cite{Brochard_DE_Gennes_1970} proposed that a suspension of magnetic nanoparticles in a nematic host could generate a spontaneous magnetization at room temperature, without any external field, substantially enhancing the NLC responses to external magnetic fields. These composite systems have been labelled as \emph{ferronematics} in the literature. In 2013, Mertelj et al. \cite{Mertelj_Lisjak_Drofenik_Copic_2013} experimentally designed  stable ferronematic suspensions using barium hexaferrite (BaHF)
		magnetic nanoplatelets in pentylcyano-biphenyl (5CB)
		liquid crystals. In ferronematic suspensions, the nematic director is coupled to the suspended magnetic nanoparticles through surface interactions, so that the nematic director influences the magnetic moments of the nanoparticles and vice-versa. In fact, this nemato-magnetic coupling induces averaged orientations of the suspended nanoparticles, leading to a spontaneous magnetization, in addition to the ambient nematic order. Consequently, this nemato-magnetic coupling strongly enhances the optical and magnetic responses of this composite system, making them attractive for 
		novel display devices, sensors \cite{Cirtoaje_Petrescu_Stan_Creanga_2016},
		telecommunications \cite{Mertelj_Lisjak_2017}, and potentially pharmaceutical applications too.
		
		In this paper, we model a dilute suspension of magnetic nanoparticles in a NLC-filled two-dimensional (2D) domain, with tangent boundary conditions. The tangent boundary conditions require the nematic directors to be tangent to the boundary e.g. for a square domain, the director is tangent to the square edges, naturally creating some sort of mismatch or discontinuity at the square vertices with two intersecting edges. The domain is typically on the micron scale, and the volume fraction of the suspended nanoparticles is small such that the distance between the nanoparticles is usually large compared to the typical nanoparticle size. In the dilute limit, one does not see the individual nanoparticles or the interactions between pairs of nanoparticles, but rather the entire suspension is modelled as a single system with two order parameters: a reduced Landau-de Gennes $\mathbf{Q}$-tensor order parameter with two independent components and a magnetization vector, $\Mvec:=(M_1, M_2)$ which models the induced spontaneous magnetization of the suspended nanoparticles \cite{Ferronematics_2D}. This reduced approach works well for thin NLC systems i.e. NLCs confined to a shallow three-dimensional (3D) system, with a 2D cross-section, such that the height is much smaller than the cross-sectional dimensions \cite{golovaty2015}, with tangent boundary conditions on the bounding surfaces. 
		
		Here  $ \Qvec\in \mathbf{S}_0:=\lbrace \Qvec=(Q_{ij})_{1\leq i,j \leq 2} \in \mathbb{M}^{2 \times 2}: \Qvec=\Qvec^T, \text{tr}\Qvec=0 \rbrace ,$ 	where $\displaystyle \Qvec:= s(2\mathbf{n}\otimes\mathbf{n}-\textit{I})$. The nematic director in the plane is a 2D unit-vector, $\mathbf{n}:=(\cos\theta,\sin\theta),$ where  $\theta(x,y)$ is the director angle, that models the preferred in-plane direction of the NLC molecules. The scalar order parameter $'s(x,y)'$ measures the degree of order about $\mathbf{n}$, so that the zero set of $s$ is identified with the set of planar nematic defects, and  $I$ is a $2 \times 2$ identity matrix. The independent component of $\Qvec$ are given by $Q_{11}= s\cos2\theta$ and $Q_{12}= s\sin2\theta.$ Building on the work in \cite{Ferronematics_2D}, \cite{Burylov_Raikher_1995}, \cite{calderer2014},   \cite{Hanetal2021Ferro2D}, the free energy of this dilute suspension of magnetic nanoparticles in a NLC-filled square well is given by
		\begin{eqnarray}
		    \label{eq:1}
		   &&  E\left[\Qvec, \mathbf{M} \right]: = \int_{\Omega} \big(\frac{K}{2}\left| \nabla \Qvec \right|^2 + \frac{A}{2}|\Qvec|^2 + \frac{C}{4}|\Qvec|^4 \big) \dx  \nonumber \\ && + \int_{\Omega}\big(\frac{\kappa}{2}\left|\nabla \mathbf{M}\right|^2 + \frac{\alpha}{2}|\mathbf{M}|^2 + \frac{\beta}{4}|\mathbf{M}|^4 \big) ~\dx  \nonumber 
		   \\ && - \int_{\Omega} \frac{\gamma \mu_0}{2} \mathbf{M}^T \Qvec \mathbf{M}~\dx,
		\end{eqnarray}
		where $\Omega$ is a 2D domain with characteristic length $L$ microns, $K$ and $\kappa$ are the nematic and magnetic stiffness constants respectively, $A$ is the re-scaled temperature as is $\alpha$, $C$ and $\beta$ are positive material-dependent constants and $\gamma \mu_0$ is the nemato-magnetic coupling parameter. Working at low temperatures requires $A$ and $\alpha$ to be negative, as has been done in this manuscript. The first line is the reduced 2D Landau-de Gennes NLC free energy, the second line is the magnetic energy and the third line is the nemato-magnetic coupling energy. The Dirichlet energy density of $\mathbf{M}$ is introduced to penalize arbitrary rotations between $\mathbf{M}$ and $-\mathbf{M}$, and can be viewed as a regularization term. Some authors argue that this term should be discarded for dilute suspensions but we retain it for the well-posedness of the associated variational problems. Further, if $\gamma >0$, then the coupling energy coerces $\mathbf{n}\cdot \mathbf{M} \approx \pm 1 $ whereas if $\gamma < 0$, then the coupling energy coerces $\mathbf{n}\cdot \mathbf{M}\approx 0$ almost everywhere. Using the re-scalings as in \cite{Ferronematics_2D} and \cite{Hanetal2021Ferro2D} and assuming that $\frac{K}{|A|} = \frac{\kappa}{|\alpha|}$ (an idealised assumption for analytical convenience), the ferronematic free energy reduces to 
		\begin{align}\label{energy functional ferronematics}
		\mathcal{E}(\Qvec,\Mvec):=& \int_\Omega \frac{1}{2}(\abs{\nabla Q_{11}}^2+\abs{\nabla Q_{12}}^2+\abs{\nabla M_{1}}^2+\abs{\nabla M_{2}}^2) \dx+ \frac{1}{\ell}\int_\Omega f_B(\Qvec,\Mvec)\dx ,
		\end{align}
	where the re-scaled domain $\Omega$ has unit characteristic length, and $\ell = \frac{K}{|A| L^2}$ depends on the nematic elastic constant, temperature and domain size $L$. The first integral is the elastic energy  of $\Qvec$ and $\mathbf{M}$ whereas $f_B$  is the quartic bulk energy density given by:
		\begin{align}\label{bulk energy potential}
		f_B(\Qvec,\Mvec) :=\frac{1}{4}(Q_{11}^2 +Q_{12}^2 -1)^2+\frac{1}{4}( M_1^2 +M_2^2 - 1)^2 - \frac{c}{2}\big(Q_{11}(M_1^2 - M_2^2) +2Q_{12}M_1M_2\big)
		\end{align}
		where the coupling parameter, $c = \frac{\gamma \mu_0}{|A|}\sqrt{\frac{C}{2|A|}}\frac{|\alpha|}{\beta}.$  In other words, the sign of $c$ is determined by the sign of $\gamma$ and has the same implications for the nemato-magnetic coupling as $\gamma$. 
		For any $\Qvec:=(Q_{11}, Q_{12})$ and $\Mvec:=(M_{1}, M_{2})$, the admissible space is, $\mathcal{A}:=\{\Psi:=(Q_{11}, Q_{12},M_{1}, M_{2}) \in \mathbf{H}^1(\Omega)|\, \Psi=\mathbf{g} \text{ on } \partial \Omega \},$ with $\mathbf{H}^1(\Omega):=( {H}^1(\Omega))^4$ and a given Dirichlet boundary condition $\mathbf{g}$ (see Section \ref{Asymptotic analysis of the minimizers}). 
		\noindent The existence of the global energy minimizers of $\mathcal{E}$ in the admissible space follows from a direct method in the calculus of variations; \cite[Section 8.2]{Evance19}  the crucial facts are that the admissible space  is non-empty; the energy functional $\mathcal{E}$ is coercive and convex in gradient of the variables $(Q_{11}, Q_{12}, M_1, M_2)$. Setting $\tilde{Q}:=\textit{Q}_{11}^2+ \textit{Q}_{12}^2 -1, $ $\tilde{M}:=M_1^2 + M_2^2-1$, the corresponding Euler-Lagrange equations are given by
	\begin{align}\label{continuous nonlinear ferro strong form}
	\Delta \textit{Q}_{ij}-\ell^{-1} (\widetilde{Q} \textit{Q}_{ij}-c(M_i M_j - \delta_{ij}\abs{M}^2/2)) =0
	\text{ and } 
	\Delta M_i -	\ell^{-1} (\widetilde{M}M_i - c \textit{Q}_{ij}M_j)=0, \quad i,j=1,2.
	\end{align} 	
	
	When $c=0$, the ferronematic free energy simply reduces to the celebrated Ginzburg-Landau free energy for superconductors \cite{Brezis_Bethual_Book}. This is a very well-studied problem, and arises naturally in reduced 2D Landau-de Gennes descriptions of confined NLCs \cite{golovaty2015}. From a purely analytic point of view, this 2D problem has been addressed in a batch of papers  \cite{han_majumdar_zhang_siap}, \cite{canevari_majumdar_spicer} etc. where the authors analyse the reduced minimizers on 2D polygons and obtain powerful results on the multiplicity of minimizers, the dimensionality, structure and locations of the corresponding defect sets and bifurcations as a function of the domain size.  In \cite{MultistabilityApalachong}, the authors study the reduced 2D NLC model on square domains and numerically compute the solution branches using finite element methods, investigating the effect of surface anchoring effects on the energy minimizers. An {\it a priori } and  {\it a posteriori  } error analysis for Nitsche's and the discontinuous Galerin methods have been analyzed for the reduced model (with $c=0$) \cite{DGFEM,AposterioriRMAMNN}. A structure-preserving finite element method for the computation of equilibrium configurations, when the $\Qvec$ is constrained to be uniaxial in 3D, has been proposed in \cite{Nochetto2020}, and the stability and consistency of this method without
regularization, and $\Gamma$-convergence of the discrete energies to the continuous
solutions, in the limit of vanishing mesh size, is studied. The reader is referred to \cite{LeiZhang} for a detailed survey of the mathematical models of liquid crystals and the developments of numerical methods to find liquid crystal configurations.

For ferronematic systems with $c \neq 0$, the volume of work is limited. In \cite{JDPEFAMJX}, the authors analyze a dilute ferronematic suspension in a one-dimensional channel geometry with Dirichlet boundary conditions for both $\Qvec$ and $\Mvec$. The authors derive some key analytic ingredients - existence theorems, uniqueness theorems, maximum principle arguments and symmetric solution profiles. They also compute bifurcation diagrams for the solution branches as a function of $\ell$ and $c$. In 2D, there is fairly elaborate numerical work in \cite{Ferronematics_2D} and \cite{Hanetal2021Ferro2D} where the authors numerically compute stable ferronematic equilibria on 2D polygons with tangent boundary conditions. They report the co-existence of stable equilibria with magnetic domain walls, and stable equilibria with pairs of interior NLC defects, again paying attention to the effects of $\ell$ and $c$. In particular, there are two distinguished limits: the $\ell \to \infty$ limit for small nano-scale systems which admits a unique equilibrium, and the $\ell \to 0$ limit for large micron-scale systems with multiple stable equilibria. Their numerical results suggest that positive $c$ suppresses multistability whereas negative $c$ strongly enhances multistability for applications.

There has been little, if any, analytic and numerical analysis for 2D ferronematic systems, to the best of our knowledge. The main contributions of this paper are: (i) asymptotic analyses of minimizers of the ferronematic free energy on smooth 2D domains, with smooth boundaries and topologically trivial Dirichlet conditions in the $\ell \to 0$ limit and (ii) an a priori finite element error analysis for the discrete approximation of the solutions of the corresponding Euler-Lagrange equations and (iii) relevant numerical experiments to tie the asymptotic analysis and finite element analysis.  The asymptotic analysis includes a strong convergence result of the global energy minimizers to minimizers of the bulk energy density in $H^1$, and an uniform convergence of the global energy minimizers to the limiting maps, as $\ell \rightarrow 0$, which follows from a  non-trivial Bochner inequality for the ferronematic free energy density. The Bochner inequality allows us to prove  a $\ell$-independent $H^2$-bound for global energy minimizers in this limit. The $H^2$-bound is subsequently exploited in the finite element analysis, and the key results are: (i) an elegant representation of the non-linear operator corresponding to the coupled system of  partial differential equations (PDEs) \eqref{continuous nonlinear ferro strong form} with cubic and quadratic non-linear lower order terms and non-homogeneous boundary datum; (ii) a finite element convergence analysis that includes the existence and uniqueness of discrete solutions approximating the regular solutions of \eqref{continuous nonlinear ferro strong form} for which $\ell$-independent $H^2$-bound hold, a best approximation result in energy norm, %the theoretical optimal convergence rate of $\mathcal{O}(h)$ in energy norm, %sub-optimal convergence rate of $\mathcal{O}(h^\frac{3}{3})$ in $L^2$ norm for conforming finite element method, 
optimal convergence rates of $\mathcal{O}(h)$ and  $\mathcal{O}(h^2)$ in energy and $L^2$ norms, respectively, for a sufficiently small choice of the discretization parameter (denoted by $h$) which depends on $\ell$ and (iii) numerical results for discrete solution landscapes with positive and negative $c$, order of convergence in energy and $L^2$ norms, and  numerical errors as a function of $\ell$ and $h$.

Throughout the paper, standard notations on Sobolev spaces and their norms are employed. The standard semi-norm and norm on $H^s(\Omega)$ $(\text{resp.} \,W^{s,p}(\Omega))$ for $s,p$ positive real numbers, are denoted by $\abs{\cdot}_s$ and $\norm{\cdot}_s$ $(\text{resp. } \abs{\cdot}_{s,p} \text{ and } \norm{\cdot}_{s,p} )$. The standard $L^2(\Omega)$ inner product is denoted by $(\cdot,\cdot)$. We use the notation $\mathbf{H}^s(\Omega)$ (resp.\,$\mathbf{L}^p(\Omega)$) to denote the product space $(H^s(\Omega))^4
% \times H^s(\Omega)\times H^s(\Omega) \times H^s(\Omega)
$ $(\text{resp. } (L^p(\Omega))^4.$ %\times L^p(\Omega) \times L^p(\Omega) \times L^p(\Omega))$. 
The standard norms $\vertiii{\cdot}_s$ ($\text{resp. }\vertiii{\cdot}_{s,p}$) in the Sobolev spaces $\mathbf{H}^s(\Omega)$ ($\text{resp. }\mathbf{W}^{s,p}(\Omega)$) and defined by $\vertiii{\Phi}_s\!=(\sum_{i =1}^{4}\norm{\varphi_i}_{s}^2 )^{\frac{1}{2}} $
%$(\norm{\varphi_1}_s^2+\norm{\varphi_2}_s^2+\norm{\varphi_3}_s^2+\norm{\varphi_4}_s^2)^{\frac{1}{2}}$
for all $ \Phi\!=\!(\varphi_1, \varphi_2,\varphi_3, \varphi_4) \in \!\mathbf{H}^s(\Omega) $ $ ( \text{resp.  }\vertiii{\Phi}_{s,p}\!\!=(\sum_{i =1}^{4}\norm{\varphi_i}_{s,p}^2 )^{\frac{1}{2}}$
%$	(\norm{\varphi_i}_{s,p}^2+\norm{\varphi_2}_{s,p}^2 +\norm{\varphi_3}_{s,p}^2+\norm{\varphi_4}_{s,p}^2)^{\frac{1}{2}}$ 
for all $\Phi\!=\!(\varphi_1, \varphi_2,\varphi_3, \varphi_4) \!\in \!\mathbf{W}^{s,p}(\Omega)).$
The norm on $\mathbf{L}^2(\Omega)$ space is defined by $\vertiii{\Phi}_0\!=(\sum_{i =1}^{4}\norm{\varphi_i}_{0}^2 )^{\frac{1}{2}}$
%	$(\norm{\varphi_1}_0^2+\norm{\varphi_2}_0^2+\norm{\varphi_1}_3^2+\norm{\varphi_4}_0^2)^{\frac{1}{2}}$ 
for all $ \Phi\!=\!(\varphi_1, \varphi_2,\varphi_3, \varphi_4) \in \!\mathbf{L}^2(\Omega). $
Set $V:=H_0^1(\Omega),$ $  \,\!\V:=\mathbf{H}_0^1(\Omega)=(H_0^1(\Omega))^4
% \times H_0^1(\Omega)\times H_0^1(\Omega) \times H_0^1(\Omega)
$, and  $X:=H^1(\Omega), \, \mathbf{X}:= \mathbf{H}^1(\Omega)$. Throughout this paper, $C$ will denote a generic constant which will always be
independent of $\ell$ and the mesh parameter $h$. Define the trace spaces $\h^{m-\frac{1}{2}}(\partial \Omega):=\{ \mathbf{w}|_{\partial \Omega}:  \mathbf{w} \in \h^m(\Omega) \} $ for $m=1,2,$ and $\h^{-\frac{1}{2}}(\partial \Omega):= (\h^{\frac{1}{2}}(\partial \Omega))^*$ equipped with the norm $\displaystyle \vertiii{\mathbf{q}}_{m-\frac{1}{2},\partial \Omega}:= \inf_{\substack{ \mathbf{w} \in \h^m(\Omega)\\ \mathbf{w}|_{\partial \Omega}=\mathbf{q}}}\vertiii{\mathbf{w}}_m$ and $\displaystyle \vertiii{\mathbf{q}}_{-\frac{1}{2},\partial \Omega}:=\sup_{\substack{ \mathbf{k} \in \h^{\frac{1}{2}}(\partial\Omega) \\  \mathbf{k}\neq 0 }}
	\frac{\dual{\mathbf{q},\mathbf{k} }_{\partial \Omega}}{\vertiii{\mathbf{k}}_{\frac{1}{2},\partial \Omega}}$, respectively. Also, the operator norm for $\mathbf{f} \in \mathbf{L}^2(\Omega)$ is defined as $\displaystyle \vertiii{\mathbf{f}}_{\mathbf{L}^2}:=\sup_{\substack{ \mathbf{v} \in \mathbf{L}^2(\Omega) \\  \mathbf{v}\neq 0 }}
	\frac{(\mathbf{f},\mathbf{v} )_{\Omega}}{\vertiii{\mathbf{v}}_{0}}$.

%\subsection{Organization}
The paper is organized
as follows. Section \ref{Asymptotic analysis of the minimizers} 
 focuses on the asymptotic analysis of global energy minimizers in the $\ell \to 0$ limit. Section \ref{Finite element analysis} discusses the finite element convergence analysis of the discrete solutions of \eqref{continuous nonlinear ferro strong form}. The conforming finite element formulation of the problem is derived and the a priori estimates are proven in Section \ref{A priori estimates}. 
 The article concludes with several numerical results for the discrete solutions, and the convergence history for various  values of $\ell$, $c$ and $h$ in Section \ref{Ferronematic numerical experiments}. 
\section{Asymptotic analysis of the minimizers}\label{Asymptotic analysis of the minimizers} The primary goal of this section is to establish  ${H}^2$-bound for the global energy minimizers independent of $\ell,$ in the $\ell \to 0$ limit, under suitable assumptions on the domain and boundary conditions. The $\ell \to 0$ limit is relevant for macroscopic domains or large domains, that are much larger than characteristic material-dependent and temperature-dependent nematic and magnetic correlation lengths \cite{Ferronematics_2D}. The proof is done in several stages: analysis of the minimizers of the bulk potential reproduced from \cite{JDPEFAMJX}, a strong convergence result for global minimizers followed by convergence results for the bulk potential that largely follow from \cite{Brezis_Bethual_Book}, followed by a delicate Bochner inequality for the energy density that combines ideas from \cite{Brezis_Bethual_Book} and \cite{Canevari2015}. Once we have the Bochner inequality, the $\ell$-independent ${H}^2$-bound for global energy minimizers in the $\ell \to 0$ limit is relatively standard, from estalished techniques in the Ginzburg-Landau theory for superconductivity although additional technical difficulties are encountered due to the four degrees of freedom in the problem.

\medskip

\begin{lem}[Global minimizers of $f_B$] \label{Global minimizer to bulk energy functional}
The bulk potential $f_B$ in \eqref{bulk energy potential} is coercive. 
 For $c>0,$ $f_B$ attains its minimum at $ \Psi^{\min}:=(Q_c\cos 2\varphi,Q_c\sin 2\varphi,M_c\cos\varphi,M_c\sin\varphi),$  where $Q_c$ and $M_c$ satisfies  \begin{subequations}\label{equations for Q_c and M_c}
	\begin{align}
	& Q_c^3-(1+\frac{c^2}{2})Q_c-\frac{c}{2} =0, \text{  and }  M_c^2=1+cQ_c \\&
	\text{  with } \, Q_c:=\bigg(\frac{c}{4}+\sqrt{\frac{c^2}{16}- \frac{1}{27}\bigg(1+\frac{c^2}{2}\bigg)^3} \bigg)^{1/3}+\bigg(\frac{c}{4}-\sqrt{\frac{c^2}{16}- \frac{1}{27}\bigg(1+\frac{c^2}{2}\bigg)^3} \bigg)^{1/3}, \,\, M_c:=\sqrt{1+cQ_c} .
	\end{align}
\end{subequations}
\end{lem}
\begin{proof}
For a fixed $c>0,$	recast the potential  ${f}_B(\Psi) $ (see \eqref{bulk energy potential}) using the parameterization: $\Psi:=(S\cos\theta,S\sin\theta,R\cos\varphi,R\sin\varphi),$ to obtain
%\begin{align*}
$${f}_B(S,R,\theta, \varphi):= \frac{1}{4}(S^2 -1)^2+\frac{1}{4}( R^2 - 1)^2 - \frac{c}{2}SR^2\cos(\theta - 2\varphi) .$$
%\end{align*}
 For any bulk energy minimizer, let $\Psi^{\min}:=(\Qvec^{\min},\Mvec^{\min})$ with $\Qvec^{\min}:=(Q_c\cos\theta_c,Q_c\sin\theta_c) $ and  $\Mvec^{\min}:=(M_c\cos\varphi_c,M_c\sin\varphi_c),$  and the minimality condition
% $\frac{\partial \tilde{f}_B }{\partial \Psi} (\Psi^\min)=0$
$\frac{\partial {f}_B}{\partial \Psi} (\Psi^{\min})=0,$ 
reduces to:
	\begin{align*}
	&\frac{\partial f_B}{\partial S}(\Psi^{\min})=Q_c(Q_c^2-1)-\frac{c}{2}M_c^2\cos(\theta_c - 2\varphi_c ) =0,	\frac{\partial f_B}{\partial R}(\Psi^{\min})=M_c(M_c^2-1)-cQ_cM_c\cos(\theta_c - 2\varphi_c ) =0,%\label{minimizer of f_B equation1} 
	\\&
%\label{minimizer of f_B equation2} \\&
	\frac{\partial f_B}{\partial \theta }(\Psi^{\min})= \frac{c}{2}Q_c M_c^2\sin(\theta_c - 2\varphi_c ) =0 ,\frac{\partial f_B}{\partial \varphi }(\Psi^{\min})= -cQ_c M_c^2\sin(\theta_c - 2\varphi_c ) =0 .%\label{minimizer of f_B equation3}
%	\\&	\label{minimizer of f_B equation4}
	\end{align*}
Since $Q_c, M_c \geq 0,$	the conditions on the second line above require that $\theta_c = 2\varphi_c +k\pi , k\in \mathbb{Z}.$ 	Note that $Q_c>0$ and $M_c>0$ for the bulk energy minimizer since one can easily check that $\displaystyle \min_{\Qvec,\Mvec \in \mathbb{R}^2}f_B(\Qvec,\Mvec)<0$ for $c>0$.	If $\theta_c = 2\varphi_c +(2k+1)\pi , k\in \mathbb{Z}$, then
%	\begin{align*}
$$	f_B(Q_c,M_c,\theta_c, \varphi_c):= \frac{1}{4}(Q_c^2 -1)^2+\frac{1}{4}( M_c^2 - 1)^2 + \frac{c}{2}Q_cM_c^2>0,$$
%	\end{align*} 
	which is not the minimum value of $f_B$  for  $c>0$. 
	\noindent Hence, the bulk energy minimizers have 
	\begin{align}\label{angle constraint for c>0}
	\theta_c = 2\varphi_c +2k\pi , k\in \mathbb{Z},
	\end{align}
 which in turn, requires that $	Q_c(Q_c^2-1)-\frac{c}{2}M_c^2 =0,  M_c^2=1+cQ_c,$ or equivalently,  
	\begin{align}\label{equation1 satisfied by the minimizers}
	Q_c^3-(1+\frac{c^2}{2})Q_c-\frac{c}{2} =0 \text{  and } M_c^2=1+cQ_c.
	\end{align}	
 By Descartes' rule of sign, this  equation has one positive and two negative roots. Therefore, $f_B$ attains its minimum at $\Psi^{\min}:=(Q_c \cos 2\varphi,Q_c \sin 2\varphi,M_c \cos\varphi,M_c \cos\varphi )$, where  $Q_c$ is the positive root of  $Q_c^3-(1+\frac{c^2}{2})Q_c-\frac{c}{2} =0$ and $ M_c=\sqrt{1+cQ_c}.$
	The Hessian matrix, $Hf_B,$ at the point $\Psi^{\min}$ is given by,
	$Hf_B=\begin{pmatrix}
	3Q_c^2- 1&  -cM_c & 0&0\\
	-cM_c& 2M_c^2 &0&0\\
	0 &0& \frac{c}{2}Q_c M_c^2 & -cQ_c M_c^2\\
	0 &0& -cQ_c M_c^2 & 2cQ_c M_c^2
	\end{pmatrix} = \begin{pmatrix}
	A & O\\
	O &  B
	\end{pmatrix},$
	where $A= \begin{pmatrix}
	3Q_c^2- 1&  -cM_c \\
	-cM_c& 2M_c^2
	\end{pmatrix}$, $B= c^2Q_c^2 M_c^4 \begin{pmatrix}
	\frac{1}{2} & -1\\
	-1 & 2
	\end{pmatrix},$ and $O$ is a $2 \times 2$ zero matrix. 
	
	\medskip
	
	\noindent Since the matrix $Hf_B$ is symmetric, all the eigen values of $Hf_B$ are real. The eigenvalues of $B$ are $0$ and $\frac{5}{2}c^2Q_c^2 M_c^4>0,$ which implies that the matrix $B $ is non-negative definite. Using \eqref{equation1 satisfied by the minimizers}, one can check that the determinant of $A,$ $\det A =M_c^2(6Q_c^2-2-c^2)=M_c^2(4+2c^2+\frac{3c}{Q_c})>0.$ Therefore, both the eigenvalues of $A$ are either negative or positive, and zero is not an eigenvalue. Moreover, the trace of $A,$ Tr${(A)}=3Q_c^2-1+2M_c^2=3Q_c^2+1 +2cM_c>0, $ so that all the eigenvalues of $A$ are positive. Therefore, $A$ is positive definite and the Hessian matrix of $f_B$ at $\Psi^{\min}$ is non-negative definite. This concludes the proof that $\Psi^{\min}$ is a global minimizer of $f_B$. For a detailed computation of $Q_c$, the positive solution of \eqref{equation1 satisfied by the minimizers}, we refer the reader to to \cite[Section 2.1]{JDPEFAMJX}.
\end{proof}
\begin{rem}
 For $c<0$,  ${f}_B$ attains its minimum at  $\Psi^{\min}:=(-Q_a\cos2\varphi,-Q_a\sin2\varphi, M_a\cos\varphi,M_a\sin\varphi)$ with $a:=-c,$ and 
	\begin{align*}
	&Q_a^3-(1+\frac{a^2}{2})Q_a-\frac{a}{2} =0, M_a^2=1+aQ_a\\&
	\text{  with }\,  Q_a:=\bigg(\frac{a}{4}+\sqrt{\frac{a^2}{16}- \frac{1}{27}\bigg(1+\frac{a^2}{2}\bigg)^3} \bigg)^{1/3}+\bigg(\frac{a}{4}-\sqrt{\frac{a^2}{16}- \frac{1}{27}\bigg(1+\frac{a^2}{2}\bigg)^3} \bigg)^{1/3}, \,\, M_a:=\sqrt{1+aQ_a} .
	\end{align*}
	For the minimizers of ${f}_B$ for $c=0$, see \cite[Remark 2.3]{JDPEFAMJX}. \qed 
\end{rem}
 Let $$ \mathcal{A}_{\min}:=\{\Psi^{\min} \in \mathbf{H}^1(\Omega) |\, \Psi^{\min}:=(Q_c \cos2\varphi, Q_c \sin2\varphi, M_c \cos\varphi, M_c \sin\varphi)\text{ with } Q_c,M_c \text{ satisfying \eqref{equations for Q_c and M_c}}\}$$  and define a non-negative bulk energy $\tilde{f}_B, $ to be 
\begin{align}\label{modified bulk energy functional}
\tilde{f}_B(\Qvec,\Mvec):= f_B(\Qvec,\Mvec) - \min_{\Qvec , \Mvec \in \mathbb{R}^2} f_B(\Qvec,\Mvec),
\end{align}
so that $\tilde{f}_B(\Qvec,\Mvec) \geq 0$  and  $\tilde{f}_B(\Qvec,\Mvec)=0 $ if and only if $(\Qvec,\Mvec) \in \mathcal{A}_{\min} $ for $c>0.$
In what follows, for $\Qvec=\left(Q_{11}, Q_{12} \right), \Mvec = \left(M_1, M_2 \right)$, we study global minimizers of the modified energy functional, 
\begin{align}\label{modified energy functional}
\tilde{\mathcal{E}}(\Qvec,\Mvec):=& \int_\Omega \frac{1}{2}(\abs{\nabla Q_{11}}^2+\abs{\nabla Q_{12}}^2+\abs{\nabla M_{1}}^2+\abs{\nabla M_{2}}^2) \dx+ \frac{1}{\ell}\int_\Omega \tilde{f}_B(\Qvec,\Mvec)\dx
\end{align}
in the admissible space $\mathcal{A}, $ with the added restrictions that $\Omega \subset \mathbb{R}^2$ is a smooth bounded domain with smooth boundary and  $\mathbf{g} \in \mathcal{A}_{\min}$ with $\textrm{deg}\left(\mathbf{g} \right) = 0$. More precisely, 
\begin{align}\label{boundary condition}
\mathbf{g}:=(\Qvec^b , \Mvec^b) \in \mathcal{A}_{\min}
\end{align}
 %$\mathbf{g}:=(\Qvec^b , \Mvec^b) \in \mathcal{A}_{\min}, $ 
 is smooth such that 
$\Qvec^b:= (Q_c \cos2\varphi_b, Q_c \sin2\varphi_b),\Mvec^b:=(M_c \cos\varphi_b, M_c \sin\varphi_b)  , \varphi_b\in C^{\infty } (\partial \Omega; \mathbb{R}),$ and $\Qvec^b, \Mvec^b$ have zero winding number around $\partial \Omega$.
%\subsection{Limiting function and uniform convergence in the interior}
Next, the limiting profiles for the global minimizers, $\Psi^{\ell}$ as $\ell \to 0 $, are analysed.

\medskip

\noindent 	Let $\Psi^{\ell}$ be a minimizer for $\tilde{\mathcal{E}}$ from \eqref{modified energy functional}. Then $\Psi^{\ell}$ is a weak solution of the Euler-Lagrange equation	
	\begin{align}\label{minimization pde}
	\Delta \Psi^{\ell} =\ell^{-1} D\tilde{f}_B(\Psi^{\ell} ) \text{ in } \Omega, \text{  and } \Psi^{\ell}=\mathbf{g} \text{  on } \partial\Omega, 
	\end{align}
	where 	$D\tilde{f}_B$ is the gradient of $\tilde{f}_B$ with respect to the variable $\Psi :=(Q_{11} ,Q_{12},M_1, M_2)$ and %for any $\Psi :=(Q_{11} ,Q_{12},M_1, M_2), $
	\begin{align}\label{definition of Df}
	    D\tilde{f}_B(\Psi):= \begin{pmatrix}
	(Q_{11}^2+Q_{12}^2 -1)Q_{11} - \frac{c}{2}(M_1^2-M_2^2)\\(Q_{11}^2+Q_{12}^2 -1)Q_{12} - cM_1M_2\\(M_{1}^2+M_{2}^2 -1)M_{1} - c(Q_{11}M_1+Q_{12}M_2)\\ (M_{1}^2+M_{2}^2 -1)M_{2} - c(Q_{12}M_1-Q_{11}M_2)
	\end{pmatrix}.
	\end{align}
\begin{prop}[$H^1$ convergence to harmonic maps] \label{Proposition: limiting functions}
	Let $\Omega \subset \mathbb{R}^2$ be a simply-connected bounded open set with smooth boundary.	Let $\Psi^{\ell}:=(\Qvec^{\ell}, \Mvec^{\ell})$ be a global minimizer of $\tilde{\mathcal{E}}$ from  \eqref{modified energy functional} in the admissible space $\mathcal{A},$  with $\mathbf{g} \in \mathcal{A}_{\min}$ defined in \eqref{boundary condition}. Then 
	the sequence $(\Qvec^{\ell}, \Mvec^{\ell}) \rightarrow (\Qvec^0, \Mvec^0)$  converges strongly in  $\mathbf{H}^1(\Omega)$ upto a subsequence as $\ell \rightarrow 0,$ where $\Qvec^0:= Q_c e^{2i\varphi_0}$ and $\Mvec^0:= M_c e^{i\varphi_0}$ and $\varphi_0$ is a solution of 
	\begin{align}
		\Delta &\varphi_0 = 0 \text{ on } \Omega,\,\, \text{  and }\,\, 
	\varphi_0 = \varphi_b \text{ on } \partial \Omega. 
   \label{limiting function angle equation}
	\end{align}
\end{prop}
\begin{rem}
The proof of Proposition \ref{Proposition: limiting functions} closely follows the methodology used in \cite[Proposition 1]{Bethuel} and \cite{Majumdar2010}. The primary difference is that we have two harmonic limits; $\Qvec^0$ corresponding to the nematic order parameter  and, $\Mvec^0$ corresponding to the magnetization vector.     The proof is given for the sake of completeness.	\qed
\end{rem}
\begin{proof}[Proof of Proposition \ref{Proposition: limiting functions}]
	Observe that $\Psi_0:=(\Qvec^0, \Mvec^0) \in \mathcal{A}_{\min} \cap \mathcal {A}$ and $ \tilde{f}_B(\Psi_0)=0$. Since $\Psi^{\ell}:=(\Qvec^{\ell}, \Mvec^{\ell})$ is a minimizer of $\tilde{\mathcal{E}}$ defined in \eqref{modified energy functional} for a fixed $\ell$ and $\tilde{f}_B(\Psi^{\ell}) \geq 0,$  
	\begin{align}\label{Strong convergence in H1 equation1}
	\int_{\Omega}\frac{1}{2}\abs{\nabla \Psi^{\ell}}^2 \dx \leq 	\int_{\Omega}\frac{1}{2}\abs{\nabla \Psi^{\ell}}^2 \dx+\frac{1}{\ell}\int_{\Omega} \tilde{f}_B(\Psi^{\ell})\dx \leq \int_{\Omega}\frac{1}{2}\abs{\nabla \Psi_0}^2\dx.
	\end{align}
	%where $ \tilde{f}_B(\Psi_0)=0$ for $\Psi_0 \in \mathcal{A}_{\min}$ is used on the right hand side of \eqref{Strong convergence in H1 equation1}. 
%	Since $\tilde{f}_B(\Psi^{\ell}) \geq 0,$ we obtain for all $\ell,$ $\int_{\Omega}\frac{1}{2}\abs{\nabla \Psi^{\ell}}^2  \dx \leq \int_{\Omega}\frac{1}{2}\abs{\nabla \Psi_0}^2 \dx\implies$ the sequence $(\Psi^{\ell})$ is bounded in $\mathbf{H}^1(\Omega)$ uniformly and the bound is independent of $\ell.$ 
	Then there exists a subsequence  $(\Psi^{\ell}) \rightharpoonup \Psi_1$ weakly in   $\mathbf{H}^1(\Omega)$ as $\ell \rightarrow 0.$ Therefore, Majur's theorem \cite[Page 723]{Evance19} yields that trace of $\Psi_1$ is $(\Qvec^b, \Mvec^b)$.
	 %Since the sequence $(\Psi^{\ell})$ is bounded in $\mathbf{H}^1(\Omega)$ uniformly in $\ell,$ and $\mathbf{H}^1(\Omega)\hookrightarrow \mathbf{L}^2(\Omega) $,
Now, Rellich–Kondrachov compactness theorem \cite[Page 286]{Evance19} gives the existence of a subsequence $(\Psi^{\ell})$ that converges strongly to $\Psi_1$ in $\mathbf{L}^2$.  The lower semi-continuity of $\mathbf{H}^1$  norm with respect to the weak convergence yields
	\begin{align}\label{Strong convergence in H1 equation2}
	\int_{\Omega}\abs{\nabla \Psi_1}^2 \dx \leq \int_{\Omega}\abs{\nabla \Psi_0}^2 \dx.
	\end{align}
	Moreover, since $\int_{\Omega} \tilde{f}_B(\Psi^{\ell}) \dx \leq \ell \int_{\Omega}\frac{1}{2}\abs{\nabla \Psi_0}^2 \dx $ from \eqref{Strong convergence in H1 equation1}, $ \int_{\Omega} \tilde{f}_B(\Psi^{\ell}) \dx \rightarrow 0 $   as $\ell \rightarrow 0.$  Since $\tilde{f}_B \geq0,$ on  a subsequence, $\tilde{f}_B(\Psi^{\ell}) \rightarrow 0$ for almost all $x \in \Omega.$   Therefore, $\Psi_1$ is of the form 
	$$\Psi_1:=(\Qvec_1, \Mvec_1)=(Q_c \cos2\varphi_1 ,Q_c \sin2\varphi_1 ,M_c \cos\varphi_1,M_c \sin\varphi_1 ) \text{ a.e. } x \in \Omega \text{ and }  \varphi_1=\varphi_b \text{ on } \partial\Omega.$$	
Let $\mathcal{A}_\varphi:= \{\varphi \in H^1(\Omega)| \, \varphi =\varphi_b \text{ on } \partial \Omega\}$. By definition, $\displaystyle \int_{\Omega} \abs{\nabla {\varphi }_0}^2 \dx =\min_{\varphi \in \mathcal{A}_\varphi}\int_{\Omega} \abs{\nabla {\varphi}}^2 \dx$ so that
	\begin{align}\label{Strong convergence in H1 equation3}
	\int_{\Omega}\abs{\nabla \Psi_0}^2 \dx = \int_{\Omega}(4Q_c^2+M_c^2) \abs{\nabla {\varphi }_0}^2 \dx \leq \int_{\Omega}(4Q_c^2+M_c^2) \abs{\nabla {\varphi }}^2 \dx = \int_{\Omega}\abs{\nabla \Psi_1}^2 \dx.
	\end{align}
	The inequalities \eqref{Strong convergence in H1 equation2} and \eqref{Strong convergence in H1 equation3} imply that $\int_{\Omega}\abs{\nabla \Psi_1}^2 \dx = \int_{\Omega}\abs{\nabla \Psi_0}^2 \dx.$ This together with the lower semi-continuity of $\mathbf{H}^1$  norm and \eqref{Strong convergence in H1 equation1} lead to the following sequence of inequalities:
	\begin{align*}
	\int_{\Omega}\abs{\nabla \Psi_0}^2  \dx\leq\liminf_{\ell \rightarrow 0}  \int_{\Omega}\abs{\nabla \Psi^{\ell}}^2 \dx \leq \limsup_{\ell \rightarrow 0}  \int_{\Omega}\abs{\nabla \Psi^{\ell}}^2 \dx \leq  \int_{\Omega}\abs{\nabla \Psi_0}^2 \dx,
	\end{align*}
	yielding the convergence, $\vertiii{\nabla\Psi^\ell}_0 \rightarrow \vertiii{\nabla\Psi_0}_0$ as $\ell \rightarrow 0.$ This norm convergence together with the weak convergence $(\Psi^{\ell}) \rightharpoonup \Psi_0$ in $\mathbf{H}^1(\Omega)$, establishes the strong convergence $\Psi^{\ell} \rightarrow \Psi_0$ in $\mathbf{H}^1(\Omega).$
\end{proof}
%{\color{blue}\begin{rem}
%	We restrict the analysis for the energy minimizers of $\tilde{\mathcal{E}}$ satisfying the energy bound in  \eqref{Strong convergence in H1 equation1}. Note that the inequality \eqref{Strong convergence in H1 equation1} holds for physically relevant critical points, which include local energy minimizers, at least the ones that might be observed numerically and experimentally.
%\end{rem}}
A $L^\infty $ bound for $\Qvec^{\ell}$ and $\Mvec^{\ell}$ follow from maximum principle arguments for the system \eqref{continuous nonlinear ferro strong form}, as has been done in \cite{JDPEFAMJX}.   
\begin{prop}[$L^\infty $ bound]\cite[Theorem 2.5]{JDPEFAMJX} \label{Proposition: L infinity bound of solutions}
	Let $\Psi^{\ell}:=(\Qvec^{\ell}, \Mvec^{\ell}) \in \mathcal{A}$ be a solution of  \eqref{minimization pde}, with $\mathbf{g} \in \mathcal{A}_{\min}$ defined in \eqref{boundary condition}. Then $\abs{\Qvec^{\ell}} \leq Q_c$ and $\abs{\Mvec^{\ell}} \leq M_c$. 	
\end{prop}
\begin{lem}\cite[Lemma A.1]{Bethuel}\label{regularity 1}
	Assume $u$ is a scalar-valued function such that $-\Delta u = f$ on $\Omega \subset \mathbb{R}^n$. Then $\abs{\nabla u(x)}^2 \leq C(\norm{f}_{L^\infty(\Omega)}\norm{u}_{L^\infty(\Omega)} + \frac{1}{{\text{dist}}^2(x, \partial \Omega )} \norm{u}_{L^\infty(\Omega)}^2)$ for all $x\in \Omega$, where the constant $C$ depends only on the dimension $n.$
\end{lem}
\begin{prop}[Uniform convergence in the interior]\label{Proposition: uniform convergence on compact sets}
	Let $\Omega \subset \mathbb{R}^2$ be a simply-connected bounded open set with smooth boundary.
	Let $\Psi^{\ell}:=(\Qvec^{\ell}, \Mvec^{\ell})$ be a global minimizer of $\tilde{\mathcal{E}}$ from  \eqref{modified energy functional} in the admissible space $\mathcal{A}$,  with $\mathbf{g} \in \mathcal{A}_{\min}$ defined in \eqref{boundary condition}. Then $\abs{\,\abs{\Qvec^{\ell}}- Q_c} $ = {\scalebox{1.2}{$\scriptscriptstyle\mathcal{O}$}}$(1)$, $\abs{\,\abs{\Mvec^{\ell}}- M_c} $ = {\scalebox{1.2}{$\scriptscriptstyle\mathcal{O}$}}$(1)$, and $\abs{\,\cos(\theta_{\ell} -2\varphi_{\ell})- 1} $ = {\scalebox{1.2}{$\scriptscriptstyle\mathcal{O}$}}$(1)$,  as $\ell \rightarrow 0$  on every compact subset $K \subset \Omega.$ 	
\end{prop}
\begin{proof}
Lemma \ref{regularity 1} and Proposition \ref{Proposition: L infinity bound of solutions} give the following upper bound for the gradient:
	\begin{align} \label{gradient bound depending on l}
	\abs{\nabla \Psi^{\ell}}\leq \frac{C}{\sqrt{\ell}} \text{ on every compact subset } K \text{of }  \Omega,
	\end{align}
	for a positive constant $C$ independent of $\ell$.
	Let $K$ be a compact set in $\Omega$. Let $x_0 \in K.$ Set $\alpha:= \Qvec^{\ell}(x_0)$ and $\beta:=\Mvec^{\ell}(x_0).$ Using \eqref{gradient bound depending on l} shows that
%	\begin{align*}
$$	\abs{\Qvec^{\ell}(x) - \Qvec^{\ell}(x_0)} \leq \frac{C \rho}{\sqrt{\ell}}, \text{ and }	\abs{\Mvec^{\ell}(x) - \Mvec^{\ell}(x_0)} \leq \frac{C \rho}{\sqrt{\ell}} \text{ for } \abs{x - x_0}<\rho <\delta:= \text{ dist}(K,\partial \Omega).$$
%	\end{align*}
 Proposition \ref{Proposition: L infinity bound of solutions}, \eqref{gradient bound depending on l} and the inequalities above are enough to show that the bulk energy density, $\tilde{f}_B(\cdot)$ is locally Lipschitz and
%	\begin{align*}
$$	\abs{\tilde{f}_B (\Qvec^{\ell}(x),\Mvec^{\ell}(x))-\tilde{f}_B (\alpha,\beta) } \leq  \frac{C\rho }{\sqrt{\ell}}  \text{ for } \abs{x - x_0}<\rho <\delta:= \text{ dist}(K,\partial \Omega).$$
%	\end{align*}
That is, $-\frac{C\rho }{\sqrt{\ell}}+	\tilde{f}_B (\alpha,\beta)  \leq \tilde{f}_B (\Qvec^{\ell}(x),\Mvec^{\ell}(x)) .$ The inequality \eqref{Strong convergence in H1 equation1} and the strong convergence, $(\Psi^{\ell}) \rightarrow \Psi_0$ in $\mathbf{H}^1(\Omega)$,  imply that $\displaystyle \lim_{\ell \rightarrow 0} \frac{1}{\ell}\int_{\Omega}\tilde{f}_B (\Psi^{\ell})\dx \rightarrow 0.$ Therefore,
%	This plus  $\int_{\Omega} \tilde{f}_B(\Psi^{\ell}) \dx \leq \ell \int_{\Omega}\frac{1}{2}\abs{\nabla \Psi_0}^2 \dx,$  from \eqref{Strong convergence in H1 equation1} leads to 
%	\begin{align*}
$$	\pi \rho^2\bigg(	\tilde{f}_B (\alpha,\beta)-\frac{C\rho }{\sqrt{\ell}} \bigg) \leq \int_{B(x_0,\rho)}\tilde{f}_B (\Qvec^{\ell}(x),\Mvec^{\ell}(x))  \dx=\ell \, {\scalebox{1.2}{$\scriptscriptstyle\mathcal{O}$}}(1) \text{ as } \ell \rightarrow 0.$$
%	\end{align*}
	For a specific choice of $\rho= \frac{\sqrt{\ell} \tilde{f}_B (\alpha,\beta) }{2C},$ we obtain  
%	\begin{align*}%\label{uniform convergence on compact sets equation}
$$	 \frac{\pi \ell  \tilde{f}_B^3 (\alpha,\beta) }{8C^2}=\ell \, {\scalebox{1.2}{$\scriptscriptstyle\mathcal{O}$}}(1) \text{ as } \ell \rightarrow 0 \text{ if and only if }  \tilde{f}_B^3 (\alpha,\beta) ={\scalebox{1.2}{$\scriptscriptstyle\mathcal{O}$}}(1)  \text{ as } \ell \rightarrow 0 .$$
%	\end{align*}
	Therefore, $\tilde{f}_B \rightarrow 0$ uniformly on compact subsets of $  \Omega \text{ if and only if }  \abs{\Qvec^{\ell}}\rightarrow  Q_c$ and  $\abs{\Mvec^{\ell}}\rightarrow  M_c$, $\cos(\theta_{\ell} -2\varphi_{\ell}) \rightarrow 1$, as $\ell \rightarrow 0$ uniformly on compact subsets of $  \Omega $. 
	%Also, $\theta_{\ell} -2\varphi_{\ell} \rightarrow $ an element in $2k\pi, k\in \mathbb{Z}  $ as $\ell \rightarrow 0,$ that is $\cos(\theta_{\ell} -2\varphi_{\ell}) \rightarrow 1$, as $\ell \rightarrow 0.$
\end{proof}
%\subsection{$H^2$ bound of minimizers $\Psi^{\ell}$ of \eqref{modified energy functional} in the interior for the parameter $\mathbf{c>0}$}\label{H2 bound of minimizers}
The next result, adapted from \cite{Canevari2015}, is a crucial ingredient for the Bochner-type inequality for the ferronematic energy density.%An inequality crucial for Bochner-type inequality is established next and the proof uses the fact that for every point $\varv $ on the  vacuum manifold referred to $\mathcal{N}:=\tilde{f}_B^{-1}(0),$ the Hessian matrix $D^2\tilde{f_B}$ restricted to the normal space of $ \mathcal{N} $ at  $\varv $ is positive definite.
	
	\begin{thm} \label{Hessian restricted to normal space is positive definite}
		%	Let $\Psi^{\ell}$ be a minimizer for $\tilde{\mathcal{E}}$ from \eqref{modified energy functional}. Then $\Psi^{\ell}$ is a weak solution of the Euler-Lagrange equation $\Delta \Psi^{\ell} =\ell^{-1} D\tilde{f}_B(\Psi^{\ell} ) \text{ in } \Omega,$
%		\begin{align*}
%		\Delta \Psi^{\ell} +\frac{1}{\ell}D\tilde{f}_B(\Psi^{\ell} )=0 \text{ in } \Omega,
%		\end{align*}
	%	where 
		Let $\tilde{f}_B:\mathbb{R}^4 \rightarrow \mathbb{R}$ be the smooth function defined in \eqref{modified bulk energy functional} and $\mathcal{N}:=\tilde{f}_B^{-1}(0).$ Then it holds:
		
		\medskip
		
		$(i)$ The  set $\mathcal{N}$ is non-empty,  smooth, compact and connected submanifold of $ \mathbb{R}^4  $ without boundary.
		
		\medskip
		
		$(ii)$  There exists some positive constants $\delta_0<1 , $ $m_0$ such that, for all $\varv \in \mathcal{N}$ and all unit normal vector \indent \indent $\nu \in \mathbb{R}^4$ to $\mathcal{N}$ at the point $\varv,$
		\begin{align}\label{Hypothesis H2}
		D\tilde{f}_B(\varv+t\nu)\cdot \nu \geq m_0 t, \text{ if } 0\leq t \leq \delta_0.
		\end{align}
	\end{thm}

\begin{proof}%[\textbf{Proof of (H2)}]
	The proof is divided into three steps.  The fact that $\displaystyle \min_{\Qvec,\Mvec \in \mathbb{R}^2} \tilde{f}_B(\Qvec,\Mvec)=0$ and the existence of global minimizers from Lemma \ref{Global minimizer to bulk energy functional}  implies that $\mathcal{N}$ is non-empty. A diffeomorphism $h$ from  $ \mathbb{S}^1,$ the unit circle in $\mathbb{R}^2,$ to the vacuum manifold  $\mathcal{N}$ is defined in Step 1.   Step 2 focuses on the construction of the tangent and normal spaces of $ \mathcal{N} $ at a point $\Psi^{\min} \in \mathcal{N} $. %The Hessian matrix of $\tilde{f}_B$ restricted to the normal space of $ \mathcal{N} $, key part of the Taylor series expansion of the term on the left hand side of \eqref{Hypothesis H2},  is derived in 
	The third step focuses on the derivation of the inequality \eqref{Hypothesis H2}.% is established.
		
\medskip
	
\noindent {\it Step 1 (Diffeomorphism $h:\mathbb{S}^1 \rightarrow \mathcal{N}$)}.	
%		
%	\medskip
%	
%	\noindent \textbf{Diffeomorphism ($h:\mathbb{S}^1 \rightarrow \mathcal{N}$):}
	 From Lemma \ref{Global minimizer to bulk energy functional}, for $c>0$, the potential  $\tilde{f}_B$ attains minimum at $\Psi^{\min}:=(Q_c\cos 2\varphi, Q_c\sin 2\varphi, $ $ M_c\cos \varphi, M_c\sin \varphi).$	Therefore, $\mathcal{N}:=\tilde{f}_B^{-1}(0)$ is computed to be
	%\begin{align*}
$$	\mathcal{N}:= \biggl\{\begin{pmatrix}
	Q_c(2\mathbf{n} \otimes \mathbf{n} -I)\e_1 \\ M_c \mathbf{n}
	\end{pmatrix} \text{  such that } \mathbf{n}:=\begin{pmatrix}
	\cos \varphi \\ \sin \varphi
	\end{pmatrix}  \in \mathbb{S}^1, M_c=\sqrt{1+cQ_c} ,  \e_1:=\begin{pmatrix}
	1 \\ 0
	\end{pmatrix}, I:=\begin{pmatrix}
	1 &0 \\ 0 &1
	\end{pmatrix} \biggr\}. $$
%	\end{align*}
	The map $h:\mathbb{S}^1 \rightarrow \mathcal{N}$ defined by
%	\begin{align*}
$$	h(\mathbf{n} ):=\begin{pmatrix}
	Q_c(2\mathbf{n} \otimes \mathbf{n} -I)\e_1 \\ M_c \mathbf{n}
	\end{pmatrix} =\begin{pmatrix}
	Q_c\cos 2\varphi \\ Q_c \sin 2\varphi \\	M_c\cos \varphi \\ M_c \sin \varphi
	\end{pmatrix}$$
%	\end{align*}
	is a diffeomorphism. Since $\mathbb{S}^1$ is compact and connected subset of $\mathbb{R}^2,$ the properties that $\mathcal{N}$ is compact and connected will follow from the properties of  $h.$  
	
	\medskip
	
%	Let $ \mathbb{S}^1$ denote the unit circle in $\mathbb{R}^2.$ 
%\noindent \textbf{Tangent plane to $\mathcal{N}$ at a point $\Psi^{\min}\in \mathcal{N}$:}
%	Let $\Psi^{\min} \in \mathcal{N}.$ 
\noindent	{\it Step 2 (Tangent and normal spaces of $\mathcal{N}$ at $\Psi^{\min}\in \mathcal{N}$)}. 	We compute the basis vectors of the tangent plane of $\mathcal{N}$ at $\Psi^{\min} \in \mathcal{N}$. The conventional notations for  tangent spaces are used, e.g., $T_\mathbf{n}\mathbb{S}^1$ denote the tangent space of $\mathbb{S}^1$ at $\mathbf{n} \in \mathbb{S}^1.$ The differential of $h$ at $\mathbf{n}$ is a linear map $dh(\mathbf{n}): T_\mathbf{n}\mathbb{S}^1 \rightarrow T_{h(\mathbf{n})}\mathcal{N}$, where $T_{h(\mathbf{n})}\mathcal{N}$ denote the tangent space of $\mathcal{N}$ at $h(\mathbf{n}) \in \mathcal{N}.$ For all tangent vectors $\mathbf{v} \in T_\mathbf{n}  \mathbb{S}^1,$ $dh(\mathbf{n})$ is defined as
	%\begin{align*}
$$	\dual{dh(\mathbf{n}), \mathbf{v}}= \begin{pmatrix}
	\dual{d(Q_c(2\mathbf{n} \otimes \mathbf{n} -I)\e_1), \mathbf{v}} \\ \dual{d(M_c \mathbf{n}), \mathbf{v}}
	\end{pmatrix} =  \begin{pmatrix}
	Q_c(2\mathbf{n} \otimes \mathbf{v}+2\mathbf{v} \otimes \mathbf{n} )\e_1 \\ M_c \mathbf{v}
	\end{pmatrix}. $$
%	\end{align*}
	For $\Psi^{\min} \in \mathcal{N},$ there exists $\mathbf{n} \in \mathbb{S}^1 $ such that $h(\mathbf{n})=\Psi^{\min}.$ Upto rotating the coordinate frame, we can assume without loss of generality that $\mathbf{n}=\e_2=(0,1)
%	\begin{pmatrix}
%	0\\ 1
%	\end{pmatrix}
	.
	 $ This implies $\varphi =\frac{\pi}{2}$ and consequently $  \Psi^{\min}=(-Q_c,0,0,M_c ).$
%	 $ \begin{pmatrix}
%	-Q_c \\0\\0\\ M_c 
%	\end{pmatrix}. $ 
	The basis vector of the tangent plane of $\mathcal{N}$ at $\Psi^{\min}$ is given by 
	\begin{align*}
	X:=\begin{pmatrix}
	Q_c(2\e_2 \otimes \mathbf{e_1}+2\e_1 \otimes \e_2 )\e_1 \\ M_c \e_1
	\end{pmatrix}= \begin{pmatrix}
	0 \\2Q_c\\ M_c \\0
	\end{pmatrix}. 
	\end{align*}
%	\textbf{Hessian matrix restricted to the normal space of $\mathcal{N}$:}
	Let $P \in \mathbb{R}^4$ be a normal vector of $\mathcal{N}$ at $\Psi^{\min}$. Then $P$ satisfies
	\begin{align}\label{normal}
	P:X=0 \implies \begin{pmatrix}
	p_1 \\p_2\\ p_3\\ p_4
	\end{pmatrix} : \begin{pmatrix}
	0 \\2Q_c\\ M_c \\0
	\end{pmatrix}=0 \implies 2Q_cp_2+M_c p_3=0.
	\end{align}
	
\noindent {\it Step 3 (Proof of \eqref{Hypothesis H2})}.	For any $\Psi := (	Q_{11}, Q_{12}, M_1, M_2 ),$ the Hessian matrix of $\tilde{f}_B$ at $\Psi$ is given by 
	$$D^2\tilde{f}_B(\Psi):= \begin{pmatrix}
3Q_{11}^2+Q_{12}^2-1 & 2Q_{11}Q_{12} & -cM_1 &cM_2\\
2Q_{11}Q_{12} & Q_{11}^2+3Q_{12}^2-1 & -cM_2 &-cM_1\\
-cM_1 &-cM_2& 3M_1^2+M_2^2-1-cQ_{11}& 2M_1M_2-cQ_{12}\\ cM_2 &-cM_1 &2M_1M_2-cQ_{12} &M_1^2+3M_2^2-1+cQ_{11}
	\end{pmatrix}.$$  	
	Here we have $\Psi^{\min} :=(	Q_{11}, Q_{12}, M_1, M_2 )$
	$ 	:= 	(-Q_c ,0, 0, M_c)$
	with $Q_{11}:=-Q_c,$  $Q_{12}:=0,$  $M_{1}:=0,$  $M_2:=M_c.$  A Taylor series expansion of $\tilde{f}_B$ at $\Psi^{\min}$ yields 
%	\begin{align*}
$$	D\tilde{f}_B(\Psi^{\min}+tP):P
	= 	D\tilde{f}_B(\Psi^{\min}):P+tD^2\tilde{f}_B(\Psi^{\min})P:P +r_{\tilde{f}_{B}}:P,$$
%	\end{align*}
	where $r_{\tilde{f}_{B}}$ is the remainder in the Taylor series expansion around $\Psi^{\min}.$  
	Observe that  $	D\tilde{f}_B(\Psi^{\min})=0$ as $\Psi^{\min}\in \mathcal{N}.$   By the definition of Taylor series expansion, there exists $\delta_0>0$ such that on  $B_{\delta_0}(\Psi^{\min})$, 
			\begin{align*}
		D\tilde{f}_B(\Psi^{\min}+tP):P&
		\geq \frac{t}{2}D^2\tilde{f}_B(\Psi^{\min})P:P\\&
		=t\{ ((3Q_{11}^2-1)p_1^2 +2cM_2p_1p_4 +(3M_2^2-1+cQ_{11})p_4^2) +( (Q_{11}^2-1) p_2^2-2cM_2p_2p_3\\&\quad+(M_2^2-1-cQ_{11})p_3^2)\}=:t(T_1+T_2).
			\end{align*}
	 Next consider the term $T_1$ and use $Q_{11}:=-Q_c$, $M_2:=M_c$, and $M_c^2 = 1+cQ_c$, $Q_c^3-(1+\frac{c^2}{2})Q_c-\frac{c}{2} =0 $ from Lemma \ref{Global minimizer to bulk energy functional} for calculations. For $\bar{P}:=(p_1 \,\, p_4),$
%	\begin{align*}
$$	T_1:=(3Q_{11}^2-1)p_1^2 +2cM_2p_1p_4 +(3M_2^2-1+cQ_{11})p_4^2=(3Q_c^2-1)p_1^2 +2cM_cp_1p_4 +2M_c^2p_4^2=\bar{P}A\bar{P}^{\intercal}
%\begin{pmatrix}p_1\\p_4
%	\end{pmatrix} 
	,$$
%	\end{align*}
	where $A:=\begin{pmatrix}3Q_c^2-1 & cM_c\\ cM_c &2M_c^2
	\end{pmatrix}. $ The determinant of $A,$ $\det A =M_c^2(6Q_c^2-2-c^2)=M_c^2(4+2c^2+\frac{3c}{Q_c})>0.$ Therefore, both the eigenvalues of $A$ are either negative or positive, and zero is not an eigenvalue. Moreover, the fact that trace of $A,$ Tr${(A)}=3Q_c^2-1+2M_c^2=3Q_c^2+1 +2cM_c>0, $ yields that all the eigenvalues of $A$ are positive. Therefore, $A$ is positive definite and there exists $\alpha_1>0$ such that $ %(p_1 \,\, p_4)A\begin{pmatrix}p_1\\p_4	\end{pmatrix}
	 \bar{P}A\bar{P}^{\intercal} \geq \alpha_1 (p_1^2+p_4^2).$ Now consider the second term 
%	\begin{align*}
$$	T_2:= (Q_{11}^2-1) p_2^2-2cM_2p_2p_3+(M_2^2-1-cQ_{11})p_3^2=(Q_c^2-1) p_2^2-2cM_cp_2p_3+2cQ_cp_3^2.$$
%	\end{align*}
	The estimates $-2Q_cp_2=M_cp_3$ from  \eqref{normal} and $Q_c^3-(1+\frac{c^2}{2})Q_c-\frac{c}{2} =0$ are used here to obtain
%	\begin{align*}
$$	T_2=(Q_c^2-1+ 4cQ_c)p_2^2+2cQ_cp_3^2=(\frac{c^2}{2}+\frac{c}{2Q_c}+4cQ_c)p_2^2+2cQ_cp_3^2.$$
%	\end{align*}
	Choose $\alpha_2:=\min (\alpha_1, (\frac{c^2}{2}+\frac{c}{2Q_c}+4cQ_c),2cQ_cp_3^2) >0$. Let the normal vector $P$  at $\Psi^{\min}$ is the unit normal vector i.e., $\abs{P}=1$. This plus  $0<\abs{t}<\delta_0,$ and $m_0=\frac{\alpha_2}{2}$ leads to
%	\begin{align*}
$$	D\tilde{f}_B(\Psi^{\min}+tP):P\geq  \frac{t\alpha_2}{2} \abs{P}^2 =m_0 t .$$
%	\end{align*}
	This completes the proof.	
\end{proof}
\begin{rem}
	Theorem \ref{Hessian restricted to normal space is positive definite} verifies the assumptions  H1-H2 of \cite{Canevari2015}, with regards to the ferronematic bulk potential. The advantage of the analysis in \cite{Canevari2015} is that it does not exploit the	matricial structure of the configuration space, nor the precise shape of the potential and it's zero set. Once these assumptions are verified, one can use the asymptotic analysis in \cite{Canevari2015} to recover a Bochner-type inequality for the ferronematic free energy density. \qed
	%The advantage of the analysis in \cite{Canevari2015} is that it does not exploit the	matricial structure of the configuration space, nor the precise shape of the potential and it's zero set. %The convergence of the minimizers of the three dimensional Landau-de Gennnes energy functional in the vanishing elastic constant limit  is recovered then as a specific case.

\end{rem}
\begin{defn}[Nearest point projection onto $\mathcal{N}$]\cite{RogerMoser2005} 
	For  a smooth, compact submanifold $\mathcal{N}$ of $\mathbb{R}^4,$ of dimension $ 1 $ and codimension $3$, there exists a number $\kappa>0$ such that in the $\kappa$-neighborhood $U_\kappa(\mathcal{N}):=\{ \varv  \in \mathbb{R}^4: \, \text{ dist}(\varv , \mathcal{N})<\kappa \}$ of $\mathcal{N}$, the following property holds : for all $\varv \in U_\kappa(\mathcal{N}),$ there exists a unique point $\pi(\varv) \in \mathcal{N}$ such that $$\abs{\varv-\pi(\varv)}=\text{ dist}(\varv, \mathcal{N}).$$
%	\begin{align*}
%	\abs{\varv-\pi(\varv)}=\text{ dist}(\varv, \mathcal{N}).
%	\end{align*}
	The mapping $\varv \in U_\kappa(\mathcal{N}) \rightarrow \pi(\varv),$ called the nearest point projection onto $\mathcal{N}$, is smooth.
\end{defn} 

\medskip

\noindent Set $u_1:= Q_{11},u_2:= Q_{12},u_3:= M_{1},u_4:= M_{2} $ and $\Psi:=(u_1,u_2,u_3,u_4)$. Let $e_{\ell}(\Psi(x))$ denote the energy density $e_{\ell}(\Psi(x)):= \frac{1}{2}\abs{\nabla \Psi(x)}^2 +{\ell}^{-1} \tilde{f}_B(\Psi(x)).$
Next, a Bochner-type inequality is established, in the regions where  the minimizer $\Psi^{\ell}$ of \eqref{modified energy functional} lies close to the vacuum manifold of bulk energy minimizers. 
\begin{thm}[Bochner-type inequality]\label{Bochner type inequality thm}
There exists a constant  $  C>0,$ independent of $\ell$, so that for $\Psi^{\ell}$, a global minimizer of $\tilde{\mathcal{E}}(\Psi)$ from  \eqref{modified energy functional} in the admissible space $\mathcal{A}, $ with $\mathbf{g}\in \mathcal{A}_{\min}$, such that $\text{ dist}(\Psi^{\ell}, \mathcal{N}) < \kappa_0 $, it holds that 
\begin{align} \label{Bochner type inequality}
	-\Delta e_{\ell}(\Psi^{\ell}) +\abs{\nabla^2\Psi^{\ell}}^2 \leq C\abs{\nabla \Psi^{\ell}}^4.
	\end{align}
\end{thm}
\begin{proof}The proof is divided into three steps. The first step focuses on the derivation of an inequality satisfied by  $ \Delta e_{\ell}(\Psi^{\ell}).$ Next it is established that for sufficiently small values of $\ell,$ a global minimizer $\Psi^{\ell}$ belongs to the $\kappa$-neighborhood $ U_\kappa(\mathcal{N})$ of $\mathcal{N}$, and hence we can use  the nearest point projection $\pi(\Psi^{\ell})$ of $\Psi^{\ell}$ onto $\mathcal{N}$ in the subsequent analysis.  The third step uses Theorem \ref{Hessian restricted to normal space is positive definite} to bound the distance between  $\Psi^{\ell}$ and $ \mathcal{N}$, leading to the Bochner inequality.
	
	\medskip
	\noindent
{\it Step 1 (Laplacian of $e_{\ell}(\Psi^{\ell})$)}.	Define  $ \abs{\nabla^2u_i}^2:= \sum_{j,k=1}^{2} (\frac{\partial^2u_i}{\partial x_j\partial x_k})^2$ for $ i=1,2,3,4.$  
For $\Psi^{\ell}_{x_j}:= (u_{1,x_j},u_{2,x_j},$ $u_{3,x_j},u_{4,x_j})$,   $\Delta\Psi^{\ell}_{x_j}:= (\Delta u_{1,x_j},$ $\Delta u_{2,x_j},\Delta u_{3,x_j},\Delta u_{4,x_j}),$ $\frac{1}{2}	\Delta (\abs{\nabla \Psi^{\ell}}^2)= \abs{\nabla^2\Psi^{\ell}}^2 + \sum_{j=1}^{2}\Psi^{\ell}_{x_j} \cdot \Delta \Psi^{\ell}_{x_j} .$
%	\begin{align}\label{Bochner type equqtion 1}
%\frac{1}{2}	\Delta (\abs{\nabla \Psi^{\ell}}^2)= \abs{\nabla^2\Psi^{\ell}}^2 + \sum_{j=1}^{2}\Psi^{\ell}_{x_j} \cdot \Delta \Psi^{\ell}_{x_j} .
%	\end{align}
%Recall from \eqref{minimization pde},	$\Delta u_i = {\ell}^{-1} \frac{\partial \tilde{f}_B }{\partial u_i}(\Psi^{\ell}) \text{ for all } i=1,2,3,4,$, which
%	Since $\Psi^{\ell}$ is a global minimizer of  $\tilde{\mathcal{E}}$, 
%	\begin{align}\label{minimization pde}
%	\Delta u_i = {\ell}^{-1} \frac{\partial \tilde{f}_B }{\partial u_i}(\Psi^{\ell}) \text{ for all } i=1,2,3,4,
%	\end{align}
%A use of \eqref{minimization pde} leads to $\Psi^{\ell}_{x_j} \cdot \Delta \Psi^{\ell}_{x_j} = {\ell}^{-1} \sum_{i, k=1}^{4}\frac{\partial^2 \tilde{f}_B }{\partial u_k\partial u_i}(\Psi^{\ell})\frac{\partial u_k}{\partial x_j}\frac{\partial u_i}{\partial x_j} \text{ for } j=1,2.$ 
	This combined with $\Psi^{\ell}_{x_j} \cdot \Delta \Psi^{\ell}_{x_j} = {\ell}^{-1} \sum_{i, k=1}^{4}\frac{\partial^2 \tilde{f}_B }{\partial u_k\partial u_i}(\Psi^{\ell})\frac{\partial u_k}{\partial x_j}\frac{\partial u_i}{\partial x_j} \text{ for } j=1,2$ obtained using \eqref{minimization pde},  yields 
%	\begin{align*}%\label{Bochner type equqtion 2}
	$$-\frac{1}{2}\Delta (\abs{\nabla \Psi^{\ell}}^2)+\abs{\nabla^2\Psi^{\ell}}^2= - {\ell}^{-1}  \sum_{j=1}^{2}\sum_{i, k=1}^{4}\frac{\partial^2 \tilde{f}_B }{\partial u_k\partial u_i}(\Psi^{\ell})\frac{\partial u_k}{\partial x_j}\frac{\partial u_i}{\partial x_j}=-{\ell}^{-1} \nabla\Psi^{\ell}:D^2\tilde{f}_B(\Psi^{\ell})\nabla\Psi^{\ell}.$$
%	\end{align*}
Also,	a use of \eqref{minimization pde} leads to 
%	\begin{align*}%\label{Expansion laplacian f_B}
$$	-{\ell}^{-1}\Delta \tilde{f}_B(\Psi^{\ell})+{\ell}^{-2}\abs{D\tilde{f}_B(\Psi^{\ell})}^2
	 =-{\ell}^{-1} \nabla\Psi^{\ell}:D^2\tilde{f}_B(\Psi^{\ell})\nabla\Psi^{\ell}.$$
%	\end{align*}
	The above two displayed inequalities and 
%	From \eqref{Bochner type equqtion 2},  \eqref{Expansion laplacian f_B} and
	 $e_{\ell}(\Psi^{\ell}):= \frac{1}{2}\abs{\nabla \Psi^{\ell}}^2 +{\ell}^{-1} \tilde{f}_B(\Psi^{\ell})$ implies
	\begin{align}\label{Expansion laplacian e_l}
		-\Delta e_{\ell}(\Psi^{\ell})+\abs{\nabla^2\Psi^{\ell}}^2+{{\ell}^{-2}}\abs{D\tilde{f}_B(\Psi^{\ell})}^2
	%	\sum_{i=1}^{4} \frac{1}{\ell^2} \bigg(\frac{\partial \tilde{f}_B }{\partial u_i} (\Psi)\bigg)^2 
		 =-{2{\ell}^{-1}} \nabla\Psi^{\ell}:D^2\tilde{f}_B(\Psi^{\ell})\nabla\Psi^{\ell}.
	\end{align}
	
	\medskip
	
	\noindent {\it Step 2 (Verify $\Psi^{\ell} \in U_\kappa(\mathcal{N})$)}.
	 In this step, we verify that $\Psi^{\ell} \in U_\kappa(\mathcal{N})$ i.e.  $\text{ dist}(\Psi^{\ell}, \mathcal{N})< \kappa $ for sufficiently small value of $\ell.$
	 For $\Psi^{\ell}= (S_{\ell} \cos\theta_{\ell},S_{\ell} \sin\theta_{\ell},R_{\ell} \cos\varphi_{\ell}, R_{\ell} \sin\varphi_{\ell}) \in \mathcal{A},$ choose $\Psi^{*}:= (Q_c \cos2\varphi_{\ell},Q_c \sin2\varphi_{\ell},M_c \cos\varphi_{\ell}, M_c \sin\varphi_{\ell})$ $\in \mathcal{N}.$  Then
	 \begin{align*}
	 &	\abs{\Psi^{\ell} -\Psi^{*}}^2\\&= (S_{\ell} \cos\theta_{\ell} -Q_c \cos2\varphi_{\ell})^2+(S_{\ell} \sin\theta_{\ell} -Q_c \sin2\varphi_{\ell})^2+  (R_{\ell} \cos\varphi_{\ell} -M_c \cos\varphi_{\ell})^2+(R_{\ell} \sin\varphi_{\ell} -M_c \sin\varphi_{\ell})^2\\& \leq 2(S_{\ell}  -Q_c)^2+ 2(Q_c \cos\theta_{\ell} -Q_c \cos2\varphi_{\ell})^2+ 2(Q_c \sin\theta_{\ell} -Q_c \sin2\varphi_{\ell})^2+(R_{\ell}-M_c)^2 \\&=2((S_{\ell} - Q_c)^2+ 2Q_c^2(1-\cos(\theta_{\ell}-2\phi_l))+ (R_{\ell} - M_c)^2).
	 \end{align*}
	 This together with Proposition \ref{Proposition: uniform convergence on compact sets} yields that $	\abs{\Psi^{\ell} -\Psi^{*}} < \kappa_0= \min({\kappa,\delta_0}), $ for sufficiently small value of $\ell.$ For $\pi(\Psi^{\ell})$ to be the nearest point projection of $\Psi^{\ell}$ onto  $\mathcal{N}$, it holds that $\abs{\Psi^{\ell} - \pi(\Psi^{\ell})} \leq \abs{\Psi^{\ell} - \Psi^{\min}} $ for any $\Psi_{\min} \in \mathcal{N}$. This implies that $\text{ dist}(\Psi^{\ell}, \mathcal{N})=\abs{\Psi^{\ell} - \pi(\Psi^{\ell})} \leq \abs{\Psi^{\ell} - \Psi^{*}} <  \kappa_0.$
 Since $\mathcal{N}$ is compact and $\tilde{f}_B$ is smooth, the local Lipschitz continuity  of $D^2\tilde{f}_B$ in \eqref{Expansion laplacian e_l} leads to
\begin{align*}
-{2{\ell}^{-1}} \nabla\Psi^{\ell}:D^2\tilde{f}_B(\Psi^{\ell})\nabla\Psi^{\ell}& \leq -{2{\ell}^{-1}} \nabla\Psi^{\ell}:D^2\tilde{f}_B(\pi(\Psi^{\ell}))\nabla\Psi^{\ell} + {2{\ell}^{-1}}\abs{D^2\tilde{f}_B(\Psi^{\ell})-D^2\tilde{f}_B(\pi(\Psi^{\ell}))}\abs{\nabla\Psi^{\ell}}^2\\&\leq
-{2{\ell}^{-1}} \nabla\Psi^{\ell}:D^2\tilde{f}_B(\pi(\Psi^{\ell}))\nabla\Psi^{\ell} +{2C_L{\ell}^{-1}}\text{ dist}(\Psi^{\ell}, \mathcal{N}) \abs{\nabla\Psi^{\ell}}^2,
\end{align*}
 where $\pi(\Psi^{\ell}) $ is the nearest point projection of $\Psi^\ell$ onto $\mathcal{N}.$ 	 
 \medskip
 
 \noindent {\it Step 3 (Bound of $\text{\rm dist}(\Psi^{\ell}, \mathcal{N})$)}. Since $\pi(\Psi^{\ell}) \in \mathcal{N},$ i.e. a minimizer of $\tilde{f}_B,$ $D^2\tilde{f}_B(\pi(\Psi^{\ell})) \geq 0.$ This combined with the above displayed expression, \eqref{Expansion laplacian e_l} and Young's inequality leads to 
	\begin{align}\label{Bochner type inequality eqn3}
-\Delta e_{\ell}(\Psi^{\ell})+\abs{\nabla^2\Psi^{\ell}}^2+\ell^{-2}\abs{D\tilde{f}_B(\Psi^{\ell})}^2
\leq { 4C_L^2\delta_1 \ell^{-2}}\text{ dist}^2(\Psi^{\ell}, \mathcal{N})  +\frac{1}{\delta_1}\abs{\nabla \Psi^{\ell}}^4,
\end{align}
where $\delta_1>0$ is small and will be chosen later. 

\medskip

\noindent	 %Moreover, $\varv -\pi(\varv)$ is a normal vector at any  point $\varv \in \mathcal{N}.$ 
	For the nearest point projection $\varv:=\pi(\Psi^{\ell}) \in \mathcal{N}$ of $\Psi^{\ell}$, the unit normal vector $\nu:=\frac{\Psi^{\ell}-\pi(\Psi^{\ell})}{\abs{\Psi^{\ell}-\pi(\Psi^{\ell})}} \in \mathbb{R}^4$ to $\mathcal{N}$ at the point $\pi(\Psi^{\ell})$, and $t:=\abs{\Psi^{\ell}-\pi(\Psi^{\ell})}=\text{ dist}(\Psi^{\ell}, \mathcal{N})< \kappa_0,$ \eqref{Hypothesis H2} implies
			\begin{align*}
m_0\text{ dist}(\Psi^{\ell}, \mathcal{N})	\leq	D\tilde{f}_B(\varv +t\nu )\cdot \nu =	D\tilde{f}_B(\Psi^{\ell})\cdot \nu \implies m_0^2\text{ dist}^2(\Psi^{\ell}, \mathcal{N})	\leq\abs{D\tilde{f}_B(\Psi^{\ell})\cdot \nu }^2\leq \abs{D\tilde{f}_B(\Psi^{\ell})}^2.
		\end{align*}
		Use this in \eqref{Bochner type inequality eqn3} and absorb the term in the left hand side for sufficiently small choice of $\delta_1$, to obtain 
			\begin{align*}
		-\Delta e_{\ell}(\Psi^{\ell})+\abs{\nabla^2\Psi^{\ell}}^2+\frac{1}{2l^2}\abs{D\tilde{f}_B(\Psi^{\ell})}^2
		\leq \frac{1}{\delta_1}\abs{\nabla \Psi^{\ell}}^4.
		\end{align*}
		This concludes the proof.
\end{proof}
\noindent The next theorem uses the Bochner-type inequality in \eqref{Bochner type inequality}, to bound the term  $\int_\Omega \abs{\nabla^2\Psi^{\ell}(x)}^2 \dx$ locally, independently of $\ell$. The proof uses the technique applied in \cite{Bethuel}.
\begin{thm}[$\h^2_{\text{loc}}(\Omega)$ bound for $(\Psi^{\ell})$ independent of $\ell$]\label{H2 bound locally}
	Let $\Omega \subset \mathbb{R}^2$ be a simply-connected bounded open set with smooth boundary.
Let $\Psi^{\ell}:=(\Qvec^{\ell}, \Mvec^{\ell})$ be a global minimizer of $\tilde{\mathcal{E}}$ from  \eqref{modified energy functional} in the admissible space $\mathcal{A}$,  with $\mathbf{g} \in \mathcal{A}_{\min}$ defined in \eqref{boundary condition}. Then the sequence $(\Psi^{\ell})$ is bounded in $\h^2_{\text{loc}}(\Omega)$, as $\ell \to 0$.
\end{thm}
\begin{proof}
	Since $(\Psi^{\ell}) \rightarrow \Psi_0$ strongly in $\h^1(\Omega),$ given a $\delta>0$ small, choose $R$ sufficiently small so that 
	\begin{align}\label{H2 regularity locally equation0}
	\int_{B(x_0,R)} \abs{\nabla \Psi^{\ell}}^2 \dx < \delta \text{ for all }x_0 \in \Omega \text{ and for all } \ell. 
	\end{align}
	Fix a point $x_0 \in \Omega$, set $d=\text{dist}(x_0, \partial \Omega).$ Let $\xi$ be a smooth function with support in $B(x_0,r) $ with $r=\min(\frac{d}{2},R)$ such that $\xi=1$ on $B(x_0, \frac{r}{2}).$ Multiply \eqref{Bochner type inequality} by $\xi^2$ and apply integration by parts to obtain
	\begin{align} \label{H2 regularity locally equation1}
	\int_\Omega \xi^2\abs{\nabla^2\Psi^{\ell}}^2 \dx\leq \int_\Omega ( \Delta\xi^2) e_{\ell}(\Psi^{\ell}) \dx+C\int_\Omega \xi^2 \abs{\nabla \Psi^{\ell}}^4 \dx. 	 
	\end{align}
%	\begin{align} \label{H2 regularity locally equation2}
%	\int_\Omega \xi^2\abs{\nabla^2\Psi^{\ell}}^2 \dx\leq C +C\int_\Omega \xi^2 \abs{\nabla \Psi^{\ell}}^4 \dx. 	 
%	\end{align}
%	Now  compute the bound of $\int_\Omega \xi^2 \abs{\nabla \Psi^{\ell}}^4\dx$ term on the right hand side of \eqref{H2 regularity locally equation1}. 
		For $\varphi : = \xi \abs{\nabla \Psi^{\ell}}^2,$ $\nabla \varphi=  \nabla  \xi \abs{\nabla \Psi^{\ell}}^2+\xi \nabla (\abs{\nabla \Psi^{\ell}}^2).$ A use of $\abs{\nabla \abs{\nabla \Psi^{\ell}}^2 } \leq c \abs{\nabla \Psi^{\ell}} \abs{\nabla^2\Psi^{\ell}}$ and $W^{1,1}(\Omega) \subset L^2(\Omega)$ 
		 i.e., $(\int_\Omega \varphi^2 \dx)^{\frac{1}{2}} \leq C\int_\Omega (\abs{\nabla \varphi} +\abs{\varphi} ) \dx $ for all $\varphi \in W^{1,1}(\Omega),$ and \eqref{Strong convergence in H1 equation1} 
		yields 
	\begin{align}  \label{H2 regularity locally equation3}
	\int_\Omega \xi^2 \abs{\nabla \Psi^{\ell}}^4 \dx\leq C\big(1+\big(\int_\Omega \xi \abs{\nabla \Psi^{\ell}} \abs{\nabla^2 \Psi^{\ell}} \dx \big)^2 \big)	 \leq C \big(1+\delta \int_\Omega \abs{\nabla^2 \Psi^{\ell}}^2\dx \big) , 
	\end{align}
where  the Cauchy-Schwarz inequality and \eqref{H2 regularity locally equation0} is applied in the last step. The definition of $e_{\ell}(\Psi^{\ell})$, \eqref{Strong convergence in H1 equation1} and the  smoothness of $\xi$  imply that $ \int_\Omega ( \Delta\xi^2) e_{\ell}(\Psi^{\ell}) \dx  \leq C\int_\Omega \abs{\nabla \Psi_0}^2 \dx \leq C.$ Apply this and  \eqref{H2 regularity locally equation3} to   \eqref{H2 regularity locally equation1}, and then   absorb $\int_\Omega \xi^2 \abs{\nabla^2 \Psi^{\ell}}^2 \dx$ term into the left hand side  for a sufficiently small choice of $\delta>0,$  leading to the expected bound $\int_\Omega \xi^2\abs{\nabla^2 \Psi^{\ell}}^2  \dx \leq C.$
% \begin{align}  \label{H2 regularity locally equation4}
% 	\int_\Omega \xi^2\abs{\nabla^2\Psi^{\ell}}^2 \dx\leq C\big(1+\delta \int_\Omega \abs{\nabla^2 \Psi^{\ell}}^2\dx \big). 	 
% \end{align}
% A sufficiently small choice of $\delta>0,$  absorbs $\int_\Omega \xi^2 \abs{\nabla \Psi^{\ell}}^2 \dx$ into the left hand side of \eqref{H2 regularity locally equation4},  leading to the expected bound $\int_\Omega \xi^2\abs{\nabla^2 \Psi^{\ell}}^2  \dx \leq C.$
%\begin{align}\label{H2 regularity locally equation4}
%\int_\Omega \xi^2\abs{\nabla \Psi^{\ell}}^4  \dx \leq C.
%\end{align}
This concludes the proof.
\end{proof}
%\subsection{The analysis near the boundary}\label{The analysis near the boundary}
The next proposition discusses the convergence of the minimizers, $\Psi^{\ell}$, near the boundary.
\begin{prop}\label{analysis near the boundary}
	Let $\Omega \subset \mathbb{R}^2$ be a simply-connected bounded open set with smooth boundary.
Let $\Psi^{\ell}:=(\Qvec^{\ell}, \Mvec^{\ell})$ be a global minimizer of \eqref{modified energy functional} in the admissible space $\mathcal{A}$,  with $\mathbf{g} \in \mathcal{A}_{\min}$ defined in \eqref{boundary condition}. Then  (i)~ 	$\abs{\nabla \Psi^{\ell}} \leq \frac{C}{\sqrt{\ell}}$ on $\Omega$, where $C$ depends on $\mathbf{g}$ and $\Omega$; (ii) $\abs{\,\abs{\Qvec^{\ell}}- Q_c} $ = {\scalebox{1.2}{$\scriptscriptstyle\mathcal{O}$}}$(1)$, $\abs{\,\abs{\Mvec^{\ell}}- M_c} $ = {\scalebox{1.2}{$\scriptscriptstyle\mathcal{O}$}}$(1)$, and $\abs{\,\cos(\theta_{\ell} -2\varphi_{\ell})- 1} $ = {\scalebox{1.2}{$\scriptscriptstyle\mathcal{O}$}}$(1)$,  as $\ell \rightarrow 0$ 
% $\abs{\Qvec^{\ell}} \rightarrow Q_c$ and $\abs{\Mvec^{\ell}} \rightarrow M_c$ as $\ell \rightarrow 0$ 
uniformly on $\bar{ \Omega}$; (iii) $\int_{\partial \Omega} \abs{\frac{\partial\Psi^{\ell}}{\partial \nu}}^2 \ds \leq C,$ where $C$ depends on $\mathbf{g}$  and $\Omega$; (iv)  $(\Psi^{\ell})$ remains bounded in $\h^2(\Omega).$ 	
\end{prop}
\begin{rem}
	The proof of $(i)$, $(ii)$ and $(iii)$ follow analogous to \cite[Theorem 1 (part B) and Proposition 3]{Bethuel}. Compared to \cite{Bethuel}, we have four variables $u_1, u_2, u_3$ and $u_4$ in the energy functional. For $x_0 \in \partial \Omega$, the  $\h^2$-bound of the minimizers $\Psi^{\ell}$ in $B(x_0, d) \cap \Omega$, for some positive $d,$ are proved in $(iv)$. The analysis  is split into two cases. First we assume that the boundary $\partial \Omega$ is flat near $x_0$ and apply the methods in \cite{Bethuel}.  When  $\partial \Omega$ is not flat near $x_0,$ the concept of 'straighten the boundary' \cite{Evance19}, which requires the smoothness of the boundary, is applied.  In the second case, we choose the change of coordinates $(x_1,x_2) \rightarrow (x_1,x_2 + h(x_1)),$ where the graph of $h$ locally represents $\partial \Omega.$ In the new coordinates, the function $\Psi^{\ell}$ becomes $\tilde{\Psi }^{\ell}$ defined in $U:=\{(x_1, x_2)|\, x_2>0\}\cap B(0,d)$ 
	and \eqref{minimization pde} becomes 
	\begin{align}\label{reformed PDE}
	L\tilde{\Psi }^{\ell} = {\ell}^{-1} \frac{\partial \tilde{f}_B(\tilde{\Psi }^{\ell} )}{\partial \tilde{\Psi }} \,\,\text{ on } U, \text{  and } \tilde{\Psi }^{\ell}=\tilde{\mathbf{g} } \,\, \text{ on } [x_2=0]\cap \partial U,
	\end{align}
	where 
	$L=\sum_{i,j=1}^{2} \frac{\partial }{\partial x_j}(a_{ij} \frac{\partial }{\partial x_i}),$ $a_{11}=1, a_{12}=h', a_{21}=h'$ and $a_{22}=(1+(h')^2). $ Then, a Bochner type inequality for the modified PDE \eqref{reformed PDE} is derived analogously to Theorem \ref{Bochner type inequality thm}. The $\mathbf{H}^2$-bounds for the additional/new terms in this inequality are established similarly to case I. A brief proof is given below. \qed
\end{rem}
\begin{proof}[\textbf{Proof of Proposition \ref{analysis near the boundary}}$(iv)$]
\noindent	{\it \textbf{Case I}: When $\partial \Omega$ is \textbf{flat} near $x_0$, i.e. $\Omega \cap B(x_0, d)=\{(x_1,x_2)| x_2>0\}\cap B(x_0, d)$,  for some positive $d.$}
	
	\medskip
	
\noindent	Let $\xi$  be a smooth function with support in $ B(x_0, r) $, $r= \min(d,R)$ such that $\xi =1$ on $ B(x_0, \frac{r}{2}) $. Multiply \eqref{Bochner type inequality} by $\xi^2$ and apply integration by parts to obtain
	\begin{align}\label{H2 bound for flat boundary1}
	\int_\Omega \xi^2\abs{\nabla^2\Psi^{\ell}}^2 \dx&\leq \int_\Omega ( \Delta\xi^2) e_{\ell}(\Psi^{\ell}) \dx+\int_{[x_2=0]}\xi^2 \frac{\partial e_{\ell}}{\partial x_2}(\Psi^{\ell}) \ds-  \int_{[x_2=0]}\frac{\partial \xi^2 }{\partial x_2}e_{\ell}(\Psi^{\ell})\ds\notag\\& \quad+C\int_\Omega \xi^2 \abs{\nabla \Psi^{\ell} }^4 \dx 
	=:T_1+T_2+T_3+T_4.	 
	\end{align}
	 A use of $\int_\Omega e_{\ell}(\Psi^{\ell}) \dx \leq \frac{1}{2} \int_\Omega \abs{\nabla \Psi_0}^2 \dx$ from \eqref{Strong convergence in H1 equation1} and the smoothness of $\xi $ leads to the following bound for $T_1:=\int_\Omega ( \Delta\xi^2) e_{\ell}(\Psi^{\ell}) \dx \leq \int_\Omega e_{\ell}(\Psi^{\ell}) \dx  \leq C. $
	Recall that $\tilde{f}_B(\Psi^{\ell})=0$ and $\frac{\partial \tilde{f}_B}{\partial u_i}(\Psi^{\ell})=0$  on $\partial \Omega$  as $\Psi^{\ell}= \mathbf{g} \in \mathcal{A}_{\min}$ on $\partial \Omega$. The smoothness of $\xi$, definition of  $e_{\ell}$, and $(iii)$ leads to   
$$	T_3: =\int_{[x_2=0]}\frac{\partial \xi^2 }{\partial x_2}e_{\ell}(\Psi^{\ell}) \ds  \leq C \int_{[x_2=0]} \abs{\nabla \Psi^{\ell}}^2\ds =  C\bigg(\int_{[x_2=0]} \abs{\frac{\partial \Psi^{\ell} }{\partial \nu}}^2 \ds+\int_{[x_2=0]} \abs{\frac{\partial \mathbf{g} }{\partial \tau}}^2 \ds \bigg) \leq C.$$
Since	$\frac{\partial \tilde{f}_B}{\partial x_2}(\Psi^{\ell}) = \sum_{i=1}^4\frac{\partial \tilde{f}_B}{\partial u_i}(\Psi^{\ell})  \frac{\partial u_i}{\partial x_2} =0$  on $\partial\Omega,$ $\frac{\partial e_{\ell}}{\partial x_2}(\Psi^{\ell})= \frac{1}{2}\frac{\partial \abs{\nabla \Psi^{\ell}}^2 }{\partial x_2}$   on $\partial\Omega.$ 
 A use of \eqref{minimization pde} leads to 
\begin{align*}
\Delta u_i ={\ell}^{-1} \frac{\partial \tilde{f}_B }{\partial u_i}(\Psi^{\ell})=0 \text{ on }\partial \Omega  \implies \frac{\partial^2 u_i}{\partial^2 x_2}=-\frac{\partial^2 u_i}{\partial^2 x_1} \text{ on }  \partial \Omega, \text{ for all } i=1,2,3,4.
\end{align*}
The Dirichlet boundary condition $\mathbf{g}:=(g_1,g_2,g_3,g_4)$, combined with the above identity yields that
\begin{align}\label{H2 bound for flat boundary2}
T_2:&=\int_{[x_2=0]}\xi^2 \frac{\partial e_{\ell}}{\partial x_2}(\Psi^{\ell}) \ds
 = \frac{1}{2}\int_{[x_2=0]} \xi^2 \frac{\partial \abs{\nabla \Psi^{\ell}}^2 }{\partial x_2}\ds
 %+ C\sum_{i=1}^{4} \int_{[x_2=0]} \xi^2 \frac{\partial \abs{\nabla u_i}^2 }{\partial x_2} \ds
%\notag\\&
=\sum_{i=1}^{4} \int_{[x_2=0]}\xi^2 \bigg(\frac{\partial g_i}{\partial x_1}\frac{\partial^2 u_i}{\partial x_2\partial x_1}- \frac{\partial u_i}{\partial x_2}\frac{\partial^2 g_i}{\partial x_1^2}\bigg)\ds.
\end{align}
Use integration by parts for the first term on the right hand side of \eqref{H2 bound for flat boundary2}. Then Holder's inequality with  
%$\int_{[x_2=0]}\abs{\frac{\partial u_i}{\partial x_2} }^2 \ds \leq C$ from 
(iii) and smoothness of $\xi$ leads to the bound $T_2 \leq C.$
The bound for $T_4$ is already established  in  Theorem \ref{H2 bound locally}.  Now combining the bounds for $T_1, T_2, T_3$ and $T_4$ 
%yields the bound $\int_\Omega \xi^2\abs{\nabla^2\Psi^{\ell}}^2 \dx \leq C.$ This 
concludes the proof for Case I.

\medskip

\noindent{\it \textbf{Case II}: When $\partial \Omega$ is \textbf{ not flat} near $x_0=0$.}

\medskip

\noindent %In this case, local coordinates, which straighten the boundary are introduced. 
 We use the similar notation $\Psi^{\ell}$ instead of $\tilde{\Psi }^{\ell}$.
  A use of  \eqref{reformed PDE},  algebraic calculations %of $\sum_{ k=1}^2  \frac{\partial \Psi }{\partial x_k}  \cdot L\big( \frac{\partial \Psi }{\partial x_k}\big)$ 
% the fact that $\frac{\partial  L{\Psi }}{\partial x_k}=L\big(\frac{\partial  {\Psi }}{\partial x_k}\big) + \big(2h''\frac{\partial^2 \Psi }{\partial x_1\partial x_2}+2h' h'' \frac{\partial^2 \Psi }{\partial x_2^2} +h''\frac{\partial \Psi }{\partial x_2} \big)
% , \,k=1,2,$ and 
% leads to
%$\sum_{ k=1}^2  \frac{\partial \Psi }{\partial x_k}  \cdot L\big( \frac{\partial \Psi }{\partial x_k}\big)= - \sum_{ k=1}^2\frac{\partial \Psi }{\partial x_k} \cdot\big(2h''\frac{\partial^2 \Psi }{\partial x_1\partial x_2}+2h' h'' \frac{\partial^2 \Psi }{\partial x_2^2} +h''\frac{\partial \Psi }{\partial x_2} \big) + {\ell}^{-1} \sum_{ k=1}^2\frac{\partial  }{\partial x_k}\big(\frac{\partial \tilde{f}_B}{\partial \Psi } \big) \cdot
% \frac{\partial \Psi }{\partial x_k}.$
% A combination of this
% and Holder's inequality in 
and   the inequality $L\big(\frac{1}{2} \abs{\nabla \Psi^{\ell}}^2\big) \geq  \sum_{ k=1}^2  \frac{\partial \Psi^{\ell} }{\partial x_k}  \cdot L\big( \frac{\partial \Psi^{\ell} }{\partial x_k}\big)
    +\alpha_3 \abs{\nabla^2\Psi^{\ell}}^2$ implies that
$$ L\big(\frac{1}{2} \abs{\nabla \Psi^{\ell}}^2\big)\geq -C\big(\abs{\nabla^2 \Psi^{\ell}}\abs{\nabla \Psi^{\ell}}+ \abs{\nabla \Psi^{\ell}}^2 \big)+\alpha_3 \abs{\nabla^2\Psi}^2+ {\ell}^{-1} \nabla\Psi^{\ell}:D^2\tilde{f}_B(\Psi^{\ell})\nabla\Psi^{\ell},$$
where $\alpha_3>0$ the ellipticity constant. A use of \eqref{reformed PDE} leads to 
$$- L( \tilde{f}_B(\Psi^{\ell})) =  -{\ell}^{-1}\sum_{i=1}^{4}  \bigg(\frac{\partial \tilde{f}_B }{\partial u_i} (\Psi^{\ell})\bigg)^2-  \sum_{j,k=1}^{4}\frac{\partial^2 \tilde{f}_B }{\partial u_j \partial u_k}\bigg(\frac{\partial u_j }{\partial x_1}\frac{\partial u_k }{\partial x_1}+2h'\frac{\partial u_j }{\partial x_1}\frac{\partial u_k }{\partial x_2}+(1+(h')^2)\frac{\partial u_j }{\partial x_2}\frac{\partial u_k }{\partial x_2}\bigg).$$
 The two inequalities above  yield that
 \begin{align}\label{H2 bound for not flat boundary7}
- L(e_{\ell}(\Psi^{\ell}))&+\alpha_3 \abs{D^2\Psi^{\ell}}^2 +\sum_{i=1}^{4} {\ell}^{-2} \bigg(\frac{\partial \tilde{f}_B }{\partial u_i} (\Psi^{\ell})\bigg)^2  \leq  C\abs{\nabla \Psi^{\ell}}\big(\abs{\nabla^2 \Psi^{\ell}}+ \abs{\nabla \Psi^{\ell}} \big)\notag\\&\quad-  {\ell}^{-1} \sum_{j,k=1}^{4}\frac{\partial^2 \tilde{f}_B }{\partial u_j \partial u_k}\bigg(2\frac{\partial u_j }{\partial x_1}\frac{\partial u_k }{\partial x_1}  +2h'\frac{\partial u_j }{\partial x_1}\frac{\partial u_k }{\partial x_2}+(2+(h')^2)\frac{\partial u_j }{\partial x_2}\frac{\partial u_k }{\partial x_2}\bigg).
\end{align}
Follow the steps of Theorem \eqref{Bochner type inequality thm} for the second term on the right hand side of \eqref{H2 bound for not flat boundary7}, and utilize  $\sum_{i=1}^{4} {\ell}^{-2} \big(\frac{\partial \tilde{f}_B }{\partial u_i} (\Psi^{\ell})\big)^2 \geq  0$ to obtain
%\begin{align}\label{H2 bound for not flat boundary8}
$$- L(e_{\ell}(\Psi^{\ell}))+\alpha_3 \abs{\nabla^2\Psi^{\ell}}^2 \leq \frac{1}{\delta_2}\abs{\nabla \Psi^{\ell}}^4+ \frac{1}{\delta_2}\abs{\nabla \Psi^{\ell}}^2 $$% \lesssim e_{\ell}^2(\Psi^{\ell}(x)),
%\end{align}
for some sufficiently small $\delta_2>0.$ 
%This implies
%\begin{align}\label{H2 bound for not flat boundary8}
%- L(e_{\ell}(\Psi))+\alpha_3 \abs{\nabla^2\Psi}^2 \leq \frac{1}{\delta_2}\abs{\nabla \Psi}^4+ \frac{1}{\delta_2}\abs{\nabla \Psi}^2.% \lesssim e_{\ell}^2(\Psi(x)).
%\end{align}
Let $\xi$  be a smooth function with support in $ B(x_0, r) $ with $r= \min(d,R)$ such that $\xi =1$ on $ B(x_0, \frac{r}{2}) $. Multiply the above displayed inequality by $\xi^2$, apply integration and \eqref{Strong convergence in H1 equation1} to obtain
%\begin{align*}%\label{H2 bound for not flat boundary9}
$$\alpha_3 \int_{U} \xi^2 \abs{\nabla^2\Psi^{\ell}}^2  \dx\leq \int_{U} \xi^2 L(e_{\ell}(\Psi^{\ell}))\dx +\frac{1}{\delta_2} \int_{U} \xi^2 \abs{\nabla \Psi^{\ell}}^4 \dx +C .$$
%\end{align*}
The second term on the right hand side of above displayed inequality can be bounded similarly to Theorem \ref{H2 bound locally}.
% Next it is established that 
%\begin{align}\label{H2 bound for not flat boundary10}
%\int_{U} \xi^2 L(e_{\ell}(\Psi)) \dx \text{ remains bounded}.
%\end{align}
A  use of integration by parts leads to 
\begin{align*}%\label{H2 bound for not flat boundary11}
\int_{U} \xi^2 L(e_{\ell}(\Psi^{\ell})) \dx&= \int_{U}  e_{\ell}(\Psi^{\ell})L(\xi^2)\dx + \int_{[x_2=0]} \big(2a_{12}\frac{\partial \xi^2}{\partial x_1}+ \xi^2 \frac{\partial a_{12}}{\partial x_1}   + a_{22} \frac{\partial \xi^2 }{\partial x_2}\big) e_{\ell}(\Psi^{\ell})\ds  \notag\\&\quad- \int_{[x_2=0]}a_{22}\xi^2 \frac{\partial e_{\ell}(\Psi^{\ell})}{\partial x_2} \ds =:T_5+T_6+T_7.
\end{align*}
The functions $a_{ij}$ for $ i,j=1,2$, and  $\xi$ are smooth and bounded. This together with \eqref{Strong convergence in H1 equation1} (resp.  $\tilde{f}_B(\Psi^{\ell}) =0$ on $\partial \Omega$ and $(iii)$) leads to the bound for  $T_5$ (resp. $T_6$).  
%Along with the smoothness of $a_{12}, a_{22}$ and $\xi$. The terms in $T_6$ are bounded using the facts that $a_{12}, a_{22}$ and $\xi$ are smooth bounded functions, $\tilde{f}_B(\Psi) =0$ on $\partial \Omega$ and $\int_{[x_2=0]} \abs{\nabla\Psi}^2\ds=\int_{[x_2=0]} \abs{\frac{\partial \mathbf{\tilde{g}}}{\partial \tau}}^2\ds+\int_{[x_2=0]}  \abs{\frac{\partial \Psi}{\partial \nu}}^2 \ds\leq C$ from $(iii)$. 
The fact  that $\frac{\partial \tilde{f}_B(\mathbf{\tilde{g}})}{\partial x_2}=0$ as $\mathbf{\tilde{g}} \in \mathcal{A}_{\min}$, applied to \eqref{reformed PDE} yields $L \Psi^{\ell} = {\ell}^{-1} \frac{\partial \tilde{f}_B}{\partial {\Psi }}=0$ on $\partial\Omega.$
% that is $\frac{\partial  }{\partial x_2}\big(a_{22} \frac{\partial \Psi^{\ell}}{\partial x_2}\big) =-\frac{\partial  }{\partial x_1}\big(a_{11} \frac{\partial \Psi^{\ell}}{\partial x_1}\big)-\frac{\partial  }{\partial x_2}\big(a_{12} \frac{\partial \Psi^{\ell}}{\partial x_1}\big)-\frac{\partial  }{\partial x_1}\big(a_{21} \frac{\partial \Psi^{\ell}}{\partial x_2}\big).$
Use this to replace the $\frac{\partial^2  \Psi^{\ell}}{\partial x_2^2}$ term in %the expansion of
% $\frac{\partial  \abs{\nabla \Psi^{\ell}}^2}{\partial x_2}
%= \frac{\partial  \Psi^{\ell}}{\partial x_1}\frac{\partial^2  \Psi^{\ell}}{\partial x_2 \partial x_1} + \frac{\partial  \Psi^{\ell}}{\partial x_2}\frac{\partial^2  \Psi^{\ell}}{\partial x_2^2}$. 
 $T_7=-\frac{1}{2}\int_{[x_2=0]}a_{22}\xi^2 \frac{\partial  \abs{\nabla \Psi^{\ell}}^2}{\partial x_2} \ds$. The bound for $T_7$  is obtained using similar ideas as for $T_2,$ and  employs several integration by parts, $(iii)$ and smoothness of $\xi,a_{22}$. This completes the proof. 
\end{proof}
%\begin{rem}\label{remark on c=0}
%	For $ c=0$,
	 %the energy functional $\mathcal{E}(\Qvec,\Mvec)$ can be written as
%	\begin{align*}%\label{energy functional c=0}
%	\mathcal{E}=& \frac{1}{2}\int_\Omega (\abs{\nabla Q_{11}}^2+\abs{\nabla Q_{12}}^2+\frac{1}{2\ell }(Q_{11}^2 +Q_{12}^2 -1)^2) \dx+  \frac{1}{2}\int_\Omega (\abs{\nabla M_{1}}^2+\abs{\nabla M_{2}}^2+\frac{1}{2\ell }( M_1^2 +M_2^2 - 1)^2) \dx. 
%	\end{align*}
	%This case is directly comparable with the Ginzburg-Landau energy functional and the analysis will follow similar lines in \cite{Bethuel}
%\end{rem}
%\subsection{The coupling parameter  $c<0$}
%\begin{rem}\label{The coupling parameter  c<0}
%	Note that Propositions \ref{Proposition: limiting functions}, \ref{Proposition: L infinity bound of solutions}, \ref{Proposition: uniform convergence on compact sets} will remain uninfluenced by the sign change of the parameter $c.$ Whereas Lemma \ref{Hessian restricted to normal space is positive definite} need to be rigorously checked.
%\end{rem} 

\section{Finite element analysis } \label{Finite element analysis}
This section is devoted to the finite element approximation of the solution  of  \eqref{minimization pde} and convergence analysis  in bounded, convex domain with polygonal boundary. We assume the boundary condition $\mathbf{g} \in \h^{\frac{3}{2}}(\partial\Omega)$ for the analysis in this section.
Section \ref{Weak and finite element formulations} presents the weak and finite element formulations  of the non-linear system \eqref{continuous nonlinear ferro strong form}.  The local existence, uniqueness of the discrete solutions and error analysis with $h-\ell$ dependency are main results of this section, and are stated in Section \ref{Main results}. Some auxiliary results required for the convergence analysis are presented in Section \ref{Auxiliary results}. This is followed by the proofs of the main results in Section \ref{A priori estimates}.
\begin{lem} [Regularity result]\label{regularity}
	Let $\Omega$ be a convex, bounded domain in $\mathbb{R}^2$ with polygonal boundary. Then for $\mathbf{g} \in \mathbf{H}^{\frac{3}{2}}(\partial \Omega),$ any solution $\Psi^{\ell}$ of \eqref{minimization pde}, i.e., $\Delta \Psi^{\ell} =\ell^{-1} D\tilde{f}_B(\Psi^{\ell} ) \text{ in } \Omega, \text{  and } \Psi^{\ell}=\mathbf{g} \text{  on } \partial\Omega, $ %solution of \eqref{continuous nonlinear ferro strong form} 
	belongs to  $\h^{2}(\Omega)$.
\end{lem}
\begin{proof}%[\textbf{Proof of Lemma \ref{regularity}}]

%	\end{align*}
  The Sobolev embedding result $H^1(\Omega) \hookrightarrow L^p(\Omega) $ for all $p \geq 1$, and the H\"older's inequality yields that $D\tilde{f}_B(\Psi^{\ell}) $ defined in \eqref{definition of Df} belongs to $\mathbf{L}^2(\Omega).$	 
%For $F \in \mathbf{L}^2(\Omega)$ and $\mathbf{g} \in \mathbf{H}^{\frac{3}{2}}(\partial\Omega)$, the elliptic regularity result \cite{grisvard} implies that  there exists a $\Psi \in \mathbf{H}^2(\Omega)$ such that $-\Delta \Psi =F \text{ in } \Omega \text{ and }\Psi = \mathbf{g} \text{ on } \partial \Omega. $
The elliptic regularity result \cite{grisvard}  with a bootstrapping argument \cite{Evance19} implies that $\Psi^\ell\in \mathbf{H}^{2}(\Omega)$. 
\end{proof}
\medskip

\noindent The finite element analysis of this section holds for any regular solution \cite{keller} $\Psi^{\ell}$ of the  Euler-Lagrange PDEs (see the weak formulation in \eqref{continuous nonlinear ferro})  with the uniform bound 
	\begin{align}\label{H2 bound for minimizers}
	\vertiii{\Psi^{\ell}}_2 < C,
	\end{align}
where	the constant $C>0$ is independent of $\ell$. 
	In particular, we  established this property in Proposition \ref{analysis near the boundary}$(iv)$ for the global minimizers $\Psi^\ell$ of $\tilde{\mathcal{E}}$ in the admissible space $\mathcal{A}.$
\subsection{Weak and finite element formulations}\label{Weak and finite element formulations}
 The weak formulation of the non-linear system in \eqref{minimization pde}  seeks $\Psi^\ell:= (u_1,u_2,u_3,u_4 ) \in {\mathcal{A}}$ such that for all $ \Phi:=(\varphi_1, \varphi_2,\varphi_3, \varphi_4) \in  \V $,
\begin{align}\label{continuous nonlinear ferro}
N(\Psi^\ell; \Phi):= A(\Psi^\ell, \Phi)+B_{1}(\Psi^\ell, \Phi)+B_2( \Psi^\ell,\Psi^\ell,\Phi)+ B_{3}(\Psi^\ell,\Psi^\ell,\Psi^\ell, \Phi)= 0,
\end{align}
where for ${\Xi}:= (\xi_1,\xi_2,\xi_3,\xi_4 ), \boldsymbol{\eta}:= (\eta_1,\eta_2,\eta_3,\eta_4 ),$ and $\Theta:=(\theta_1, \theta_2,\theta_3,\theta_4) \in \X$,	
\begin{align*}
 &A(\Theta, \Phi):= \sum_{i=1}^{4}\int_{\Omega}  \nabla \theta_i \cdot \nabla \varphi_i \dx, 
\quad  B_{1}(\Theta, \Phi):=- {\ell}^{-1}\sum_{i=1}^{4}\int_\Omega \theta_i \varphi_i \dx 
 ,\\&
B_2(  \boldsymbol{\eta},\Theta,\Phi):= \frac{{c{\ell}^{-1}}}{2}\bigg(\int_\Omega (\eta_4 \theta_4- \eta_3 \theta_3) \varphi_1\dx -\int_\Omega (\eta_3 \theta_4 +\eta_4 \theta_3 )\varphi_2 \dx-\int_\Omega (\eta_1 \theta_3 +\eta_3 \theta_1 )\varphi_3\dx\\&\quad\quad\quad\quad\quad\quad\quad\quad\quad-\int_\Omega(\eta_2 \theta_4 +\eta_4 \theta_2 )\varphi_3\dx-\int_\Omega (\eta_2\theta_3 +\eta_3\theta_2 )\varphi_4\dx +\int_\Omega (\eta_1 \theta_4+\eta_4\theta_1 )\varphi_4\dx \bigg),
\end{align*}
for  $\bar{\xi}_{ij}= (\xi_i, \xi_j),\bar{\eta}_{ij}= (\eta_i, \eta_j),\bar{\theta}_{ij}= (\theta_i, \theta_j), $  $\bar{\varphi}_{ij}= (\varphi_i, \varphi_j) \in (H^1(\Omega))^2$ with $(i,j)=(1,2) \text{ or } (3,4)$,
\begin{align*}
&B_3({\Xi},  \boldsymbol{\eta}, \Theta, \Phi):=\frac{1}{3\ell}\int_\Omega \left((\bar{\xi}_{12} \cdot \bar{\eta}_{12} )(\bar{\theta}_{12} \cdot \bar{\varphi}_{12})+(\bar{\xi}_{12} \cdot \bar{\theta}_{12})(\bar{\eta}_{12}  \cdot \bar{\varphi}_{12}) +(\bar{\eta}_{12}  \cdot \bar{\theta}_{12})(\bar{\xi}_{12} \cdot \bar{\varphi}_{12})\right)\dx \\& \quad\quad\quad\quad\quad\quad\quad+\frac{1}{3\ell}\int_\Omega \left((\bar{\xi}_{34} \cdot \bar{\eta}_{34} )(\bar{\theta}_{34} \cdot \bar{\varphi}_{34})+(\bar{\xi}_{34} \cdot \bar{\theta}_{34})(\bar{\eta}_{34}  \cdot \bar{\varphi}_{34}) +(\bar{\eta}_{34}  \cdot \bar{\theta}_{34})(\bar{\xi}_{34} \cdot \bar{\varphi}_{34})\right)\dx.
\end{align*}
Note that the trilinear form $B_2(\cdot,\cdot,\cdot)$ and the quadrilinear form $B_3(\cdot,\cdot,\cdot,\cdot)$ are symmetric in any two variables.
The superscript $\ell$ is suppressed in $\Psi^\ell$ from now on for brevity of notations. With the notations 
\begin{align*}
&a(\theta,\varphi):=\int_{\Omega}  \nabla \theta \cdot \nabla \varphi \dx, \,\,b_1(\theta,\varphi):=- {\ell}^{-1}\int_\Omega \theta \varphi \dx,\,\,b_2(\eta ,\theta,\varphi):={c{\ell}^{-1}} \int_\Omega \eta \theta \varphi \dx, \\& \text{ and  }b_{3}(\xi,\eta,\theta,\varphi):= {\ell}^{-1}\int_\Omega  \xi\eta\theta \varphi \dx \,\,\text{ 	for } \xi,\eta,\theta,\varphi \in X,
\end{align*}
 the terms in \eqref{continuous nonlinear ferro} can be expressed as
\begin{align*}
 &A(\Psi, \Phi):=\sum_{i=1}^{4} a(u_i, \varphi_i) , \quad B_{1}(\Psi, \Phi):=\sum_{i=1}^{4}b_1(u_i, \varphi_i),
% + a(u_2, \varphi_2) + a(u_3, \varphi_3) +a(u_4, \varphi_4) , \\&
%B_{1}(\Psi, \Phi):=b_1(u_1, \varphi_1) + b_1(u_2, \varphi_2) + b_1(u_3, \varphi_3) + b_1(u_4, \varphi_4) ,
\\&
B_2( \Psi,\Psi,\Phi):=-\frac{1}{2}b_2(u_3,u_3, \varphi_1)+\frac{1}{2}b_2(u_4,u_4, \varphi_1)-b_2(u_3,u_4, \varphi_2)-b_2(u_1,u_3, \varphi_3)\\& \quad\quad\quad\quad\quad\quad\quad\,\,-b_2(u_2,u_4, \varphi_3)-b_2(u_2,u_3, \varphi_4)+b_2(u_1,u_4, \varphi_4).
\end{align*}
The scalar product expansions of the terms in  $B_3(\cdot,\cdot,\cdot,\cdot)$, for example,   $${\ell^{-1}}\int_\Omega (\bar{\xi}_{12} \cdot \bar{\eta}_{12} )(\bar{\theta}_{12} \cdot \bar{\varphi}_{12})=b_3(\xi_1, \eta_1, \theta_1, \varphi_1)+b_3(\xi_2, \eta_2, \theta_1, \varphi_1)+b_3(\xi_1, \eta_1, \theta_2, \varphi_2)+b_3(\xi_2, \eta_2, \theta_2, \varphi_2),$$ leads to
\begin{align*}
&B_{3}(\Psi,  \Psi,\Psi, \Phi): =b_{3}(u_1,u_1,u_1,\varphi_1)+b_{3}(u_2,u_2,u_1,\varphi_1) + b_{3}(u_1,u_1,u_2,\varphi_2)+b_{3}(u_2,u_2,u_2,\varphi_2)\\&\quad\quad\quad\quad\quad\quad\quad+ b_{3}(u_3,u_3,u_3,\varphi_3)+ b_{3}(u_4,u_4,u_3,\varphi_3) +  b_{3}(u_3,u_3,u_4,\varphi_4)+ b_{3}(u_4,u_4,u_4,\varphi_4) .
\end{align*}
\begin{rem}
	The linear terms of the system \eqref{continuous nonlinear ferro} are denoted by  $A(\cdot, \cdot)$ and $B_1(\cdot, \cdot)$, and the terms with quadratic non-linearity (resp.   cubic non-linearity) are denoted by $B_{2}(\cdot,\cdot, \cdot)$ (resp. $B_{3}(\cdot,\cdot, \cdot, \cdot)$). The representation of the forms $A, B_i,i=1,2,3$ in terms of $a, b_i, i=1,2,3$ are noteworthy and eases the understanding of the properties (e.g. coercivity, boundedness) over the complicated vectorized formulations and the analysis. \qed
\end{rem} 
\noindent Let $\mathcal{T}$ be a shape regular triangulation \cite{ciarlet} of a convex polygonal domain $\Omega\subset \mathbb{R}^2$  into triangles. The mesh discretization parameter is $h= \max_{T \in \mathcal{T}} h_T,$ where $h_T= diam(T) $. 
%$\mathcal{E}_h^{i} $( resp. $  \mathcal{E}_h^{\partial}$) denote the interior (resp. boundary) edges of $\mathcal{T}$ and  $\mathcal{E}:=\mathcal{E}_h^{i} \cup  \mathcal{E}_h^{\partial}$. 
%The length of the edge $E$ is denoted by $h_E.$
Let $P_1(T)$ denote polynomials of degree at most one on $T.$ Define the finite element subspace of $\X$ by $\X_h:=(X_h)^4 $ with $$X_h:= \{v \in C^0(\overline{\Omega})|\,\, v|_T \in P_1(T) \text{ for all } T \in \mathcal{T} \}$$ equipped with the $H^1$ norm. 
The space $\X_h $ is equipped with the product norm $\vertii{\Phi_h}:=\sum_{j=1}^{4} \norm{\varphi_{j}}_1  $ for all $\Phi_h=(\varphi_{1},\varphi_{2}, \varphi_{3},\varphi_{4}) \in \X_h.$  Define $\V_h:=(V_h)^4 $ with $$V_h:= \{v \in C^0(\overline{\Omega})|\,\, v|_T \in P_1(T) \text{ for all } T \in \mathcal{T} \text{ and } v|_{\partial\Omega}=0\} \subset H^1_0(\Omega).$$

\medskip

\noindent The discrete non-linear problem corresponding to \eqref{continuous nonlinear ferro} seeks ${\Psi}_h:= (u_{1,h},u_{2,h},u_{3,h},u_{4,h} )\!\in\!\X_h$ such that ${\Psi}_h=\mathbf{g}_h$ on $\partial \Omega$ and for all $ \Phi_h \in  \V_h, $
\begin{align}\label{discrete nonlinear problem ferro}
N( {\Psi}_h; \Phi_h):= A( {\Psi}_h, \Phi_h)+B_{1}( {\Psi}_h, \Phi_h)+B_{2}(  {\Psi}_h, {\Psi}_h,\Phi_h) + B_{3}( {\Psi}_h, {\Psi}_h, {\Psi}_h, \Phi_h)= 0 ,
\end{align}
 where
%  $\mathbf{g}_h$ denotes an approximation of $\mathbf{g}$ (the choice of which will be specified later) on $\X_h|_{\partial\Omega}.$
$\mathbf{g}_h:={\rm I}_h\mathbf{g}$ be the Lagrange $P_1$ interpolation of $\mathbf{g}$ along $\partial \Omega$. 
\subsection{Main results}\label{Main results}
The main results of this section are presented now. This includes the energy norm error estimate, a best approximation result in $\X_h$, and the $\mathbf{L}^2 $ norm  error estimate. The proofs are provided in Section \ref{A priori estimates} and the  hidden constants in $"\lesssim"$ are detailed there.  The conforming finite element analysis for ferronematics and the imperative $h-\ell $ dependency  are not investigated earlier as far as we are aware. The methodology explored here is  non-identical to the analysis of Landau-de Gennes model  for nematic liquid crystals in \cite{FiniteelementanalysisLDG, DGFEM}, the non-linearity is different for ferronematics, and lifting technique is utilized to deal with the non-homogeneous boundary condition. 
\begin{thm}[Energy norm error estimate]\label{energy  norm error estimate ferro}
	Let ${\Psi}$ be a regular solution of \eqref{continuous nonlinear ferro} such that \eqref{H2 bound for minimizers} holds. 
	For
	a given fixed $\ell>0$,   a sufficiently small discretization parameter chosen as $h=O(\ell^{1+\varsigma})$ with $\varsigma>0,$ there exists a unique solution ${\Psi}_{h}$ to the discrete problem \eqref{discrete nonlinear problem ferro}  that approximates ${\Psi}$ such that	
	\begin{align*}
	\vertii{{\Psi}-{\Psi}_h} \lesssim  h.
	\end{align*}
\end{thm}
\begin{thm}[Best approximation result]\label{thmbestapproximationComforming}
	Let $\Psi$ be a regular solution of the non-linear system
	\eqref{continuous nonlinear ferro}  such that \eqref{H2 bound for minimizers} holds. 
	For a given fixed $\ell>0,$ a sufficiently small discretization parameter chosen as $h=O(\ell^{1+\varsigma} )$ with $\varsigma > 0$,  
	the unique discrete solution $\Psi_{h} $ of \eqref{discrete nonlinear problem ferro} that approximates $\Psi $ satisfies the best-approximation property 
	\begin{align*}
	\vertiii{\Psi- \Psi_{h }}_1\lesssim (1+\ell^{-1})\big( \min_{\Psi^*_{h } \in \X_h}\vertiii{\Psi- \Psi^*_{h }}_{1} + \vertiii{\mathbf{g} - \mathbf{g}_h}_{\frac{1}{2}, \partial \Omega}\big),
	\end{align*}
where	$\mathbf{g}_h$ denotes the Lagrange $P_1$ interpolation of $\mathbf{g}$. 
%along $\partial \Omega$. 
%	where $\mathbf{g}_h \in \X_{h}|_{\partial \Omega}$ denotes an approximation to $\mathbf{g}.$
\end{thm}
\begin{thm}[$\mathbf{L}^2 $ norm  error estimate]\label{L2_error_estimate}
	Let $\Psi$ be a regular solution of the non-linear system
	\eqref{continuous nonlinear ferro}  such that \eqref{H2 bound for minimizers} holds.   For a given fixed $\ell>0 $ and a sufficiently small discretization parameter chosen as $h=O(\ell^{1+\varsigma} )$ for  $\varsigma > 0$, the  unique discrete solution $\Psi_h$ that approximates $\Psi $  satisfies 
	\begin{align*}
	\vertiii{\Psi-\Psi_h}_{0}\lesssim   (1+\ell^{-1} )\big(h^2(1+\ell^{-1} ) +\vertiii{\mathbf{g}-\mathbf{g}_h}_{-\frac{1}{2}, \partial \Omega} \big),
	\end{align*}
	where	$\mathbf{g}_h$ denotes the Lagrange $P_1$ interpolation of $\mathbf{g}$. 
	%	where $\mathbf{g}_h \in \X_{h}|_{\partial \Omega}$ denotes an approximation to $\mathbf{g}.$
\end{thm}
\begin{rem}
Since,  the data approximation term $\vertiii{\mathbf{g}-\mathbf{g}_h}_{-\frac{1}{2}, \partial \Omega}  \lesssim \vertiii{\mathbf{g}-\mathbf{g}_h}_{0, \partial \Omega} \lesssim h^{\frac{3}{2}}$ for $\mathbf{g} \in \h^{\frac{3}{2}}(\partial\Omega),$ the $\mathbf{L}^2$ norm error estimate in Theorem \ref{L2_error_estimate} does not exhibit the optimal order convergence rate. In Section \ref{Nitsche's method}, we provide an analysis which leads to optimal order of convergence in $\mathbf{L}^2$ norm using Nitsche's method \cite{Nitsche1971}. In case of higher regularity,   $\mathbf{g} \in \h^{2}(\partial\Omega),$ we obtain the improved data approximation error $ \vertiii{\mathbf{g}-\mathbf{g}_h}_{0, \partial \Omega} \lesssim h^{2},$ which leads to the optimal convergence rate in  $\mathbf{L}^2$ norm for conforming FEM.  
\end{rem}
\subsection{Auxiliary results}\label{Auxiliary results}
This section presents some results that are useful for the analysis and establishes the discrete inf-sup condition   for a perturbed bilinear form.
 The next lemma states the boundedness and coercivity results frequently employed in the analysis. 
The proofs are skipped and are a consequence of  Holder's inequality,  and the Sobolev embedding results $H^{1}(\Omega) \hookrightarrow L^{3}(\Omega)$, $H^{1}(\Omega) \hookrightarrow L^{4}(\Omega)$    and $H^{2}(\Omega) \hookrightarrow L^{\infty}(\Omega)$ for $\Omega \subset \mathbb{R}^2$.
\begin{lem}[Boundedness and coercivity]\cite{ DGFEM, AposterioriRMAMNN} \label{boundedness_all ferro}\\
	$(i)$  For $\theta$, $\varphi \in X$, 	and  $\xi \in V$, there exists a constant $\alpha_0 >0$ such that 
	\begin{align*}
	a(\theta,\varphi)\leq \norm{\theta}_1 \norm{\varphi}_1, \,	b_1(\theta,\varphi)\leq \ell^{-1} \norm{\theta}_0 \norm{\varphi}_0, \text{ 	and } \alpha_0\norm{\xi}_1^2 \leq  a(\xi,\xi) .
	\end{align*} 
	For all $\Theta$, $\Phi \in \X$, and   $\Xi \in \V, $ it holds that
	$$A(\Theta,\Phi)\leq \vertiiii{\Theta}_1 \vertiiii{\Phi}_1,\, B_{1}(\Theta,\Phi)\leq  \ell^{-1} \vertiiii{\Theta}_0 \vertiiii{\Phi}_0,\text{ 	and } \alpha_0\vertiii{\Xi}_1^2  \leq A(\Xi,\Xi).$$	
	$(ii)$ For $\eta, \theta, \varphi \in X$, it holds that $b_2(\eta,\theta,\varphi)\lesssim  \ell^{-1} \abs{c}\norm{\eta}_1 \norm{\theta}_1 \norm{\varphi}_1.$
 For $ \eta \in L^{\infty}(\Omega)$ (resp. $ \eta \in H^{2}(\Omega)$), $\theta$, $\varphi\in  X$,  
	% and
	\begin{align*}
	&b_{2}(\eta,\theta, \varphi)
	\lesssim \ell^{-1}\abs{c} \norm{\eta}_{\infty} \norm{\theta}_0 \norm{\varphi}_0\,\,	(\text{resp. } 	b_{2}(\eta,\theta, \varphi)
	\lesssim \ell^{-1}\abs{c} \norm{\eta}_{2} \norm{\theta}_0 \norm{\varphi}_0).
	\end{align*}
For $\boldsymbol \eta,\Theta$, $\Phi \in \X$, it holds that	$B_2(\boldsymbol \eta,\Theta,\Phi)\lesssim  \ell^{-1} \abs{c}\vertiiii{\boldsymbol \eta}_1 \vertiiii{\Theta}_1 \vertiiii{\Phi}_1.$
	For	$ \boldsymbol \eta \in \mathbf{L}^{\infty}(\Omega)$ (resp. $\boldsymbol \eta \in \mathbf{H}^{2}(\Omega) $), $\Theta$, $\Phi\in  \X$,
	\begin{align*}
		B_{2}(\boldsymbol \eta,\Theta, \Phi)
		\lesssim \ell^{-1} \abs{c} \vertiiii{\boldsymbol \eta}_{\infty} \vertiiii{\Theta}_0 \vertiiii{\Phi}_0\,\,(\text{resp. } 	B_{2}(\boldsymbol \eta,\Theta, \Phi)
		\lesssim \ell^{-1}\abs{c}	\vertiiii{\boldsymbol \eta}_{2} \vertiiii{\Theta}_0 \vertiiii{\Phi}_0).
	\end{align*}
	$(iii)$ 	For $ \xi, \eta ,\theta,\varphi \in X$, it holds that $b_{3}(\xi, \eta ,\theta,\varphi) \lesssim \ell^{-1}\norm{ \xi}_1\norm{\eta }_1\norm{\theta}_1 \norm{\varphi}_1.$ For $\xi, \eta \in L^{\infty}(\Omega)$ (resp. $\xi, \eta \in H^{2}(\Omega)$), $\theta$, $\varphi\in  X$, 
	\begin{align*}
 b_{3}(\xi, \eta ,\theta,\varphi) \lesssim \ell^{-1}\norm{ \xi}_{\infty}\norm{\eta }_{\infty}\norm{\theta}_0 \norm{\varphi}_0 \,\,(\text{resp. } 	b_{3}(\xi, \eta ,\theta,\varphi) \lesssim  \ell^{-1}\norm{ \xi}_{2} \norm{ \eta }_{2} \norm{\theta}_0 \norm{\varphi}_0).
	\end{align*}
	For $\Xi, \boldsymbol{\eta}, \Theta, \Phi \in \X$, it holds that $	B_{3}({\Xi},\boldsymbol \eta,\Theta,\Phi)\lesssim \ell^{-1} \vertiiii{\Xi}_1\vertiiii{\boldsymbol \eta}_1 \vertiiii{\Theta}_1 \vertiiii{\Phi}_1.$	For $\Xi , \boldsymbol \eta \in \mathbf{L}^{\infty}(\Omega)$ (resp. $\Xi ,\boldsymbol \eta \in \mathbf{H}^{2}(\Omega)$), $\Theta$, $\Phi\in  \X$,
	\begin{align*}
	B_{3}({\Xi},\boldsymbol \eta,\Theta,\Phi)\lesssim \ell^{-1} \vertiiii{\Xi}_{\infty}\vertiiii{\boldsymbol \eta}_{\infty} \vertiiii{\Theta}_0 \vertiiii{\Phi}_0\,\, (\text{resp. } B_{3}(\Xi,\boldsymbol \eta,\Theta, \Phi)
	\lesssim \ell^{-1}
	\vertiiii{\Xi}_{2} \vertiiii{\boldsymbol \eta}_{2} \vertiiii{\Theta}_0 \vertiiii{\Phi}_0),
	\end{align*}
	where $"\lesssim"$ absorbs the constants in Sobolev embedding results.
\end{lem} 
\begin{lem}[Interpolation estimate]\label{Interpolation estimate}\cite{ciarlet, ErnGuermond}
	For ${ v} \in H^{2}(\Omega)$, there exists  ${{\rm{I}}_h v} \in X_h$ such that
	\begin{align*}
	\norm{{v}-{\rm{I}}_{h}v }_{0} +h	\norm{{v}-{\rm{I}}_{h}v }_{1} \leq C_I h^{2}  \abs{{ v}}_{H^{2}(\Omega)}, \text{ 	and }	\norm{{v}-{\rm{I}}_{h}v }_{L^{\infty}(\Omega)} \leq C_I h  \abs{{ v}}_{H^{2}(\Omega)},
	\end{align*}	
	where $C_I$  is a positive constant independent of $h$.
\end{lem}
\begin{lem}[Properties of bilinear, trilinear, and quadrilinear forms] \label{Properties of bilinear trilinear and quadrilinear forms}The following bounds  hold. \\
$(i)$ For $\boldsymbol \eta \in \mathbf{H}^2(\Omega)$ and for all $\Phi_h \in \X_h,$ $$A( \boldsymbol \eta-{\rm I}_h \boldsymbol \eta, \Phi_h)\lesssim h \vertiii{\boldsymbol \eta}_2 \vertiii{\Phi_h}_1, \text{ 	and }B_1(\boldsymbol \eta -{\rm I}_h \boldsymbol \eta, \Phi_h)\lesssim \ell^{-1} h^2 \vertiii{\boldsymbol \eta}_2 \vertiii{\Phi_h}_1.$$
		$(ii)$ For $\boldsymbol \xi,\boldsymbol \eta \in \mathbf{H}^2(\Omega)$ and for all $\Theta_h,\Phi_h \in \X_h,$
		\begin{align*}
 &B_2( \boldsymbol\eta-{\rm{I}}_h \boldsymbol\eta ,\Theta_h, \Phi_h )\lesssim \ell^{-1}h \abs{c}   \vertiii{\boldsymbol\eta}_2 \vertiii{\Theta_h}_1\vertiii{\Phi_h}_1,\\&
	 B_2(\boldsymbol \xi,  \boldsymbol \eta,\Phi_h )- B_2( {\rm{I}}_h\boldsymbol \xi,{\rm{I}}_h\boldsymbol \eta, \Phi_h ) \lesssim \ell^{-1}h^2 \abs{c}  \vertiii{\boldsymbol \xi}_2\vertiii{\boldsymbol \eta}_2 \vertiii{\Phi_h}_1.
		\end{align*}
	$(iii)$	For $\boldsymbol \xi,\boldsymbol \eta \in \mathbf{H}^2(\Omega)$ and for all $\Theta_h,\Phi_h \in \X_h,$
	\begin{align*}
	&B_3( \boldsymbol \xi , \boldsymbol \eta, \Theta_h, \Phi_h)-B_3({\rm I}_h \boldsymbol \xi , {\rm I}_h \boldsymbol \eta, \Theta_h, \Phi_h) \lesssim \ell^{-1}h  \vertiii{\boldsymbol \xi}_2 \vertiii{\boldsymbol \eta}_2\vertiii{\Theta_h}_1\vertiii{\Phi_h}_1,\\&
 B_3(\boldsymbol \xi, \boldsymbol \eta,\boldsymbol \eta, \Phi_h )-	B_3({\rm I}_h \boldsymbol \xi, {\rm I}_h \boldsymbol \eta,{\rm I}_h \boldsymbol \eta, \Phi_h )  \lesssim  \ell^{-1}h^2\vertiii{\boldsymbol \xi}_2\vertiii{\boldsymbol \eta}_2^2\vertiii{\Phi_h}_1.
		\end{align*}
		For $\boldsymbol \eta \in \mathbf{H}^1(\Omega)$ and for all $\Theta_h,\Phi_h \in \X_h,$
			\begin{align*}
&	2B_3( \boldsymbol \eta, \boldsymbol \eta, \boldsymbol \eta, \Phi_h ) - 3B_3( \boldsymbol \eta,\boldsymbol \eta,\Theta_{h}, \Phi_h )+B_3(\Theta_{h}, \Theta_{h},\Theta_{h}, \Phi_h )\\&\quad \lesssim \ell^{-1}\vertii{  \Theta_{h}-\boldsymbol \eta}^2(\vertii{ \Theta_{h}-\boldsymbol \eta}
%	\\&\qquad\qquad\qquad\qquad\qquad\qquad\qquad\qquad\qquad\qquad\qquad\qquad\qquad\qquad\qquad
	+\vertiii{ \boldsymbol \eta}_1)\vertiii{\Phi_h}_1.
	\end{align*}
	For $\boldsymbol \eta \in \mathbf{H}^1(\Omega)$ and for all $\Theta_1,\Theta_2,\Phi_h \in \X_h,$
	\begin{align*}
	&	3B_3(\boldsymbol \eta,\boldsymbol \eta, \Theta_1 , \Phi_{h} ) - B_3(\Theta_1,\Theta_1,\Theta_1,\Phi_{h})-3B_3(\boldsymbol \eta,\boldsymbol \eta, \Theta_2 , \Phi_{h} )+B_3(\Theta_2,\Theta_2,\Theta_2,\Phi_{h}) \\&\lesssim \vertiii{\Theta_2-\Theta_1}_1\big(\vertiii{\Theta_1-\boldsymbol \eta}_1^2+\vertiii{\Theta_2-\boldsymbol \eta}_1^2 +(\vertiii{\Theta_1-\boldsymbol \eta}_1+\vertiii{\Theta_2-\boldsymbol \eta}_1)\vertiii{\boldsymbol \eta}_1\big) \vertiii{\Phi_{h}}_1,
	\end{align*}
	where  "$\lesssim$" depends on  $C_I$, measure of the domain,  and the constants in Sobolev embedding results. 
\end{lem}
\begin{proof}
	$(i)$ Lemmas \ref{boundedness_all ferro}$(i)$ and \ref{Interpolation estimate}   yield 
\begin{align*}
&	A(\boldsymbol\eta-{\rm{I}}_h \boldsymbol\eta, \Phi_h) \leq \vertiiii{\boldsymbol\eta-{\rm{I}}_h \boldsymbol\eta}_1 \vertiiii{ \Phi_h}_1 \lesssim h\vertiiii{\boldsymbol\eta}_{2}\vertiii{\Phi_h}_1 .\\&
		B_1(\boldsymbol\eta-{\rm{I}}_h \boldsymbol\eta, \Phi_h) \leq \ell^{-1}\vertiiii{\boldsymbol\eta-{\rm{I}}_h \boldsymbol\eta}_0 \vertiiii{ \Phi_h}_0 \lesssim  \ell^{-1}h^2 \vertiiii{\boldsymbol\eta}_{2}\vertiii{\Phi_h}_1 .
\end{align*}
$(ii)$ Lemmas \ref{boundedness_all ferro}$(ii)$ and \ref{Interpolation estimate} imply
$$	B_2( \boldsymbol\eta-{\rm{I}}_h \boldsymbol\eta,\Theta_h, \Phi_h ) \leq \ell^{-1}\abs{c}\vertiiii{\boldsymbol\eta-{\rm{I}}_h \boldsymbol\eta}_1\vertiiii{\Theta_h}_1 \vertiiii{ \Phi_h}_1 \lesssim  \ell^{-1}h\abs{c} \vertiiii{\boldsymbol\eta}_{2}\vertiiii{\Theta_h}_1\vertiii{\Phi_h}_1 .$$
Lemma \ref{Interpolation estimate} and the Sobolev embedding result  $\mathbf{H}^2(\Omega) \hookrightarrow \mathbf{L}^{\infty}(\Omega)$  lead to $\vertiii{\textrm{I}_{h}\boldsymbol\xi}_{\infty} \leq \vertiii{\textrm{I}_{h}\boldsymbol\xi-\boldsymbol\xi}_{\infty}+\vertiii{\boldsymbol\xi}_{\infty}\lesssim  \vertiii{\boldsymbol\xi}_2$. This plus the linearity of $B_2(\cdot,\cdot,\cdot)$ in first and second variables, Lemma \ref{boundedness_all ferro}$(ii)$ and Lemma  \ref{Interpolation estimate} show 
\begin{align*}%\label{ball to ball equation10 ferro}
&  B_2(\boldsymbol\xi,  \boldsymbol\eta,\Phi_h )-B_2( {\rm{I}}_h \boldsymbol\xi,\textrm{I}_{h}\boldsymbol\eta, \Phi_h )= B_2(\boldsymbol\xi-{\rm{I}}_h \boldsymbol\xi , \boldsymbol\eta,\Phi_{h} ) + B_2(\textrm{I}_{h}\boldsymbol\xi , \boldsymbol\eta-\textrm{I}_{h}\boldsymbol\eta,\Phi_{h} ) \\& \lesssim  \ell^{-1}\abs{c} (\vertiii{\boldsymbol\xi-{\rm{I}}_h\boldsymbol\xi}_0\vertiii{\boldsymbol\eta}_2 +\vertiii{\textrm{I}_{h}\boldsymbol\xi}_{\infty}\vertiii{\boldsymbol\eta-\textrm{I}_{h}\boldsymbol\eta}_0)\vertiii{\Phi_h}_0 \lesssim  \ell^{-1}h^2\abs{c}\vertiiii{\boldsymbol\xi}_{2}\vertiiii{\boldsymbol\eta}_{2}\vertiii{\Phi_h}_1 .
\end{align*}
$(iii)$ {\it Proof of 1st inequality}.	The linearity of $B_3(\cdot,\cdot,\cdot,\cdot)$ in first two variables,  Lemma \ref{boundedness_all ferro}$(iii)$, and Lemma  \ref{Interpolation estimate} with  $\vertiii{\textrm{I}_{h}\boldsymbol\eta}_{1} \lesssim\vertiii{\boldsymbol\eta}_2$ lead to 
	\begin{align*}
&	B_3( \boldsymbol \xi , \boldsymbol \eta, \Theta_h, \Phi_h)-B_3({\rm I}_h \boldsymbol \xi , {\rm I}_h \boldsymbol \eta, \Theta_h, \Phi_h) = B_3( \boldsymbol \xi, \boldsymbol \eta-{\rm{I}}_h\boldsymbol \eta, \Theta_h, \Phi_h)+B_3( \boldsymbol \xi-{\rm{I}}_h \boldsymbol \xi, {\rm{I}}_h\boldsymbol \eta, \Theta_h, \Phi_h)\\&
\lesssim \ell^{-1}(	\vertiii{\boldsymbol \xi}_1 	\vertiii{\boldsymbol \eta-{\rm{I}}_h\boldsymbol \eta}_1 +	\vertiii{\boldsymbol \xi-{\rm{I}}_h \boldsymbol \xi}_1 	\vertiii{{\rm{I}}_h\boldsymbol \eta}_1)\vertiii{\Theta_h}_1\vertiii{\Phi_h}_1 \lesssim  \ell^{-1}h \vertiii{\boldsymbol \xi}_2 \vertiii{\boldsymbol \eta}_2 \vertiii{\Theta_h}_1\vertiii{\Phi_h}_1.
	\end{align*}
{\it Proof of 2nd inequality}. We utilize the linearity of $B_3(\cdot,\cdot,\cdot,\cdot)$ in first three variables, Lemma \ref{boundedness_all ferro}$(iii)$, Lemma  \ref{Interpolation estimate} with  $\vertiii{\textrm{I}_{h}\boldsymbol\xi}_{\infty} \lesssim \vertiii{\boldsymbol\xi}_2$, $\vertiii{\textrm{I}_{h}\boldsymbol\eta}_{\infty} \lesssim\vertiii{\boldsymbol\eta}_2$, and $\mathbf{H}^2(\Omega) \hookrightarrow \mathbf{L}^{\infty}(\Omega)$  to prove the second inequality in $(iii).$
\begin{align*}
&	 B_3(\boldsymbol{\xi},\boldsymbol{\eta}, \boldsymbol{\eta}, \Phi_h ) -B_3(\textrm{I}_{h}\boldsymbol{\xi}, {\rm{I}}_h\boldsymbol{\eta}, {\rm{I}}_h\boldsymbol{\eta},\Phi_h )\notag \\&= B_3(\boldsymbol{\xi}-\textrm{I}_{h}\boldsymbol{\xi},\boldsymbol{\eta}, \boldsymbol{\eta}, \Phi_h ) +B_3(\textrm{I}_{h}\boldsymbol{\xi},\boldsymbol{\eta}-{\rm{I}}_h\boldsymbol{\eta},\boldsymbol{\eta} , \Phi_{h} ) + B_3({\rm{I}}_h\boldsymbol{\xi},{\rm{I}}_h\boldsymbol{\eta},\boldsymbol{\eta}-{\rm{I}}_h\boldsymbol{\eta},\Phi_{h})
\\& \lesssim  \ell^{-1} (\vertiii{\boldsymbol\xi-{\rm{I}}_h\boldsymbol\xi}_0\vertiii{\boldsymbol\eta}_{2}^2 +\vertiii{\textrm{I}_{h}\boldsymbol\xi}_{\infty}\vertiii{\boldsymbol\eta-\textrm{I}_{h}\boldsymbol\eta}_0\vertiii{\boldsymbol\eta}_{\infty}+\vertiii{\textrm{I}_{h}\boldsymbol\xi}_{\infty}\vertiii{\textrm{I}_{h}\boldsymbol\eta}_{\infty}\vertiii{\boldsymbol\eta-\textrm{I}_{h}\boldsymbol\eta}_0)\vertiii{\Phi_h}_0 \\& 
\lesssim  \ell^{-1}h^2\vertiii{\boldsymbol{\xi}}_{2}\vertiii{\boldsymbol{\eta}}_{2}^2\vertiii{\Phi_h}_1 . \notag
\end{align*}
{\it Proof of 3rd inequality}. The linearity (resp. symmetry) of $B_3(\cdot,\cdot,\cdot,\cdot)$ in first three (resp. first and third, or second and third) variables  and  re-grouping of terms shows
\begin{align*}
 &2B_3( \boldsymbol \eta, \boldsymbol \eta, \boldsymbol \eta, \Phi_h ) - 3B_3( \boldsymbol \eta, \boldsymbol \eta,\Theta_{h}, \Phi_h )+B_3(\Theta_{h}, \Theta_{h},\Theta_{h}, \Phi_h )\notag\\&=
- 2B_3( \boldsymbol \eta, \boldsymbol \eta,\Theta_{h}- \boldsymbol \eta, \Phi_h ) +B_3( \boldsymbol \eta,\Theta_{h}-\boldsymbol \eta,\Theta_{h}, \Phi_h )+ B_3( \Theta_{h}-\boldsymbol \eta, \Theta_{h},\Theta_{h}, \Phi_h )\\&
 = (B_3( \Theta_{h}-\boldsymbol \eta, \Theta_{h},\Theta_{h}, \Phi_h )-B_3( \Theta_{h}-\boldsymbol \eta, \boldsymbol \eta,\boldsymbol \eta, \Phi_h ))+B_3(\boldsymbol \eta, \Theta_{h}-\boldsymbol \eta, \Theta_{h}-\boldsymbol \eta, \Phi_h ).
\end{align*}
This plus the identity 
\begin{align*}
B_3(\cdot , \Theta_{h},\Theta_{h}, \Phi_h )-B_3(\cdot, \boldsymbol \eta,\boldsymbol \eta, \Phi_h )= B_3( \cdot, \Theta_{h}-\boldsymbol \eta,\Theta_{h}-\boldsymbol \eta, \Phi_h ) + 2B_3(\cdot, \Theta_{h}-\boldsymbol \eta,\boldsymbol \eta, \Phi_h ),
\end{align*}
for the first term in the above displayed equation, and Lemma \ref{boundedness_all ferro}$(iii)$ allow
\begin{align}\label{boundedness B3 iii}
&2B_3( \boldsymbol \eta, \boldsymbol \eta, \boldsymbol \eta, \Phi_h ) - 3B_3( \boldsymbol \eta, \boldsymbol \eta,\Theta_{h}, \Phi_h )+B_3(\Theta_{h}, \Theta_{h},\Theta_{h}, \Phi_h )\notag\\&=
B_3( \Theta_{h}-\boldsymbol \eta, \Theta_{h}-\boldsymbol \eta,\Theta_{h}-\boldsymbol \eta, \Phi_h )+ 3B_3( \Theta_{h}-\boldsymbol \eta, \Theta_{h}-\boldsymbol \eta,\boldsymbol \eta, \Phi_h )\\&\lesssim \ell^{-1} \vertiii{\Theta_{h}-\boldsymbol \eta}_1^2(\vertiii{\Theta_{h}-\boldsymbol \eta}_1 + \vertiii{\boldsymbol \eta}_1)\vertiii{\Phi_h}_1.\notag
\end{align}
 {\it Proof of 4th inequality}. Add and subtract the term $2B_3(\boldsymbol \eta,\boldsymbol \eta, \boldsymbol \eta , \Phi_{h} )$, and then re-arrange the terms following the re-grouping of \eqref{boundedness B3 iii} to obtain 
 \begin{align}\label{Boundedness B3 equn2}
&3B_3(\boldsymbol \eta,\boldsymbol \eta, \Theta_1 , \Phi_{h} ) - B_3(\Theta_1,\Theta_1,\Theta_1,\Phi_{h})-3B_3(\boldsymbol \eta,\boldsymbol \eta, \Theta_2 , \Phi_{h} )+B_3(\Theta_2,\Theta_2,\Theta_2,\Phi_{h}) 
\notag \\&=( 2B_3(\boldsymbol \eta,\boldsymbol \eta, \boldsymbol \eta , \Phi_{h} ) -3B_3(\boldsymbol \eta,\boldsymbol \eta, \Theta_2 , \Phi_{h} ) + B_3(\Theta_2,\Theta_2,\Theta_2	,\Phi_{h}))\notag\\& 
\quad-(2 B_3(\boldsymbol \eta,\boldsymbol \eta, \boldsymbol \eta , \Phi_{h} )-3B_3(\boldsymbol \eta,\boldsymbol \eta, \Theta_1 , \Phi_{h} ) + B_3(\Theta_1,\Theta_1,\Theta_1,\Phi_{h}))\notag\\& =(B_3(\Theta_2-\boldsymbol \eta,\Theta_2-\boldsymbol \eta, \Theta_2-\boldsymbol \eta , \Phi_{h} )-B_3(\Theta_1-\boldsymbol \eta,\Theta_1-\boldsymbol \eta, \Theta_1-\boldsymbol \eta , \Phi_{h} ))\notag\\&\quad+(3B_3(\Theta_2-\boldsymbol \eta,\Theta_2-\boldsymbol \eta, \boldsymbol \eta , \Phi_{h} )-3B_3(\Theta_1-\boldsymbol \eta,\Theta_1-\boldsymbol \eta, \boldsymbol \eta , \Phi_{h} )).
\end{align}
Re-arrange the first term of \eqref{Boundedness B3 equn2} as 
\begin{align}\label{Boundedness B3 equn3}
&B_3(\Theta_2-\boldsymbol \eta,\Theta_2-\boldsymbol \eta, \Theta_2-\boldsymbol \eta , \Phi_{h} )-B_3(\Theta_1-\boldsymbol \eta,\Theta_1-\boldsymbol \eta, \Theta_1-\boldsymbol \eta , \Phi_{h} )\notag\\&= B_3(\Theta_2-\boldsymbol \eta,\Theta_2-\boldsymbol \eta, \Theta_2-\Theta_1 , \Phi_{h} )+(B_3(\Theta_2-\boldsymbol \eta,\Theta_2-\boldsymbol \eta, \Theta_1 -\boldsymbol \eta , \Phi_{h} )-B_3(\Theta_1-\boldsymbol \eta,\Theta_1-\boldsymbol \eta, \Theta_1-\boldsymbol \eta , \Phi_{h} ) ).
\end{align}
Next we apply the identity 
\begin{align*}
B_3(\Theta_2-\boldsymbol \eta,\Theta_2-\boldsymbol \eta, \cdot , \Phi_{h} )-B_3(\Theta_1-\boldsymbol \eta,\Theta_1-\boldsymbol \eta, \cdot , \Phi_{h} )=B_3(\Theta_2-\Theta_1,(\Theta_1+\Theta_2)-2\boldsymbol \eta, \cdot , \Phi_{h} )
\end{align*}
 to the second terms of \eqref{Boundedness B3 equn2} and \eqref{Boundedness B3 equn3}, and Lemma \ref{boundedness_all ferro}$(iii)$ to obtain 
\begin{align*}
&3B_3(\boldsymbol \eta,\boldsymbol \eta, \Theta_1 , \Phi_{h} ) - B_3(\Theta_1,\Theta_1,\Theta_1,\Phi_{h})-3B_3(\boldsymbol \eta,\boldsymbol \eta, \Theta_2 , \Phi_{h} )+B_3(\Theta_2,\Theta_2,\Theta_2,\Phi_{h}) 
\\&=B_3(\Theta_2-\boldsymbol \eta,\Theta_2-\boldsymbol \eta, \Theta_2-\Theta_1 , \Phi_{h} )+ B_3(\Theta_2-\Theta_1,(\Theta_1+\Theta_2)-2\boldsymbol \eta, \Theta_1-\boldsymbol \eta  , \Phi_{h} )\\& \quad + 3 B_3(\Theta_2-\Theta_1,(\Theta_1+\Theta_2)-2\boldsymbol \eta, \boldsymbol \eta  , \Phi_{h} )\\& \lesssim \vertiii{\Theta_2-\Theta_1}_1(\vertiii{\Theta_1-\boldsymbol \eta}_1^2+\vertiii{\Theta_2-\boldsymbol \eta}_1^2 +(\vertiii{\Theta_1-\boldsymbol \eta}_1+\vertiii{\Theta_2-\boldsymbol \eta}_1)\vertiii{\boldsymbol \eta}_1) \vertiii{\Phi_{h}}_1,
\end{align*}
where % Lemma \ref{boundedness_all ferro}$(iii)$ and 
the triangle inequality $\vertiii{(\Theta_1+\Theta_2)-2\boldsymbol \eta}_1 \leq \vertiii{\Theta_1-\boldsymbol \eta}_1+\vertiii{\Theta_2-\boldsymbol \eta}_1$ is applied for the last step. 
 This completes the proof.
 \end{proof}
\noindent\textbf{Modified weak formulation}.
The nonlinear system \eqref{continuous nonlinear ferro} is equipped with non-homogeneous boundary conditions. The analysis in this paper is based on reformulation of \eqref{continuous nonlinear ferro} using lifting technique  \cite{ErnGuermond}
that reduces the problem to a system of nonlinear PDEs with homogeneous boundary conditions. 

\medskip

\noindent For $\mathbf{g} \in \h^{\frac{3}{2}}(\partial\Omega),$ trace theorem \cite[Page 41]{Lions_Magenes} shows the existence of a $\Psi_{\mathbf{g}} \in \h^{2}(\Omega)$  such that $\Psi_{\mathbf{g}}= \mathbf{g}$ on $\partial \Omega$ and
\begin{align}\label{bound for Psig}
\vertiii{\Psi_{\mathbf{g}}}_{2} \lesssim \vertiii{\mathbf{g}}_{\frac{3}{2}, \partial\Omega}.
\end{align}
Set $\widetilde{\Psi}:={\Psi}-\Psi_{\mathbf{g}} \in \V$. A substitution of  ${\Psi}=\widetilde{\Psi} + {\Psi}_{\mathbf{g}}$ in \eqref{continuous nonlinear ferro} leads to a new non-linear system  given by: find $\widetilde{\Psi} \in \V$ such that  for all $\Phi \in  \V,$
\begin{align}\label{reduced continuous nonlinear ferro}
\widetilde{N}(\widetilde{\Psi}; \Phi): = N(\widetilde{\Psi}; \Phi)+2B_2({\Psi}_{\mathbf{g}}, \widetilde{\Psi},\Phi)+3B_3({\Psi}_{\mathbf{g}}, \widetilde{\Psi},\widetilde{\Psi},\Phi)+  3B_3({\Psi}_{\mathbf{g}},{\Psi}_{\mathbf{g}},\widetilde{\Psi}, \Phi)+ N({\Psi}_{\mathbf{g}}; \Phi)=0 .
\end{align}
\noindent The regular solutions ${\Psi}$ of \eqref{continuous nonlinear ferro} such that \eqref{H2 bound for minimizers} holds  are approximated. 
The solution ${\Psi}$ is regular  \cite{keller} implies that the Fr\'echet derivative $DN({\Psi}) \in \mathcal{L}(\V;\V^*)$ of $N(\cdot)$ at ${\Psi}$  is an isomorphism. That is, the inf-sup condition \cite{Ern} holds for the Fr\'echet derivative at  ${\Psi}$, 
\begin{align}\label{continuous inf sup}
0< \beta := \inf_{\substack{\Theta \in \V \\  \vertiii{\Theta}_1=1}} \sup_{\substack{\Phi\in \V \\  \vertiii{\Phi}_1=1}}\dual{DN({{\Psi}})\Theta, \Phi},
% =\inf_{\substack{\Theta \in \V \\  \vertiii{\Theta}_1=1}} \sup_{\substack{\Phi \in \V \\  \vertiii{\Phi}_1=1}}\dual{D\widetilde{N}(\widetilde{\Psi})\Theta, \Phi},
\end{align}
where $\dual{DN({{\Psi}})\Theta, \Phi}:=A(\Theta, \Phi)+B_1(\Theta, \Phi)+2B_2(\Psi,\Theta, \Phi)+3B_3(\Psi,\Psi,\Theta, \Phi)$. 
Note that 
\begin{align}\label{reformed bilinear form}
\dual{D\widetilde{N}(\widetilde{\Psi})\Theta, \Phi}:= \dual{DN(\widetilde{\Psi})\Theta, \Phi}+2B_2({\Psi}_{\mathbf{g}}, \Theta ,\Phi) + 3B_3({\Psi}_{\mathbf{g}}, {\Psi}_{\mathbf{g}}, \Theta, \Phi)+6B_3( {\Psi}_{\mathbf{g}},\widetilde{\Psi}, \Theta, \Phi).
\end{align}
 Algebraic manipulations with $\widetilde{\Psi}:={\Psi}-\Psi_{\mathbf{g}}$ leads to $\dual{DN({{\Psi}})\Theta, \Phi}=\dual{D\widetilde{N}(\widetilde{\Psi})\Theta, \Phi}$ and hence the inf-sup condition
 \begin{align}\label{reduced continuous inf sup}
 \displaystyle 0< \beta  =\inf_{\substack{\Theta \in \V \\  \vertiii{\Theta}_1=1}} \sup_{\substack{\Phi \in \V \\  \vertiii{\Phi}_1=1}}\dual{D\widetilde{N}(\widetilde{\Psi})\Theta, \Phi}
 \end{align} 
  holds.  Let $\mathbf{g}_h:={\rm I}_h\mathbf{g}$ be the Lagrange $P_1$ interpolation of $\mathbf{g}$ along $\partial \Omega$. 
 For $\widetilde{\Psi}_h:= {\Psi}_h-\textrm{I}_h {\Psi}_{\mathbf{g}} \in \V_h$,  \eqref{discrete nonlinear problem ferro} yields that $\widetilde{\Psi}_h$ solves 
\begin{align}\label{reduced discrete nonllinear problem}
\widetilde{N}(\widetilde{\Psi}_h; \Phi_h):= &N(\widetilde{\Psi}_h; \Phi_h)+2B_2(\textrm{I}_h {\Psi}_{\mathbf{g}},\widetilde{\Psi}_h, \Phi_h)+ 3B_3( \textrm{I}_h {\Psi}_{\mathbf{g}},\widetilde{\Psi}_h, \widetilde{\Psi}_h, \Phi_h)+ 3B_3(\textrm{I}_h {\Psi}_{\mathbf{g}}, \textrm{I}_h {\Psi}_{\mathbf{g}},\widetilde{\Psi}_h, \Phi_h)  \notag\\&+N(\textrm{I}_h {\Psi}_{\mathbf{g}}; \Phi_h) =0 \text{ for all } \Phi_h \in \V_h.
\end{align}
We first prove that $\widetilde{\Psi}_h$ approximates the solution $\widetilde{\Psi}$ of \eqref{reduced continuous nonlinear ferro} and this leads to the existence of the discrete solution ${\Psi}_h$ that approximates ${\Psi}$. The next lemma is crucial for the analysis.

\begin{lem}[Wellposedness of a linear system] \label{linear equation for discrete inf-sup ferro}For a given $\Theta_{h} \in \X_h$ with $\vertii{\Theta_{h}}=1$, and %$\Phi \in \V,$
	\begin{align*}
	B_L(\Theta_{h},\Phi):&= B_1(\Theta_{h},\Phi) +2 B_2(\widetilde{\Psi}, \Theta_h, \Phi)+2 B_2({\Psi}_{\mathbf{g}}, \Theta_h, \Phi)+3B_3(\widetilde{\Psi},\widetilde{\Psi},\Theta_{h},\Phi)+ 3B_3({\Psi}_{\mathbf{g}}, {\Psi}_{\mathbf{g}}, \Theta_h, \Phi)\\&\quad+6B_3({\Psi}_{\mathbf{g}},\widetilde{\Psi}, \Theta_h, \Phi),
	\end{align*} there exists $\boldsymbol{\zeta} \in \mathbf{H}^{2}(\Omega)\cap \V $ that solves the linear system 
	\begin{align}		
	& \quad	 \quad  \quad A(\boldsymbol\zeta,\Phi)=B_L(\Theta_{h},\Phi) \text{ for all }  \Phi \in \V,\label{linear equation1 }\\ &
		\text{ with }\quad	\vertiii{\boldsymbol \zeta}_{2}\lesssim \ell^{-1}(1+\vertiiii{\widetilde{\Psi}}_{2}+	\vertiiii{\widetilde{\Psi}}_{2}^2+\vertiiii{\mathbf{g}}_{{\frac{3}{2}}, \partial \Omega} +\vertiiii{\mathbf{g}}_{{\frac{3}{2}}, \partial \Omega}^2 ).\label{2.4.4.0 Ferro}
	\end{align}
	Here $'\lesssim'$ absorbs $\abs{c}$, the constants from elliptic regularity and Sobolev embedding results.
\end{lem}
\begin{proof}
For $\widetilde{\Psi}, {\Psi}_{\mathbf{g}} \in \mathbf{H}^2(\Omega)$ and $\Theta_h \in \X_h,$ Lemma \ref{boundedness_all ferro}$(i)$-$(iii)$  yields that $ B_3(\widetilde{\Psi},\widetilde{\Psi},\Theta_{h},\cdot),$  $B_3({\Psi}_{\mathbf{g}}, {\Psi}_{\mathbf{g}}, \Theta_h,  \cdot), $ $B_3({\Psi}_{\mathbf{g}},\widetilde{\Psi},  \Theta_h,  \cdot)$, $  B_2(\widetilde{\Psi}, \Theta_h, \cdot),$ $  B_2({\Psi}_{\mathbf{g}}, \Theta_h, \cdot),$  $B_1(\Theta_{h},\cdot) \in \mathbf{L}^2.$  This and an elliptic regularity result \cite{grisvard} implies that there exists a unique solution $\boldsymbol\zeta \in \V \cap \mathbf{H}^2(\Omega)$ of \eqref{linear equation1 } and a constant $C_{\rm reg}>0$ such that  
\begin{align*}
\vertiii{\boldsymbol\zeta}_2 \leq C_{\rm reg} \vertiii{B_L(\Theta_{h},\cdot)}_{\mathbf{L}^2}\lesssim  \ell^{-1}(1+\vertiiii{\widetilde{\Psi}}_{2}+	\vertiiii{\widetilde{\Psi}}_{2}^2+\vertiiii{\mathbf{g}}_{{\frac{3}{2}}, \partial \Omega} +\vertiiii{\mathbf{g}}_{{\frac{3}{2}}, \partial \Omega}^2 ).
\end{align*}
This completes the proof.
\end{proof}

\medskip

\noindent Now, for $\Theta_h,\Phi_h \in \V_h,$ the discrete inf-sup conditions are established for the bilinear form $\dual{D\widetilde{N}(\widetilde{\Psi})\Theta_h,\Phi_h}$ from  \eqref{reformed bilinear form}, and the perturbed form 
\begin{align} \label{definition of perturbed bilinear form}
\dual{D\widetilde{N}({\rm{I}}_h\widetilde{\Psi})\Theta_h, \Phi_h}= &  \dual{D{N}({\rm{I}}_h\widetilde{\Psi})\Theta_h, \Phi_h} +2 B_2({\rm{I}}_h{\Psi}_{\mathbf{g}}, \Theta_h, \Phi_h)+ 3B_3({\rm{I}}_h{\Psi}_{\mathbf{g}}, {\rm{I}}_h{\Psi}_{\mathbf{g}}, \Theta_h, \Phi_h) \notag \\&+6B_3( {\rm{I}}_h{\Psi}_{\mathbf{g}},{\rm{I}}_h\widetilde{\Psi}, \Theta_h, \Phi_h), 
\end{align}
where $\dual{D{N}({\rm{I}}_h\widetilde{\Psi})\Theta_h, \Phi_h}=A(\Theta_h, \Phi_h)+  B_1(\Theta_h, \Phi_h)+2 B_2({\rm{I}}_h\widetilde{\Psi}, \Theta_h, \Phi_h)+ 3B_3({\rm{I}}_h\widetilde{\Psi}, {\rm{I}}_h\widetilde{\Psi},\Theta_h, \Phi_h ). $
\begin{thm}[Discrete inf-sup conditions]\label{discrete inf sup}
Let ${\Psi}$ be a regular solution of \eqref{continuous nonlinear ferro} such that \eqref{H2 bound for minimizers} holds and $\widetilde{\Psi}$ solves the non-linear system
\eqref{reduced continuous nonlinear ferro}.  For a given fixed $\ell>0$,  a sufficiently small discretization parameter chosen such that $h=O(\ell)$,  the  discrete inf-sup conditions stated below hold:
	\begin{align*}
(i)\,\, 0< \frac{\beta}{2}\leq \inf_{\substack{\Theta_h \in \V_h \\  \vertii{\Theta_h}=1}} \sup_{\substack{\Phi_h \in \V_h \\  \vertii{\Phi_h}=1}}\dual{D\widetilde{N}(\widetilde{\Psi})\Theta_h, \Phi_h}, \,\, (ii)\,\,	0< \frac{\beta}{4}\leq \inf_{\substack{\Theta_h \in \V_h \\  \vertii{\Theta_h}=1}} \sup_{\substack{\Phi_h \in \V_h \\  \vertii{\Phi_h}=1}}\dual{D\widetilde{N}({\rm{I}}_h\widetilde{\Psi})\Theta_h, \Phi_h}.
	\end{align*}
\end{thm}
\begin{proof}[Proof of $(i)$]
	For $\Theta_h \in \V_h \subset \V$  
	with $\vertii{\Theta_h}=1$, 
	the continuous inf-sup condition in \eqref{reduced continuous inf sup} yields that there exists $\Phi\in \V$ with $\vertii{\Phi}=1$ such that 
	\begin{align*}%\label{discrete infsup 1}
	\beta \vertii{\Theta_h}& \leq \dual{D\widetilde{N}(\widetilde{\Psi})\Theta_h, \Phi} =A(\Theta_h, \Phi)+ B_L(\Theta_h, \Phi).
	\end{align*}
	The linear problem in \eqref{linear equation1 },  Lemma \ref{boundedness_all ferro}$(i)$   and a triangle inequality show
	\begin{align} \label{inf sup equation1 ferro}
	\beta=\beta\vertii{\Theta_h}	\leq A(\Theta_h +\boldsymbol\zeta, \Phi) \leq \vertiiii{\Theta_h +\boldsymbol\zeta}_1 \leq \vertiiii{\Theta_h +{\rm I}_h\boldsymbol\zeta}_1 +\vertiiii{{\rm I}_h\boldsymbol\zeta-\boldsymbol\zeta}_1.
	\end{align}
	The coercivity of $A(\cdot, \cdot)$ stated in Lemma \ref{boundedness_all ferro}$(i)$ yields that for $\Theta_h+ {\rm I}_h \boldsymbol\zeta \in \V_h \subset \V,$ there exists $\Phi_h \in \V_h$ with $\vertii{\Phi_h}=1$ such that $ \alpha_0	\vertii{\Theta_h+ {\rm I}_h \boldsymbol\zeta }   \leq A(\Theta_h+ {\rm I}_h \boldsymbol\zeta, \Phi_h) 	. $
	The definition of $\dual{D\widetilde{N}(\widetilde{\Psi})\cdot, \cdot}$ in  \eqref{reformed bilinear form}, \eqref{linear equation1 } and Lemma \ref{Properties of bilinear trilinear and quadrilinear forms}$(i)$ imply
	\begin{align*}%\label{inf sup equation2 ferro}
	&\vertii{\Theta_h+ {\rm I}_h \boldsymbol\zeta } 	\lesssim %A(\Theta_h, \Phi_h) + A({\rm I}_h \boldsymbol\zeta, \Phi_h)=
	\dual{D\widetilde{N}(\widetilde{\Psi})\Theta_h, \Phi_h}+A({\rm{I}}_h \boldsymbol\zeta -\boldsymbol\zeta, \Phi_h) \lesssim	\dual{D\widetilde{N}(\widetilde{\Psi})\Theta_h, \Phi_h}+h\vertiii{\boldsymbol\zeta}_2. 
%	\\&+ 2\big(B_2({\Psi}_{\mathbf{g}}-{\rm{I}}_h{\Psi}_{\mathbf{g}}, \Theta_h, \Phi_h)+ B_2(\widetilde{\Psi}-{\rm{I}}_h\widetilde{\Psi}, \Theta_h, \Phi_h) \big)\notag\\& \quad +3\big(B_3(\widetilde{\Psi}, \widetilde{\Psi}, \Theta_h, \Phi_h)-B_3({\rm{I}}_h\widetilde{\Psi}, {\rm{I}}_h\widetilde{\Psi}, \Theta_h, \Phi_h)\big)+ 3\big(B_3({\Psi}_{\mathbf{g}}, {\Psi}_{\mathbf{g}}, \Theta_h, \Phi_h)-B_3({\rm{I}}_h{\Psi}_{\mathbf{g}}, {\rm{I}}_h{\Psi}_{\mathbf{g}}, \Theta_h, \Phi_h)\big)\notag\\& \quad+ 6\big(B_3({\Psi}_{\mathbf{g}},\widetilde{\Psi},  \Theta_h, \Phi_h)-B_3( {\rm{I}}_h{\Psi}_{\mathbf{g}},{\rm{I}}_h\widetilde{\Psi}, \Theta_h, \Phi_h)\big)
%	=\dual{D\widetilde{N}({\rm I}_h\widetilde{\Psi})\Theta_h, \Phi_h}+\sum_{i=1}^{5}T_i.%+T_2+T_3+T_4 +T_5.
	\end{align*}
	This combined with \eqref{inf sup equation1 ferro},  Lemma \ref{Interpolation estimate} and \eqref{2.4.4.0 Ferro} leads to $$\beta \leq C_1(	\dual{D\widetilde{N}(\widetilde{\Psi})\Theta_h, \Phi_h}+  \ell^{-1}h),$$
	where the positive constant $C_1$ depends on $\abs{c}$,  $\vertiii{\widetilde{\Psi}}_2,\vertiii{\mathbf{g}}_{\frac{3}{2},\partial \Omega},$ $ C_{\rm reg},C_I, \alpha_0$ and the constants in Sobolev embedding results. For a sufficiently small choice of the discretization parameter  $h< h_0:= \frac{\beta \ell}{2C_1},$ the assertion holds. 
%	\begin{align*}%\label{Discrete inf-sup for best approximation}
%	\frac{\beta }{2} \vertiii{\Theta_h}_1\leq \dual{D\widetilde{N}(\widetilde{\Psi})\Theta_h, \Phi_h}.
%	\end{align*}
%	This completes the proof of the first part.\\
\end{proof}
\begin{proof}[Proof of $(ii)$] The definition of $\dual{D\widetilde{N}(\widetilde{\Psi})\Theta_h, \Phi_h}$ (resp.  $\dual{D\widetilde{N}({\rm I}_h\widetilde{\Psi})\Theta_h, \Phi_h}$) in \eqref{reduced continuous inf sup} (resp. \eqref{definition of perturbed bilinear form}) and a re-arrangement of terms allows
		\begin{align}\label{inf sup equation2 ferro2}
	&\dual{D\widetilde{N}({\rm I}_h\widetilde{\Psi})\Theta_h, \Phi_h}=	\dual{D\widetilde{N}(\widetilde{\Psi})\Theta_h, \Phi_h}
- 2\big(B_2({\Psi}_{\mathbf{g}}-{\rm{I}}_h{\Psi}_{\mathbf{g}}, \Theta_h, \Phi_h)+ B_2(\widetilde{\Psi}-{\rm{I}}_h\widetilde{\Psi}, \Theta_h, \Phi_h) \big)\notag\\& \quad -3\big(B_3(\widetilde{\Psi}, \widetilde{\Psi}, \Theta_h, \Phi_h)-B_3({\rm{I}}_h\widetilde{\Psi}, {\rm{I}}_h\widetilde{\Psi}, \Theta_h, \Phi_h)\big)-3\big(B_3({\Psi}_{\mathbf{g}}, {\Psi}_{\mathbf{g}}, \Theta_h, \Phi_h)-B_3({\rm{I}}_h{\Psi}_{\mathbf{g}}, {\rm{I}}_h{\Psi}_{\mathbf{g}}, \Theta_h, \Phi_h)\big)\notag\\& \quad- 6\big(B_3({\Psi}_{\mathbf{g}},\widetilde{\Psi},  \Theta_h, \Phi_h)-B_3( {\rm{I}}_h{\Psi}_{\mathbf{g}},{\rm{I}}_h\widetilde{\Psi}, \Theta_h, \Phi_h)\big)
		=\dual{D\widetilde{N}(\widetilde{\Psi})\Theta_h, \Phi_h}-\sum_{i=1}^{4}T_i.%+T_2+T_3+T_4 +T_5.
	\end{align}
		Lemma \ref{Properties of bilinear trilinear and quadrilinear forms}$(ii)$  with $\boldsymbol{\eta}:={\Psi}_{\mathbf{g}} $ (resp. $\boldsymbol{\eta}:=\widetilde{\Psi} $) for the first term (resp. second term) in $T_1$,  and  \eqref{bound for Psig}  implies  
	\begin{align*}
	&\frac{1}{2}T_1:=	B_2({\Psi}_{\mathbf{g}}- {\rm{I}}_h{\Psi}_{\mathbf{g}}, \Theta_h, \Phi_h) +B_2(\widetilde{\Psi}-{\rm{I}}_h\widetilde{\Psi}, \Theta_h, \Phi_h)  \lesssim  \ell^{-1}h\vertiii{\Theta_h}_1\vertiiii{ \Phi_h}_1.
	\end{align*}
	The term $T_2$ (resp. $T_3$) is estimated using Lemma \ref{Properties of bilinear trilinear and quadrilinear forms}$(iii)$ for $\boldsymbol{\xi}:=\boldsymbol{\eta}:=\widetilde{\Psi} $ (resp. $\boldsymbol{\xi}:=\boldsymbol{\eta}:={\Psi}_{\mathbf{g}}$) and \eqref{bound for Psig} below.
	\begin{align*}
	& \frac{1}{3}T_2:=	B_{3}(\widetilde{\Psi}, \widetilde{\Psi}, \Theta_h, \Phi_h)-B_{3}({\rm I}_h\widetilde{\Psi}, {\rm I}_h\widetilde{\Psi}, \Theta_h, \Phi_h)
	%	= 	B_{3}(\widetilde{\Psi}-{\rm I}_h\widetilde{\Psi}, \widetilde{\Psi}, \Theta_h, \Phi_h)+	B_{3}({\rm I}_h\widetilde{\Psi}, \widetilde{\Psi}-{\rm I}_h\widetilde{\Psi}, \Theta_h, \Phi_h) 
	\lesssim 	\ell^{-1}h\vertiii{\Theta_h}_1\vertiiii{ \Phi_h}_1.
	\\&
(\text{resp. }	\frac{1}{3}T_3:=B_3({\Psi}_{\mathbf{g}}, {\Psi}_{\mathbf{g}}, \Theta_h, \Phi_h)-B_3({\rm{I}}_h{\Psi}_{\mathbf{g}}, {\rm{I}}_h{\Psi}_{\mathbf{g}}, \Theta_h, \Phi_h)
	%	=	B_{3}({\Psi}_{\mathbf{g}}-{\rm I}_h{\Psi}_{\mathbf{g}}, {\Psi}_{\mathbf{g}}, \Theta_h, \Phi_h)+	B_{3}({\rm I}_h{\Psi}_{\mathbf{g}}, {\Psi}_{\mathbf{g}}-{\rm I}_h{\Psi}_{\mathbf{g}}, \Theta_h, \Phi_h)
	\lesssim	\ell^{-1}h\vertiii{\Theta_h}_1\vertiiii{ \Phi_h}_1.)
	\end{align*}
	Apply	Lemma \ref{Properties of bilinear trilinear and quadrilinear forms}$(iii)$ with  $\boldsymbol{\xi}:= {\Psi}_{\mathbf{g}},$ $ \boldsymbol{\eta}:=\widetilde{\Psi}$ and utilize \eqref{bound for Psig} to obtain
	\begin{align*}
	\frac{1}{6}T_4:=	 	B_3({\Psi}_{\mathbf{g}},\widetilde{\Psi},  \Theta_h, \Phi_h)-B_3( {\rm{I}}_h{\Psi}_{\mathbf{g}},{\rm{I}}_h\widetilde{\Psi}, \Theta_h, \Phi_h) \lesssim  \ell^{-1}h\vertiii{\Theta_h}_1\vertiiii{ \Phi_h}_1.
	\end{align*}
	%	The linearity of $B_3(\cdot,\cdot,\cdot,\cdot)$ in first and second variables, a re-arrangement of terms along with Lemma \ref{Interpolation estimate} and  $\vertiii{{\Psi}_{\mathbf{g}}}_{2}\leq \vertiii{\mathbf{g}}_{\frac{3}{2}}$ yields
	%$$	B_3(\widetilde{\Psi}, {\Psi}_{\mathbf{g}}, \Theta_h, \Phi_h)-B_3({\rm{I}}_h\widetilde{\Psi}, {\rm{I}}_h{\Psi}_{\mathbf{g}}, \Theta_h, \Phi_h) = B_3(\widetilde{\Psi}, {\Psi}_{\mathbf{g}}-{\rm{I}}_h{\Psi}_{\mathbf{g}}, \Theta_h, \Phi_h)+B_3(\widetilde{\Psi}-{\rm{I}}_h\widetilde{\Psi}, {\rm{I}}_h{\Psi}_{\mathbf{g}}, \Theta_h, \Phi_h) \lesssim  \ell^{-1}h.$$
	The above displayed estimates for $T_1, \cdots, T_4$ substituted in  \eqref{inf sup equation2 ferro2} and the discrete inf-sup condition in $(i)$ leads to $$\sup_{\substack{\Phi_h \in \V_h \\  \vertii{\Phi_h}=1}}\dual{D\widetilde{N}({\rm I}_h\widetilde{\Psi})\Theta_h, \Phi_h}\geq 	\sup_{\substack{\Phi_h \in \V_h \\  \vertii{\Phi_h}=1}}\dual{D\widetilde{N}(\widetilde{\Psi})\Theta_h, \Phi_h} -C_2  \ell^{-1}h\vertiii{\Theta_h}_1\geq  \big(\frac{\beta}{2}-  C_2 \ell^{-1}h \big)\vertiii{\Theta_h}_1,$$ where the positive constant $C_2$ depends on $\abs{c}$, $\vertiii{\widetilde{\Psi}}_2,\vertiii{\mathbf{g}}_{\frac{3}{2},\partial \Omega},$ $C_I, $ and the constants in Sobolev embedding results.
	For a sufficiently small choice of the discretization parameter $h< h_{2}:=\min(h_0,h_1)$ with $h_1<\frac{\beta\ell}{4C_2}$, the proof follows.
\end{proof}
\begin{rem}[A discrete inf-sup condition]\label{discrete inf-sup condition for psi}
	The discrete inf-sup condition established in Theorem \ref{discrete inf sup}$(i)$ is equivalent to $$\displaystyle 0< \frac{\beta}{2}\leq \inf_{\substack{\Theta_h \in \V_h \\  \vertii{\Theta_h}=1}} \sup_{\substack{\Phi_h \in \V_h \\  \vertii{\Phi_h}=1}}\dual{D{N}({\Psi})\Theta_h, \Phi_h}$$ and follows from the identity $\dual{D\widetilde{N}(\widetilde{\Psi})\Theta_h, \Phi_h}=\dual{D{N}({\Psi})\Theta_h, \Phi_h}$ for all $\Theta_h, \Phi_h \in \V_h.$ \qed
\end{rem}
\subsection{Proof of main results}  \label{A priori estimates}
The proofs of the results stated in Section \ref{Main results} are presented here.  The next theorem establishes the existence and uniqueness of the discrete solution that approximates the solution $\widetilde{\Psi}$ of \eqref{reduced continuous nonlinear ferro} and is an application of Brouwer's fixed point theorem.  This  result is required to prove Theorem \ref{energy  norm error estimate ferro}.
%The existence and uniqueness of discrete solution that approximates the solution ${\Psi}$ of \eqref{continuous nonlinear ferro}  is a straightforward application of Theorem \ref{reduced energy norm error estimate} and triangle inequality.
\begin{thm}[Energy norm error estimate to approximate $\widetilde{\Psi}$]\label{reduced energy norm error estimate}
	Let ${\Psi}$ be a regular solution of \eqref{continuous nonlinear ferro} such that \eqref{H2 bound for minimizers} holds and $\widetilde{\Psi}$ solves the non-linear system
	\eqref{reduced continuous nonlinear ferro}.  
	 For
	a given fixed $\ell>0$,   a sufficiently small discretization parameter chosen as $h=O(\ell^{1+\varsigma})$ with $\varsigma>0,$  there exists a unique solution ${\widetilde{\Psi}}_{h}$ to the discrete problem \eqref{reduced discrete nonllinear problem}  that approximates ${\widetilde{\Psi}}$ such that	$$	\vertii{{\widetilde{\Psi}}-{\widetilde{\Psi}}_h} \lesssim  h,$$
	where the constant suppressed in  $'\lesssim'$ is independent of $h$ and $\ell$.
\end{thm}
\begin{proof}
	The proof is divided into four steps.\\
\noindent{\it 	Step 1 (Non-linear map)}. 
%	The key idea of the proof is to construct a non-linear map $\mu_{h} $ such that any fixed point of $\mu_{h}$ is a solution of the discrete non-linear problem \eqref{reduced discrete nonllinear problem} and then establish the existence of a unique fixed point
%	of $\mu_{h}$ using Brouwer’s fixed point theorem.
	 For $\Theta_{h} \in \V_h$, define the non-linear map $\mu_{h}:\V_h \rightarrow \V_h$ by
	\begin{align}\label{mu nonlinear map ferro}
	&\dual{D\widetilde{N}({{\rm I}_{h}\widetilde{\Psi}}) \mu_{h}(\Theta_{h}),\Phi_{h} }:
	= 3B_3(\textrm{I}_{h}\widetilde{\Psi}, \textrm{I}_{h}\widetilde{\Psi},\Theta_{h},\Phi_{h})+ 6B_3(\textrm{I}_{h}{\Psi}_{\mathbf{g}},\textrm{I}_{h}\widetilde{\Psi}, \Theta_{h},\Phi_{h}) - B_3(\Theta_{h},\Theta_{h}, \Theta_{h},\Phi_{h})\notag\\& \quad-3B_3(\textrm{I}_{h}{\Psi}_{\mathbf{g}},\Theta_{h},\Theta_{h},\Phi_{h})+2 B_2(\textrm{I}_{h}\widetilde{\Psi}, \Theta_{h},\Phi_{h}) -  B_2(\Theta_{h}, \Theta_{h},\Phi_{h}) - N(\textrm{I}_{h}{\Psi}_{\mathbf{g}};\Phi_{h}) \text{ for all }\Phi_{h} \in \V_h.
	\end{align}
The map  $\mu_{h}$ is well-defined follows from Theorem \ref{discrete inf sup}$(ii)$ and any fixed point of $\mu_{h}$ is a solution of the discrete non-linear problem \eqref{reduced discrete nonllinear problem}.
% The proof is divided into three steps. 
 %Next we prove that for sufficiently small value of the discretization parameter $h$, $\mu_{h}$ maps the ball $\mathbb{B}_R(\Ihpsi)$ to itself and also it is a contraction map.	
	\medskip 
	
\noindent{\it 	Step 2 (Mapping of ball to ball)}. Define $\mathbb{B}_R(\textrm{I}_h\widetilde{\Psi}):= \{\Phi_h \in \V_h: \vertii{\textrm{I}_h\widetilde{\Psi}-\Phi_{h}} \leq R\}.$ This step establishes that  there exists a positive constant $R(h)$ such that $\Theta_{h} \in \mathbb{B}_{R(h)}({\rm I}_{h}\widetilde{\Psi})$ implies $\mu_{h}(\Theta_{h}) \in \mathbb{B}_{R(h)}({\rm I}_{h}\widetilde{\Psi})$ for all   $\Theta_{h} \in \V_{h} .$
%$$	\vertii{\Theta_{h} - {\rm I}_{h}\widetilde{\Psi} } \leq R(h) \implies \vertii{\mu_{h}(\Theta_{h}) - {\rm I}_{h}\widetilde{\Psi}} \leq R(h) \text{ for all } \Theta_{h} \in \V_{h} .$$	
\medskip

\noindent	%Let $\Theta_{h} \in \mathbb{B}_{R(h)}({\rm I}_{h}\widetilde{\Psi}).$  
	Theorem \ref{discrete inf sup}$(ii)$ and the linearity of $ \dual{D\widetilde{N}({{\rm I}_{h}\widetilde{\Psi}}) \cdot ,\cdot }$ yields that there exists a $\Phi_{h}\in \V_{h}$ with $\vertii{\Phi_{h}}=1$ such that  
$$	\frac{\beta}{4}\vertii{{{\rm I}_{h}\widetilde{\Psi}}- \mu_{h}(\Theta_{h})} \leq
 %\dual{D\widetilde{N}({{\rm I}_{h}\widetilde{\Psi}}) ({{\rm I}_{h}\widetilde{\Psi}}- \mu_{h}(\Theta_{h})),\Phi_{h} }= 
 \dual{D\widetilde{N}({{\rm I}_{h}\widetilde{\Psi}}) {{\rm I}_{h}\widetilde{\Psi}},\Phi_{h} }-\dual{D\widetilde{N}({{\rm I}_{h}\widetilde{\Psi}})  \mu_{h}(\Theta_{h}),\Phi_{h} }.$$
The definition of the linearized operator $\dual{D\widetilde{N}({{\rm I}_{h}\widetilde{\Psi}}),\cdot}$ in \eqref{definition of perturbed bilinear form}, the non-linear map in \eqref{mu nonlinear map ferro}, the consistency $\widetilde{N}(\widetilde{\Psi}; \Phi_h)	=0$,   and a  re-arrangement of the terms leads to
	\begin{align}\label{ball to ball equation1 ferro1}
	&	\vertii{{{\rm I}_{h}\widetilde{\Psi}}- \mu_{h}(\Theta_{h})} \lesssim \big(A({\rm I}_{h}\widetilde{\Psi}-\widetilde{\Psi}, \Phi_h)+ B_1({\rm I}_{h}\widetilde{\Psi}-\widetilde{\Psi}, \Phi_h) \big)+ \big(B_2({\rm{I}}_h\widetilde{\Psi},{\rm{I}}_h\widetilde{\Psi}, \Phi_{h} ) - B_2(\widetilde{\Psi},\widetilde{\Psi},\Phi_{h})\big)\notag\\&\quad+B_2(\textrm{I}_{h}\widetilde{\Psi}-\bTheta_{h},\textrm{I}_{h}\widetilde{\Psi}-\bTheta_{h} , \Phi_{h} ) +\big(B_3({\rm{I}}_h\widetilde{\Psi}, {\rm{I}}_h\widetilde{\Psi},{\rm I}_{h}\widetilde{\Psi}, \Phi_h )- B_3(\widetilde{\Psi}, \widetilde{\Psi},\widetilde{\Psi}, \Phi_h )\big)  +\big(	2B_3({\rm{I}}_h\widetilde{\Psi}, {\rm{I}}_h\widetilde{\Psi},{\rm I}_{h}\widetilde{\Psi}, \Phi_h ) \notag\\&\quad- 3B_3({\rm{I}}_h\widetilde{\Psi}, {\rm{I}}_h\widetilde{\Psi},\Theta_{h}, \Phi_h )+B_3(\Theta_{h}, \Theta_{h},\Theta_{h}, \Phi_h )\big)
	+2\big( B_2( {\rm{I}}_h{\Psi}_{\mathbf{g}},\textrm{I}_{h}\widetilde{\Psi}, \Phi_h )- B_2({\Psi}_{\mathbf{g}}, \widetilde{\Psi}, \Phi_h )\big)
	 \notag \\&\quad+3\big(B_3( {\rm{I}}_h{\Psi}_{\mathbf{g}},{\rm{I}}_h\widetilde{\Psi}, {\rm{I}}_h\widetilde{\Psi}, \Phi_h)-B_3({\Psi}_{\mathbf{g}}, \widetilde{\Psi},\widetilde{\Psi},\Phi_h) \big)+3\big( B_3({\rm{I}}_h{\Psi}_{\mathbf{g}}, {\rm{I}}_h{\Psi}_{\mathbf{g}}, {\rm{I}}_h\widetilde{\Psi}, \Phi_h)-B_3({\Psi}_{\mathbf{g}},{\Psi}_{\mathbf{g}},\widetilde{\Psi}, \Phi_h)\big)\notag \\&\quad+3B_3(\textrm{I}_{h}{\Psi}_{\mathbf{g}},{\rm{I}}_h\widetilde{\Psi}-\Theta_{h}, {\rm{I}}_h\widetilde{\Psi}-\Theta_{h}, \Phi_h )
	 +\big( N(\textrm{I}_{h}{\Psi}_{\mathbf{g}};\Phi_{h}) - N({\Psi}_{\mathbf{g}};\Phi_{h}) \big)
	 =: \sum_{i=1}^{10}T_i.%+T_2+T_3+T_4+T_5+T_6+T_7+T_8+T_9+T_{10}.
	\end{align}
	Here the term $T_3$ (resp. $T_9$) is a re-grouping of terms as
	\begin{align*}
&	B_2( {\rm{I}}_h\widetilde{\Psi}, {\rm{I}}_h\widetilde{\Psi}, \Phi_h ) -2B_2({\rm{I}}_h\widetilde{\Psi}, \Theta_{h}, \Phi_h ) +B_2(\Theta_h,\Theta_h, \Phi_h )\\&= B_2( {\rm{I}}_h\widetilde{\Psi} , {\rm{I}}_h\widetilde{\Psi}-\Theta_h, \Phi_h )-B_2({\rm{I}}_h\widetilde{\Psi}-\Theta_h,\Theta_h, \Phi_h )=B_2({\rm{I}}_h\widetilde{\Psi}-\Theta_h,{\rm{I}}_h\widetilde{\Psi}-\Theta_h, \Phi_h )\\&
(\text{resp. } B_3(\textrm{I}_{h}\Psi_{\mathbf{g}}, {\rm{I}}_h\widetilde{\Psi}, {\rm{I}}_h\widetilde{\Psi}, \Phi_h ) -2B_3({\rm{I}}_h\Psi_{\mathbf{g}},{\rm{I}}_h\widetilde{\Psi}, \Theta_{h}, \Phi_h ) +B_3(\textrm{I}_{h}\Psi_{\mathbf{g}},\Theta_h,\Theta_h, \Phi_h )\\&=B_3(\textrm{I}_{h}{\Psi}_{\mathbf{g}},{\rm{I}}_h\widetilde{\Psi},{\rm{I}}_h\widetilde{\Psi}-\Theta_{h},  \Phi_h )-B_3(\textrm{I}_{h}{\Psi}_{\mathbf{g}},{\rm{I}}_h\widetilde{\Psi}-\Theta_{h}, \Theta_{h}, \Phi_h )=B_3(\textrm{I}_{h}{\Psi}_{\mathbf{g}},{\rm{I}}_h\widetilde{\Psi}-\Theta_{h}, {\rm{I}}_h\widetilde{\Psi}-\Theta_{h}, \Phi_h ))
	\end{align*}
% (resp. $3(B(\textrm{I}_{h}\Psi_{\mathbf{g}}, {\rm{I}}_h\widetilde{\Psi}, {\rm{I}}_h\widetilde{\Psi}, \Phi_h ) -2B({\rm{I}}_h\Psi_{\mathbf{g}},{\rm{I}}_h\widetilde{\Psi}, \Theta_{h}, \Phi_h ) +B(\textrm{I}_{h}\Psi_{\mathbf{g}},\Theta_h,\Theta_h, \Phi_h ))=B_3(\textrm{I}_{h}{\Psi}_{\mathbf{g}},{\rm{I}}_h\widetilde{\Psi}-\Theta_{h}, {\rm{I}}_h\widetilde{\Psi}-\Theta_{h}, \Phi_h )$) 
 and  achieved by utilizing the linearity and symmetry of $B_2(\cdot,\cdot,\cdot)$ (resp. $B_3(\cdot,\cdot,\cdot,\cdot)$) in first two variables (resp. second and third variables).
%The definition of $\dual{D{N}({\rm{I}}_h\widetilde{\Psi})\cdot , \cdot}$, \eqref{continuous nonlinear ferro} and re-arrangement of the terms in $T_1$ leads to 
%	\begin{align*}%\label{ball to ball equation1 ferro2}
%T_1:=&\big(A({\rm I}_{h}\widetilde{\Psi}-\widetilde{\Psi}, \Phi_h)+ B_1({\rm I}_{h}\widetilde{\Psi}-\widetilde{\Psi}, \Phi_h) \big)+ \big(B_2({\rm{I}}_h\widetilde{\Psi},{\rm{I}}_h\widetilde{\Psi}, \Phi_{h} ) + B_2(\widetilde{\Psi},\widetilde{\Psi},\Phi_{h})\big)\notag\\&+B_2(\textrm{I}_{h}\widetilde{\Psi}-\bTheta_{h},\textrm{I}_{h}\widetilde{\Psi}-\bTheta_{h} , \Phi_{h} ) +\big(B_3({\rm{I}}_h\widetilde{\Psi}, {\rm{I}}_h\widetilde{\Psi},{\rm I}_{h}\widetilde{\Psi}, \Phi_h )- B_3(\widetilde{\Psi}, \widetilde{\Psi},\widetilde{\Psi}, \Phi_h )\big) \notag\\& +\big(	2B_3({\rm{I}}_h\widetilde{\Psi}, {\rm{I}}_h\widetilde{\Psi},{\rm I}_{h}\widetilde{\Psi}, \Phi_h ) - 3B_3({\rm{I}}_h\widetilde{\Psi}, {\rm{I}}_h\widetilde{\Psi},\Theta_{h}, \Phi_h )+B_3(\Theta_{h}, \Theta_{h},\Theta_{h}, \Phi_h )\big).
%\end{align*}
Lemma \ref{Properties of bilinear trilinear and quadrilinear forms}$(i)$ (resp. $(ii)$) for  $\boldsymbol{\eta}:=\widetilde{\Psi} $ (resp. $\boldsymbol{\xi}:=\boldsymbol{\eta}:=\widetilde{\Psi} $) shows
%	A use of Lemmas \ref{boundedness_all ferro}$(i)$, \ref{Interpolation estimate} and $\vertiiii{\Phi_{h}}_1=1$ yields
	\begin{align*}%\label{ball to ball equation2 ferro}
&T_1:=A({\rm I}_{h}\widetilde{\Psi}-\widetilde{\Psi}, \Phi_h)+ B_1({\rm I}_{h}\widetilde{\Psi}-\widetilde{\Psi}, \Phi_h) \lesssim h+ \ell^{-1}h^2.\\&
(\text{resp. }T_2:=B_2({\rm{I}}_h\widetilde{\Psi},{\rm{I}}_h\widetilde{\Psi}, \Phi_{h} ) - B_2(\widetilde{\Psi},\widetilde{\Psi},\Phi_{h})\lesssim \ell^{-1}h^2.)
	\end{align*}
	For $\e:= {\rm I}_{h}\widetilde{\Psi}-\Theta_{h},$ Lemma   \ref{boundedness_all ferro}$(ii)$ with $\vertiiii{\Phi_{h}}_1=1$  shows
	$$T_3:=B_2(\textrm{I}_{h}\widetilde{\Psi}-\bTheta_{h},\textrm{I}_{h}\widetilde{\Psi}-\bTheta_{h} , \Phi_{h} )  \lesssim \ell^{-1} \vertii{\textrm{I}_{h}\widetilde{\Psi}-\bTheta_{h}}^2=\ell^{-1}\vertii{{\e}}^2 .$$
		The estimate for $T_4$ (resp. $T_5$) follows from the second (resp. third) inequality of Lemma \ref{Properties of bilinear trilinear and quadrilinear forms}$(iii)$ for $\boldsymbol{\xi}:=\boldsymbol{\eta}:=\widetilde{\Psi} $ (resp. $\boldsymbol{\eta}:= {\rm{I}}_h\widetilde{\Psi}$), and Lemma \ref{Interpolation estimate}. 
	\begin{align*}
&	T_4:= B_3({\rm{I}}_h\widetilde{\Psi}, {\rm{I}}_h\widetilde{\Psi},{\rm I}_{h}\widetilde{\Psi}, \Phi_h )- B_3(\widetilde{\Psi}, \widetilde{\Psi},\widetilde{\Psi}, \Phi_h ) \lesssim \ell^{-1}h^2.\\&
(\text{resp. } T_5:= 	2B_3({\rm{I}}_h\widetilde{\Psi}, {\rm{I}}_h\widetilde{\Psi},{\rm I}_{h}\widetilde{\Psi}, \Phi_h ) - 3B_3({\rm{I}}_h\widetilde{\Psi}, {\rm{I}}_h\widetilde{\Psi},\Theta_{h}, \Phi_h )+B_3(\Theta_{h}, \Theta_{h},\Theta_{h}, \Phi_h )\lesssim \ell^{-1}\vertii{\e}^2(\vertii{\e}+1).)
	\end{align*}
% The estimates of the following two expressions follows technique used for	$T_3,T_4$ in \cite[Theorem 5.1]{DGFEM}.
%	\begin{align*}%\label{ball to ball equation3 ferro}
%	&B_3({\rm{I}}_h\widetilde{\Psi}, {\rm{I}}_h\widetilde{\Psi},{\rm I}_{h}\widetilde{\Psi}, \Phi_h ) - B_3(\widetilde{\Psi}, \widetilde{\Psi},\widetilde{\Psi}, \Phi_h ) \lesssim  \ell^{-1}h^2,\\&
%		2B_3({\rm{I}}_h\widetilde{\Psi}, {\rm{I}}_h\widetilde{\Psi},{\rm I}_{h}\widetilde{\Psi}, \Phi_h ) - 3B_3({\rm{I}}_h\widetilde{\Psi}, {\rm{I}}_h\widetilde{\Psi},\Theta_{h}, \Phi_h )+B_3(\Theta_{h}, \Theta_{h},\Theta_{h}, \Phi_h ) \lesssim \ell^{-1}\vertii{\e}^2(\vertii{\e}+\vertiii{ {\rm{I}}_h\widetilde{\Psi}}_1).
%	\end{align*}
%A combination of the above displayed five estimates and \eqref{ball to ball equation1 ferro2} leads to
%Therefore, we obtain 
%$$T_1\lesssim h+ \ell^{-1}h^2+ \ell^{-1}\vertii{\e}^2(\vertii{\e}+1).$$ 
Lemma \ref{Properties of bilinear trilinear and quadrilinear forms}$(ii)$ (resp. $(iii)$) with $\boldsymbol{\xi}:={\Psi}_{\mathbf{g}}$, $\boldsymbol{\eta}:=\widetilde{\Psi}$ and \eqref{bound for Psig} shows
\begin{align*}%\label{ball to ball equation10 ferro}
&\frac{1}{2}T_6:= B_2( {\rm{I}}_h{\Psi}_{\mathbf{g}},\textrm{I}_{h}\widetilde{\Psi}, \Phi_h )- B_2({\Psi}_{\mathbf{g}},  \widetilde{\Psi},\Phi_h )
%&= B_2(\Ihpsi_{\mathbf{g}}-\Psi_{\mathbf{g}} , \textrm{I}_{h}\widetilde{\Psi},\Phi_{h} ) + B_2(\Psi_{\mathbf{g}} , \textrm{I}_{h}\widetilde{\Psi}-\widetilde{\Psi},\Phi_{h} )
% \\& \leq   \vertiii{\Ihpsi_{\mathbf{g}}-\Psi_{\mathbf{g}}}_0\vertiii{\textrm{I}_{h}\widetilde{\Psi}}_{\infty} +\vertiii{\Psi_{\mathbf{g}}}_2\vertiii{\textrm{I}_{h}\widetilde{\Psi}-\widetilde{\Psi}}_0
 \lesssim  \ell^{-1}h^2 .
%\end{align*}
%	The linearity of $B_3$ in first three variables,
%	re-arrangement of terms and  Lemmas \ref{boundedness_all ferro}$(ii)$ and  \ref{Interpolation estimate}   yields 
%	\begin{align*}%\label{ball to ball equation5 ferro}
\\&	(\text{resp. }\frac{1}{3}T_7:=	B_3(\textrm{I}_{h}{\Psi}_{\mathbf{g}}, {\rm{I}}_h\widetilde{\Psi}, {\rm{I}}_h\widetilde{\Psi},\Phi_h )- B_3({\Psi}_{\mathbf{g}},\widetilde{\Psi}, \widetilde{\Psi}, \Phi_h ) 
%	\notag \\&= 3(B_3(\textrm{I}_{h}{\Psi}_{\mathbf{g}}-{\Psi}_{\mathbf{g}},{\rm{I}}_h\widetilde{\Psi}, {\rm{I}}_h\widetilde{\Psi}, \Phi_h ) +B_3({\Psi}_{\mathbf{g}},{\rm{I}}_h\widetilde{\Psi}-\widetilde{\Psi},{\rm{I}}_h\widetilde{\Psi} , \Phi_{h} ) + B_3({\Psi}_{\mathbf{g}},\widetilde{\Psi},{\rm{I}}_h\widetilde{\Psi}-\widetilde{\Psi},\Phi_{h}) \\&
%	\lesssim  \vertiii{\textrm{I}_{h}\widetilde{\Psi}}_{\infty}^2\vertiii{\Ihpsi_{\mathbf{g}}-\Psi_{\mathbf{g}}}_0+\vertiii{\textrm{I}_{h}\widetilde{\Psi}}_{\infty}\vertiii{{\Psi}_{\mathbf{g}}}_{2}\vertiii{\textrm{I}_{h}\widetilde{\Psi}-\widetilde{\Psi}}_0+\vertiii{\widetilde{\Psi}}_{2}\vertiii{{\Psi}_{\mathbf{g}}}_{2} \vertiii{\textrm{I}_{h}\widetilde{\Psi}-\widetilde{\Psi}}_0
	\lesssim  \ell^{-1}h^2.)
	\end{align*}
	Apply	Lemma \ref{Properties of bilinear trilinear and quadrilinear forms}$(iii)$ with $\boldsymbol{\xi}:=\widetilde{\Psi}$, $\boldsymbol{\eta}:={\Psi}_{\mathbf{g}}$ and \eqref{bound for Psig} to obtain
	\begin{align*}%\label{ball to ball equation7 ferro}
	\frac{1}{3}T_8:&=B_3({\rm{I}}_h{\Psi}_{\mathbf{g}}, {\rm{I}}_h{\Psi}_{\mathbf{g}},\textrm{I}_{h}\widetilde{\Psi}, \Phi_h )- B_3({\Psi}_{\mathbf{g}}, {\Psi}_{\mathbf{g}},\widetilde{\Psi}, \Phi_h )
	 \lesssim  \ell^{-1}h^2.
	\end{align*}
 Lemmas \ref{boundedness_all ferro}(iii), \ref{Interpolation estimate} and \eqref{bound for Psig} leads to
	\begin{align*}%\label{ball to ball equation6 ferro}
	\frac{1}{3}T_9:
	%&=	3(B_3(\textrm{I}_{h}{\Psi}_{\mathbf{g}}, {\rm{I}}_h\widetilde{\Psi}, {\rm{I}}_h\widetilde{\Psi},\Phi_h ) -2B_3({\rm{I}}_h{\Psi}_{\mathbf{g}},{\rm{I}}_h\widetilde{\Psi}, \Theta_{h}, \Phi_h ) +B_3(\textrm{I}_{h}{\Psi}_{\mathbf{g}}, \Theta_h,\Theta_h,\Phi_h ))\notag \\&
	=B_3(\textrm{I}_{h}{\Psi}_{\mathbf{g}},{\rm{I}}_h\widetilde{\Psi}-\Theta_{h},{\rm{I}}_h\widetilde{\Psi}-\Theta_{h}, \Phi_h )\lesssim  \ell^{-1}\vertii{\e}^2.
	\end{align*}
	The definition of $N(\cdot;\cdot)$ in \eqref{continuous nonlinear ferro}, a re-arrangement of terms, Lemma \ref{Properties of bilinear trilinear and quadrilinear forms}$(i)$-$(iii)$ with $\boldsymbol{\xi}:=\boldsymbol{\eta}:={\Psi}_{\mathbf{g}}$, and \eqref{bound for Psig}  yields  
	%four terms, which are similar to some of the re-grouped terms obtained in $T_1.$ The proof follows similar line for the terms corresponding to ${\Psi}_{\mathbf{g}}$ here and  apply Lemmas \ref{boundedness_all ferro}, \ref{Interpolation estimate} to obtain
	\begin{align*}%\label{ball to ball equation11 ferro}
		T_{10}&:=	 N(\textrm{I}_{h}{\Psi}_{\mathbf{g}};\Phi_{h}) - N({\Psi}_{\mathbf{g}};\Phi_{h})
	=\big(A(\textrm{I}_{h}{\Psi}_{\mathbf{g}}-{\Psi}_{\mathbf{g}},\Phi_{h})+ B_1(\textrm{I}_{h}{\Psi}_{\mathbf{g}}-{\Psi}_{\mathbf{g}},\Phi_{h})\big)+\big(B_2(\textrm{I}_{h}{\Psi}_{\mathbf{g}},\textrm{I}_{h}{\Psi}_{\mathbf{g}},\Phi_{h})\notag\\&\quad-B_2({\Psi}_{\mathbf{g}},{\Psi}_{\mathbf{g}},\Phi_{h})\big)+\big(B_3(\textrm{I}_{h}{\Psi}_{\mathbf{g}},\textrm{I}_{h}{\Psi}_{\mathbf{g}},\textrm{I}_{h}{\Psi}_{\mathbf{g}},\Phi_{h})-B_3({\Psi}_{\mathbf{g}},{\Psi}_{\mathbf{g}},{\Psi}_{\mathbf{g}},\Phi_{h})\big)  %\notag\\&
	\lesssim h+\ell^{-1}h^2.
	\end{align*}
	Substitute the estimates for $T_1, \ldots, T_{10}$ in \eqref{ball to ball equation1 ferro1} and utilize  $\vertiii{\e}_1 \leq R(h)$ to obtain %A use of discrete inf sup condition in Lemma \ref{discrete inf sup} yields that there exists a $\Phi_{h}\in \V_{h}$ with $\vertii{\Phi_{h}}=1$ such that  
	\begin{align*}%\label{ball to ball equation12 ferro}
	\vertii{{{\rm I}_{h}\widetilde{\Psi}}- \mu_{h}(\Theta_{h})} \leq C_3( h+ \ell^{-1}h^2+ \ell^{-1}R(h)^2(R(h)+1)), 
	\end{align*}
	where the constant $C_3$ is independent of $h$ and $\ell.$ 
		Assume $h\leq \ell^{1+\varsigma}$ with $\varsigma>0$ so that $\ell^{-1}h\leq h^{\frac{\varsigma}{1+\varsigma}}.$ Choose $R(h) = 2C_3 h.$ 
	For $h< h_{4}:=\min({h_3,h_2})
	$ with $h_{3}^{\frac{\varsigma}{1+\varsigma}}< \frac{1}{2(1+4C^3_2)^2}< \frac{1}{2}$, 
	\begin{align*}
	\vertiii{\textrm{I}_{h}\widetilde{\Psi} -\mu_{h}(\Theta_h)}_1
	%\leq & C_3 h+C_3h^{1+\varsigma}+ 4 C_3^3h^{2\varsigma}(2C_3 h+1)  \\&
	 \leq  C_3 h\left(1+ h^{\frac{\varsigma}{1+\varsigma}}(1+4C_3^2)+8C_3^3hh^{\frac{\varsigma}{1+\varsigma}} \right)\leq { C_3 h\bigg(1+\frac{1}{2}+\frac{1}{2}\frac{8C_3^3 h}{(1+4C_3^2)^2}\bigg)} .
	%\\&\leq Ch(1+\frac{1}{2}(\frac{1}{4}+1))\leq 2Ch=R(h).
	\end{align*}
	Since $h<h_{3}< \frac{1}{2^{\frac{1+\varsigma}{\varsigma}}}<1$ and $\frac{8C_3^3}{(1+4C_3^2)^2}<1,$ $\vertiii{\textrm{I}_{h}\widetilde{\Psi} -\mu_{h}(\Theta_{h})}_1\leq 2C_3 h=R(h).$ 
	\medskip
	
	\noindent	{\it 	Step 3 ($\mu_h$ is a contraction)}.  
%	In this step, we prove that for $\Theta_1,  \Theta_2 \in \mathbb{B}_{R(h)}({\rm I}_{h}\widetilde{\Psi})$,
%$$	\vertii{\mu_{h}(\Theta_1)-\mu_{h}(\Theta_2)}\lesssim  h^{\frac{\varsigma}{1+\varsigma}}\vertii{\Theta_1-\Theta_2}.$$
Let $\Theta_{1}, \Theta_{2} \in \mathbb{B}_{R(h)}({\rm I}_{h}\widetilde{\Psi})$ and set  $\e_1:= \textrm{I}_{h}\widetilde{\Psi}-\Theta_1,$ $\e_2:= \textrm{I}_{h}\widetilde{\Psi}-\Theta_2,$ and $\e_3:= \Theta_1-\Theta_2.$	
The linearity of $\dual{D\widetilde{N}({{\rm I}_{h}\widetilde{\Psi}}) \cdot,\cdot},$ \eqref{mu nonlinear map ferro},  and a re-arrangement of terms yields that for all $\Phi_h\in \V_h,$
	\begin{align}\label{contraction equation1 ferro}
	&\dual{D\widetilde{N}({{\rm I}_{h}\widetilde{\Psi}}) (\mu_{h}(\Theta_1)- \mu_{h}(\Theta_2)),\Phi_{h} }=\dual{D\widetilde{N}({{\rm I}_{h}\widetilde{\Psi}}) \mu_{h}(\Theta_1) ,\Phi_h } -\dual{D\widetilde{N}({{\rm I}_{h}\widetilde{\Psi}}) \mu_{h}(\Theta_2),\Phi_{h} }\notag\\&
%	=(2 B_2(\textrm{I}_{h}\widetilde{\Psi},\Theta_{1}, \Phi_{h} ) - B_2(\Theta_1,\Theta_1, \Phi_{h})-2 B_2(\textrm{I}_{h}\widetilde{\Psi},\Theta_{2}, \Phi_{h} )+ B_2(\Theta_2,\Theta_2, \Phi_{h}))+(6B_3(\textrm{I}_{h}{\Psi}_{\mathbf{g}} ,\textrm{I}_{h}\widetilde{\Psi},\Theta_{1}, \Phi_{h} ) \notag\\&\quad -3B_3(\textrm{I}_{h}{\Psi}_{\mathbf{g}},\Theta_1,\Theta_1, \Phi_{h})-6B_3(\textrm{I}_{h}{\Psi}_{\mathbf{g}} ,\textrm{I}_{h}\widetilde{\Psi},\Theta_{2}, \Phi_{h} )+3B( \textrm{I}_{h}{\Psi}_{\mathbf{g}},\Theta_2,\Theta_2,\Phi_{h}))+
%	(3B_3(\textrm{I}_{h}\widetilde{\Psi},\textrm{I}_{h}\widetilde{\Psi}, \Theta_1 , \Phi_{h} )\notag\\&	\quad- B_3(\Theta_1,\Theta_1, \Theta_1,\Phi_{h}) -3B(\textrm{I}_{h}\widetilde{\Psi},\textrm{I}_{h}\widetilde{\Psi}, \Theta_2 , \Phi_{h} )+B_3(\Theta_2,\Theta_2,\Theta_2,\Phi_{h}))\notag\\&
= \big(B_2(\e_1,\e_3, \Phi_{h}) + B_2(\e_2,\e_3, \Phi_{h}) \big)+3\big( B_3(\textrm{I}_{h}{\Psi}_{\mathbf{g}},\e_1,\e_3 ,\Phi_{h})+B_3(\textrm{I}_{h}{\Psi}_{\mathbf{g}},\e_2, \e_3,\Phi_{h})	\big)\notag\\&\quad+\big(3B_3(\textrm{I}_{h}\widetilde{\Psi},\textrm{I}_{h}\widetilde{\Psi}, \Theta_1 , \Phi_{h} )
- B_3(\Theta_1,\Theta_1, \Theta_1,\Phi_{h})-3B_3(\textrm{I}_{h}\widetilde{\Psi},\textrm{I}_{h}\widetilde{\Psi}, \Theta_2 , \Phi_{h} )+B_3(\Theta_2,\Theta_2,\Theta_2,\Phi_{h})\big)
	\notag\\&=: T_1'+T_2'+T_3'. 
	\end{align}
 %Since $ \Theta_1,  \Theta_2 \in \mathbb{B}_{R(h)}({\rm I}_{h}\widetilde{\Psi})$,
 Note that $\vertii{\e_1} \leq 2C_3 h$, $\vertii{\e_2} \leq 2C_3 h$. 
 This combined with Lemma \ref{boundedness_all ferro}$(ii)$-$(iii)$ and \eqref{bound for Psig} implies 
 \begin{align*}
&T_1'
%= B_2(\textrm{I}_{h}\widetilde{\Psi}-\Theta_1,\Theta_1-\Theta_2, \Phi_{h}) + B_2(\textrm{I}_{h}\widetilde{\Psi}-\Theta_2,\Theta_1-\Theta_2, \Phi_{h})  
\lesssim  \ell^{-1}h\vertii{\e_3} \vertii{\Phi_h}
 \text{ 	and }
T_2' 
%:=3(B_3(\Theta_2- \Theta_{1}, \Theta_2-\textrm{I}_{h}\widetilde{\Psi},\textrm{I}_{h}{\Psi}_{\mathbf{g}},\Phi_{h})+ B_3(\Theta_1-\textrm{I}_{h}\widetilde{\Psi},\Theta_2- \Theta_{1} ,\textrm{I}_{h}{\Psi}_{\mathbf{g}},\Phi_{h}))
 \lesssim  \ell^{-1}h\vertii{{\e_3}}\vertii{\Phi_h}.
 \end{align*}
 The estimation of the term $T_3'$  utilizes fourth inequality of Lemma \ref{Properties of bilinear trilinear and quadrilinear forms}$(iii)$ with $\boldsymbol{\eta}:=\textrm{I}_{h}\widetilde{\Psi},$   Lemma \ref{Interpolation estimate}, and  $ \vertii{\e_1} \leq 2C_3 h$,  $\vertii{\e_2} \leq 2C_3 h$.
%Lemma \ref{Properties of bilinear trilinear and quadrilinear forms}$(iii)$ and $\vertii{\e_1} \leq 2C_3 h$, $\vertii{\e_2} \leq 2C_3 h$ yields
	\begin{align*}%\label{contraction equation2 ferro}
T_3'\lesssim \ell^{-1}\vertiii{\e_3}_1(\vertiii{\e_1}_1^2+\vertiii{\e_2}_1^2 +(\vertiii{\e_1}_1+\vertiii{\e_2}_1)\vertiii{\textrm{I}_{h}\widetilde{\Psi}}_1) \vertiii{\Phi_{h}}_1	\lesssim  \ell^{-1}h(1+ h)\vertii{\e_3} \vertii{\Phi_h}.
	\end{align*}
	Substitute the above displayed three estimates for $T_1' ,\cdots ,T_3' $ in \eqref{contraction equation1 ferro}. This plus the discrete inf-sup condition in Lemma \ref{discrete inf sup}$(ii)$ implies that there exists a $\Phi_{h} \in \V_h$ with $\vertii{\Phi_{h}}=1$ such that 
	\begin{align*}%\label{contraction equation5 ferro}
	\vertii{\mu_{h}(\Theta_1)- \mu_{h}(\Theta_2)} \leq  4\beta^{-1}\dual{D N({{\rm I}_{h}\widetilde{\Psi}}) (\mu_{h}(\Theta_1)- \mu_{h}(\Theta_2)),\Phi_{h} }
	\lesssim  \ell^{-1}h(1+ h) \vertii{\Theta_1-\Theta_2}.
	\end{align*}
 The assumption $h\leq \ell^{1+\varsigma}$ allows %$\ell^{-1}h\leq h^{\frac{\varsigma}{1+\varsigma}}$.
	$$	\vertii{\mu_{h}(\Theta_1)-\mu_{h}(\Theta_2)}\lesssim  h^{\frac{\varsigma}{1+\varsigma}}\vertii{\Theta_1-\Theta_2},$$
	 where 	the hidden constant in $"\lesssim"$ depends on $\vertiii{\widetilde{\Psi}}_{2}$, $\vertiii{\mathbf{g}}_{\frac{3}{2}}$, $\beta$, $C_I$, $C_3$.
	 % and the constants in Sobolev embedding results.
	\medskip
	
\noindent {\it Step 4 (Existence and uniqueness)}. The nonlinear map $\mu_h$ is well-defined, continuous and maps a closed convex subset $\mathbb{B}_R(\textrm{I}_{h}\widetilde{\Psi})$ of a Hilbert space $\V_{h}$ to itself. Therefore, Brouwer's fixed point theorem \cite{KesavaTopicsFunctinal} and {\it Step 3} imply the existence and uniqueness of the fixed point, say $\widetilde{\Psi}_{h}$ in the ball $\mathbb{B}_R(\textrm{I}_{h}\widetilde{\Psi})$. A triangle inequality, $\vertii{\textrm{I}_{h}\widetilde{\Psi}- \widetilde{\Psi}_h}\leq 2C_3  h$ and 
	Lemma \ref{Interpolation estimate} show  $	\vertii{{\widetilde{\Psi}}-{\widetilde{\Psi}}_h} \leq (C_I + 2C_3) h$. This concludes the proof.
\end{proof}
\begin{proof}[\textbf{Proof of  Theorem  \ref{energy  norm error estimate ferro}}]
	Recall that $\widetilde{\Psi}_h= {\Psi}_h-\textrm{I}_h {\Psi}_{\mathbf{g}}$, where ${\Psi}_h$ satisfies the discrete non-linear system \eqref{discrete nonlinear problem ferro}.
	 % that approximates $\Psi $.  
	   Theorem \ref{reduced energy norm error estimate} shows the existence and local uniqueness of the discrete solution ${\Psi}_{h}$. Moreover, $\widetilde{\Psi}_{h} \in \mathbb{B}_R(\textrm{I}_{h}\widetilde{\Psi}) $ yields that ${\Psi}_{h} \in \mathbb{B}_R(\textrm{I}_{h}{\Psi}). $ This, Lemma \ref{Interpolation estimate} and triangle inequality lead to 
	\[	\vertii{{\Psi}-{\Psi}_h}  \lesssim \vertii{{\Psi}-\textrm{I}_{h}{\Psi}} + \vertii{\textrm{I}_{h}{\Psi}-{\Psi}_h} \lesssim   h. \qedhere
	\]
%	\begin{align*}
%	\vertii{{\Psi}-{\Psi}_h}  \lesssim \vertii{{\Psi}-\textrm{I}_{h}{\Psi}} + \vertii{\textrm{I}_{h}{\Psi}-{\Psi}_h} \lesssim   h. 
%	\end{align*}
\end{proof}
%\begin{rem}
%	In \cite{Gunzburger1992}, the authors discussed the finite element approximation of Stokes and Navier-Stokes equations with inhomogeneous essential boundary conditions. Lagrange multipliers are used to enforce the inhomogeneous boundary condition. The existence and uniqueness of discrete solutions are established through a discrete-infsup condition whereas the error estimates are obtained by a judicious  choice of finite element spaces for which the Lagrange multiplier calculation along the boundary uncouples from that of the primary interior variables. The non-linearity appeared in the model problem considered in this paper is different. Our method utilizes the lifting technique which modifies the inhomogeneous system to a homogeneous essential boundary valued problem. The existence and uniqueness of discrete solutions are guaranteed by a fixed point argument. An optimal order energy norm estimate is obtained with respect to the smoothness of the boundary data which is approximated by interpolation. However, once we have the energy norm error estimate, the best approximation result and the $L^2$ norm error estimate established next will work for any approximation of the boundary data on the finite element space. 
%\end{rem}
The best approximation result presented in Theorem \ref{thmbestapproximationComforming} is established next. The technique used in \cite[Theorem 3.3]{DGFEM} yields a  best approximation result in $\widehat{\X}_h := \{\Theta_h \in \X_h |\, \Theta_h=\mathbf{g}_h  \text{ 	on } \partial \Omega \} \subset \X_h.$ In this article, we use elliptic projection of $\Psi$ onto $\X_h$ to establish the best approximation result in $\X_h$.
\begin{proof}[\textbf{Proof of Theorem  \ref{thmbestapproximationComforming}}] 
{\it Step 1 (Best approximation on $\widehat{\X}_h := \{\Theta_h \in \X_h |\, \Theta_h=\mathbf{g}_h  \text{ 	on } \partial \Omega \}$)}. Set $\e:=\Psi -\Psi_h.$ The Taylor series expansion of ${N}(\cdot; \cdot)$ around $\Psi$ imply  
%and \eqref{continuous nonlinear ferro}, \eqref{discrete nonlinear problem ferro} 
\begin{align*}
{N}(\Psi_{h }; \Phi_{h})&= {N}(\Psi; \Phi_{h}) - \dual{{DN}(\Psi)\e, \Phi_{h}} +\frac{1}{2}\dual{{ D}^2N(\Psi)(\e)\e, \Phi_{h}}-\frac{1}{6}\dual{{ D}^3N(\Psi)(\e)(\e)\e, \Phi_{h}}.
%\\&
%=  \dual{DN(\Psi)\e, \Phi_{h}} -\frac{1}{2}\dual{{D}^2N(\Psi)(\e)\e, \Phi_{h}}+\frac{1}{6}\dual{{D}^3N(\Psi)(\e)(\e)\e, \Phi_{h}}.
\end{align*}
%A use of \eqref{continuous nonlinear ferro} and \eqref{discrete nonlinear problem ferro} leads to 
%\begin{align*}
%0=  \dual{DN(\Psi)\e, \Phi_{h}} -\frac{1}{2}\dual{{D}^2N(\Psi)(\e)\e, \Phi_{h}}+\frac{1}{6}\dual{{D}^3N(\Psi)(\e)(\e)\e, \Phi_{h}}.
%\end{align*}
For $\widehat{\Psi}_h \in \widehat{\X}_h$, use $\e= (\Psi-\widehat{\Psi}_h)+ (\widehat{\Psi}_h - \Psi_h)$,  the linearity of $\dual{DN(\Psi)\cdot, \cdot} $, and  \eqref{continuous nonlinear ferro}, \eqref{discrete nonlinear problem ferro} to obtain 
\begin{align}\label{4.30}
\dual{DN(\Psi)(\widehat{\Psi}_h - \Psi_h), \Phi_{h}}=\dual{DN(\Psi)(\widehat{\Psi}_h-\Psi), \Phi_{h}}+\frac{1}{2}\dual{{D}^2N(\Psi)(\e)\e, \Phi_{h}}-\frac{1}{6}\dual{{D}^3N(\Psi)(\e)(\e)\e, \Phi_{h}}.
\end{align}
 The identities $\dual{DN(\Psi)\e, \Phi_{h}}=A(\e,\Phi_{h})+B_1(\e,\Phi_{h})+2B_2(\Psi,\e,\Phi_{h})+3B_3(\Psi,\Psi,\e,\Phi_{h}),$ $\dual{{D}^2N(\Psi)(\e)\e, \Phi_{h}} = 2B_2(\e,\e, \Phi_{h})+6 B_3(\e,\e, \Psi, \Phi_{h})$, $\dual{{D}^3N(\Psi)(\e)(\e)\e, \Phi_{h}}=6B_3(\e,\e, \e, \Phi_{h}),$ and Lemma \ref{boundedness_all ferro}  leads to
	\begin{align}\label{4.31}
	\vertiii{DN(\Psi)}_{\mathbf{L}^2}\lesssim (1+\ell^{-1}),\, \vertiii{{ D}^2N(\Psi)}_{\mathbf{L}^2}\lesssim \ell^{-1} \text{ and } \vertiii{{ D}^3N(\Psi)}_{\mathbf{L}^2}\lesssim \ell^{-1}.
	\end{align}
	A triangle inequality yields 
	$\vertiii{\e}_1\leq \vertiii{\Psi- \widehat{\Psi}_h}_1+ \vertiii{\widehat{\Psi}_h - \Psi_h}_1$. 	Since $(\widehat{\Psi}_h-{\Psi}_h )|_{\partial \Omega}=0, $ the discrete inf-sup condition from Remark \ref{discrete inf-sup condition for psi} with $\Theta_h:=\widehat{\Psi}_h-{\Psi}_h \in \V_h \subset \V$, \eqref{4.30} and \eqref{4.31} yields
%	\begin{align*}%\label{4.32}
%	\frac{\beta}{2}\vertiii{\widehat{\Psi}_h-{\Psi}_h}_1 \leq \dual{DN(\Psi)(\widehat{\Psi}_h-{\Psi}_h), \Phi_h}\lesssim (1+\ell^{-1})\vertiii{\widehat{\Psi}_h-\Psi}_1+\ell^{-1} \vertiii{\e}_1^2+\ell^{-1} \vertiii{\e}_1^3.
%	\end{align*}
%	The triangle inequality and the above displayed estimate leads to 
	\begin{align}\label{5.19}
	\vertiii{\e}_1&\lesssim \vertiii{\Psi- \widehat{\Psi}_h}_1+\dual{DN(\Psi)(\widehat{\Psi}_h-{\Psi}_h), \Phi_h} \lesssim (1+\ell^{-1})\vertiii{\Psi- \widehat{\Psi}_h}_1+\ell^{-1}\vertiii{\e}_1^2(1+\vertiii{\e}_1). 
	\end{align}
	For a sufficiently small choice of the discretization parameter $h=O(\ell^{1+\varsigma}) $ with $\varsigma>0$, use  $\vertiii{\e}_1\lesssim h$ from Theorem \ref{energy  norm error estimate ferro}, $\ell^{-1}h\leq h^{\frac{\varsigma}{1+\varsigma}},$ and $h<1$ (this holds for $h<h_4$) in \eqref{5.19} to obtain 
	\begin{align*}
 C_4\vertiii{\e}_1  \leq (1+\ell^{-1})\vertiii{\Psi-\widehat{\Psi}_h}_1 +h^{\frac{\varsigma}{1+\varsigma}}\vertiii{\e}_1,
	\end{align*}
	where the constant $C_4$ depends on  $\abs{c}, \vertiii{\Psi}_2$, $C_I, C_3$ and $\beta$.
A sufficiently small choice of $h < h_6:=\min(h_5, h_4)$ with $h_5^{\frac{\varsigma}{1+\varsigma}}=\frac{C_4}{2}$ leads to 
	\begin{align}\label{Best approximation step 2}
	\vertiii{\Psi - \Psi_{h}}_1 \lesssim (1+\ell^{-1})\vertiii{\Psi-\widehat{\Psi}_h}_1.
	\end{align}
\noindent		{\it Step 2 (Best approximation on $\X_h$).} Let $R_h\Psi$ be the elliptic projection ($\mathbf{H}^1(\Omega)$ projection) of $\Psi$ onto $\X_h$ defined by 
\begin{align*}
A(\Psi -R_h\Psi, \Phi_h )+(\Psi -R_h\Psi, \Phi_h)=0 \text{ for all } \Phi_h \in \X_h.
\end{align*}
%Note that the $\mathbf{H}^1(\Omega)$-inner product is written in terms of $A(\cdot,\cdot)+ (\cdot,\cdot)$.
 Then it holds
\begin{align}\label{elliptic projection}
\vertiii{\Psi - R_h\Psi}_1= \inf_{ {\Psi}^*_h \in \X_h}\vertiii{\Psi -{\Psi}^*_h}_1.
\end{align}
 Let $\mathbf{g}_h^*= R_h \Psi|_{\partial \Omega}.$ Choose $\boldsymbol{\eta}_h \in \X_h$ such that $\boldsymbol{\eta}_h|_{\partial \Omega}=\mathbf{g}_h-\mathbf{g}_h^* $ and $\vertiii{\boldsymbol{\eta}_h}_1\leq \vertiii{\mathbf{g}_h-\mathbf{g}_h^*}_{\frac{1}{2},\partial \Omega} $. For $\widehat{\Psi}_h= \boldsymbol{\eta}_h+ R_h\Psi \in \widehat{\X}_h$, this plus a triangle inequality  lead to
\begin{align*}
\vertiii{\Psi - \widehat{\Psi}_h}_1 \leq \vertiii{\Psi - R_h\Psi}_1 + \vertiii{R_h\Psi-\widehat{\Psi}_h }_1 
\leq \vertiii{\Psi - R_h\Psi}_1 + \vertiii{\mathbf{g}_h-\mathbf{g}_h^*}_{\frac{1}{2},\partial \Omega}.
\end{align*}
A triangle inequality $\vertiii{\mathbf{g}_h-\mathbf{g}_h^*}_{\frac{1}{2},\partial \Omega} \leq \vertiii{\mathbf{g}-\mathbf{g}_h}_{\frac{1}{2},\partial \Omega}+\vertiii{\mathbf{g}-\mathbf{g}_h^*}_{\frac{1}{2},\partial \Omega}$, trace inequality
\begin{align}\label{a trace inequality for boundary term}
\vertiii{\mathbf{g}-\mathbf{g}_h^*}_{\frac{1}{2},\partial \Omega} \leq \vertiii{{\Psi} -R_h\Psi}_1, 
\end{align} and the definition of elliptic projection in \eqref{elliptic projection} show  
\begin{align*}
\vertiii{\Psi - \widehat{\Psi}_h}_1 \leq \inf_{ {\Psi}^*_h \in \X_h}\vertiii{\Psi -{\Psi}^*_h}_1 + \vertiii{\mathbf{g}-\mathbf{g}_h}_{\frac{1}{2},\partial \Omega} .
\end{align*}
This combined with the best approximation result \eqref{Best approximation step 2} obtained in {\it Step 1} leads to the desired estimate. 
\end{proof}
\begin{lem}[Estimate for boundary term]\cite{Gunzburger1992} \label{An estimate for elliptic projection} Let $\widehat{\Psi} \in \X$ with $\widehat{\Psi} |_{\partial \Omega}=\widehat{\mathbf{g}}$ be given. Let $R_h\widehat{\Psi} \in \X_h$ be the $\mathbf{H}^1(\Omega)$-projection of $\widehat{\Psi}$ onto $\X_h$ and $\widehat{\mathbf{g}}_h: =( R_h\widehat{\Psi})|_{\partial \Omega}.$ Then 
\begin{align*}
\vertiii{\widehat{\mathbf{g}} - \widehat{\mathbf{g}}_h }_{0, \partial \Omega} \leq Ch^\frac{1}{2} \vertiii{\widehat{\Psi} -R_h\widehat{\Psi}}_1.
\end{align*}
	
\end{lem}
\begin{rem}[Best approximation on $\X_h$] A triangle inequality and an inverse inequality yield 	\begin{align*}
	\vertiii{\mathbf{g}-\mathbf{g}_h}_{\frac{1}{2},\partial \Omega} \leq \vertiii{\mathbf{g}-{\mathbf{g}}^*_h}_{\frac{1}{2},\partial \Omega}+ \vertiii{\mathbf{g}^*_h-\mathbf{g}_h}_{\frac{1}{2},\partial \Omega}\lesssim \vertiii{\mathbf{g}-\mathbf{g}^*_h}_{\frac{1}{2},\partial \Omega}+h^{-\frac{1}{2}} \vertiii{\mathbf{g}_h-\mathbf{g}^*_h}_{0,\partial \Omega}.
	\end{align*}
The trace inequality \eqref{a trace inequality for boundary term}, a triangle inequality 
and Lemma \ref{An estimate for elliptic projection} 
	applied  to the above inequality  yields 
\begin{align*}
\vertiii{\mathbf{g}-\mathbf{g}_h}_{\frac{1}{2},\partial \Omega} \lesssim \vertiii{ \Psi-R_h\Psi}_1+ h^{-\frac{1}{2}}\vertiii{\mathbf{g}-\mathbf{g}_h}_{0,\partial \Omega}.
\end{align*}
This combined with the best approximation result obtained in Theorem \ref{thmbestapproximationComforming} and \eqref{elliptic projection} leads to 
\[\pushQED{\qed} 
\vertiii{\Psi- \Psi_{h }}_1\lesssim (1+\ell^{-1})( \min_{\Psi^*_{h } \in \X_h}\vertiii{\Psi- \Psi^*_{h }}_{1} + h^{-\frac{1}{2}} \vertiii{\mathbf{g}-\mathbf{g}_h}_{0,\partial \Omega}). \qedhere
\popQED
\]
\end{rem}

The proof of $\mathbf{L}^2$ norm error in Theorem \ref{L2_error_estimate} is established next.
\begin{proof}[\textbf{Proof of  Theorem  \ref{L2_error_estimate}}] 
		Set $G=\Psi -\Psi_h \in \mathbf{L}^2(\Omega)$. Consider the well-posed dual linear problem that  
	seeks $\boldsymbol{\chi} \in  \mathbf{H}^2(\Omega) \cap \V$ such that
	\begin{align}\label{dualweak}
 A(\Phi, \boldsymbol{\chi})+ B_1(\Phi, \boldsymbol{\chi})+ 2B_2(\Psi,\Phi, \boldsymbol{\chi})+ 3B_3(\Psi,\Psi,\Phi, \boldsymbol{\chi})=(G, \Phi) \text{ for all }  \Phi \in \V
	\end{align}
 and satisfies 
 \begin{align}\label{regularity of dual problem}
 \vertiii{\boldsymbol{\chi}}_2 \lesssim (1+\ell^{-1})\vertiii{G}_0,
 \end{align}
	where the hidden constant in $"\lesssim"$ depends on $\abs{c},\vertiii{\Psi}_2$, $\beta$ and the constants in Sobolev embedding results.  	
For $G:=(w_1,w_2,w_3,w_4)$ and $\boldsymbol{\chi}:= (\chi_1,\chi_2,\chi_3,\chi_4)$, the strong form of the dual linear problem \eqref{dualweak} is defined as
\begin{equation}\label{Strong form of the dual problem}
\left.
\begin{split}
&	-\Delta \chi_1 + \ell^{-1}\big(-\chi_1 - \frac{c}{2}(u_3\chi_3- u_4\chi_4)+ \frac{1}{3}((3u_1^2 + u_2^2)\chi_1+2u_1u_2\chi_2)\big)=w_1,\\&
-\Delta \chi_2 + \ell^{-1}\big(-\chi_2 - \frac{c}{2}(u_3\chi_4+ u_4\chi_3)+ \frac{1}{3}((u_1^2 + 3u_2^2)\chi_2+2u_1u_2\chi_1)\big)=w_2,\\&
-\Delta \chi_3 + \ell^{-1}\big(-\chi_3 - \frac{c}{2}(u_3\chi_1+ u_1\chi_3+u_2\chi_4+ u_4\chi_2)+ \frac{1}{3}((3u_3^2 + u_4^2)\chi_3+2u_3u_4\chi_4)\big)=w_3,\\&
-\Delta \chi_4 + \ell^{-1}\big(-\chi_4 - \frac{c}{2}(u_2\chi_3+ u_3\chi_2-u_1\chi_4- u_4\chi_1)+ \frac{1}{3}((u_3^2 + 3u_4^2)\chi_4+2u_3u_4\chi_3)\big)=w_4.
\end{split}
\right\}
\end{equation}
	Let $\Psi^* \in \X $ (resp. $\Psi_h^* \in \X_h $) be extension of $\mathbf{g}$ (resp. $\mathbf{g}_h$)  such that $\Psi^*|_{\partial \Omega} =\mathbf{g}$ (resp.  $\Psi^*_h|_{\partial \Omega} =\mathbf{g}_h $).  Let $\Psi^0:= \Psi - \Psi^* ,$ and $\Psi^0_h:= \Psi_h - \Psi_h^* .$
	Set $\Phi=\Psi^0 -\Psi^0_h \in \V $ in \eqref{dualweak} to obtain
	\begin{align} \label{L2 estimate1}
(\Psi-\Psi_h, \Psi^0 -\Psi_h^0)	=\dual{{DN}(\Psi)(\Psi^0 -\Psi_h^0), \boldsymbol{\chi}}.
	\end{align}
	Test \eqref{Strong form of the dual problem}  with $\Psi_h^*-\Psi^*$ and use integration by parts to obtain 
	\begin{align} \label{L2 estimate2}
	(\Psi-\Psi_h, \Psi_h^*-\Psi^*)	=\dual{{DN}(\Psi)(\Psi_h^*-\Psi^*), \boldsymbol{\chi}}+ \int_{\partial \Omega} (\mathbf{g} -\mathbf{g}_h)\cdot \frac{\partial \boldsymbol{\chi}}{\partial \nu} \ds.
	\end{align}
	The trace inequality   $\vertiii{\frac{\partial \boldsymbol{\chi}}{\partial \nu}}_{\frac{1}{2}, \partial \Omega} \lesssim \vertiii{\boldsymbol{\chi}}_2$ leads to 
	$$ \int_{\partial \Omega} (\mathbf{g}-\mathbf{g}_h)\cdot \frac{\partial \boldsymbol{\chi}}{\partial \nu} \ds \leq \vertiii{\mathbf{g}-\mathbf{g}_h}_{-\frac{1}{2}, \partial \Omega}\vertiii{\frac{\partial \boldsymbol{\chi}}{\partial \nu}}_{\frac{1}{2}, \partial \Omega}\lesssim \vertiii{\mathbf{g}-\mathbf{g}_h}_{-\frac{1}{2},  \partial \Omega} \vertiii{\boldsymbol{\chi}}_2.$$ 
This and $\Psi^0 -\Psi^0_h= (\Psi -\Psi_h)+(\Psi_h^*-\Psi^*)$ applied to \eqref{L2 estimate1}  leads to 
	\begin{align} \label{L2 error equn1}
\vertiii{\Psi-\Psi_h}^2_0	\leq  \dual{{DN}(\Psi)(\Psi -\Psi_h), \boldsymbol{\chi}}+ \vertiii{\mathbf{g}-\mathbf{g}_h}_{-\frac{1}{2},  \partial \Omega} \vertiii{\boldsymbol{\chi}}_2.
\end{align}
A term  $\dual{{DN}(\Psi)(\Psi -\Psi_h), \textrm{I}_{h}\boldsymbol{\chi}}$ with $\textrm{I}_{h}\boldsymbol{\chi} \in \V_h \subset \V$ is added and subtracted to the first term on the right hand side of \eqref{L2 error equn1}, \eqref{continuous nonlinear ferro}-\eqref{discrete nonlinear problem ferro}  are utilized and simple manipulations are performed to arrive at
	\begin{align*}
	&\dual{{DN}(\Psi)(\Psi -\Psi_h), \boldsymbol{\chi}}= \dual{{DN}(\Psi)(\Psi -\Psi_h), \boldsymbol{\chi}- \textrm{I}_{h}\boldsymbol{\chi}}+\dual{{DN}(\Psi)(\Psi -\Psi_h), \textrm{I}_{h}\boldsymbol{\chi}} + N(\Psi_h, {\rm I}_{h}\boldsymbol{\chi}) -N(\Psi,{\rm I}_h \boldsymbol{\chi}) \notag \\&
	= \dual{{DN}(\Psi)(\Psi -\Psi_h), \boldsymbol{\chi}- \textrm{I}_{h}\boldsymbol{\chi}}+B_2({\Psi}_h-{\Psi}, {\Psi}_h- {\Psi}, \textrm{I}_{h}\boldsymbol{\chi} )
	+  (2B_3(\Psi,\Psi, \Psi, {\rm I}_{h}\boldsymbol{\chi})
	-3B_3(\Psi, \Psi, \Psi_h, {\rm I}_{h}\boldsymbol{\chi})   \notag\\&\quad+B_3(\Psi_h,\Psi_h, \Psi_h , {\rm I}_{h}\boldsymbol{\chi}) ) := T_1+ T_2+T_3.
	\end{align*}
	Here the term $T_2$  is a re-grouping of $$2B_2({\Psi}, {\Psi}- {\Psi}_h,  \textrm{I}_{h}\boldsymbol{\chi} )+B_2({\Psi}_h, {\Psi}_h,  \textrm{I}_{h}\boldsymbol{\chi} )-B_2({\Psi}, {\Psi},  \textrm{I}_{h}\boldsymbol{\chi} )= B_2(  {\Psi}_h-{\Psi}, {\Psi}_h, \textrm{I}_{h}\boldsymbol{\chi} ) -  B_2({\Psi},  {\Psi}_h-{\Psi},  \textrm{I}_{h}\boldsymbol{\chi} )$$ and obtained by applying the linearity and symmetry of $B_2(\cdot,\cdot,\cdot)$ in first two variables.
The definition of $\dual{{DN}(\Psi)\cdot,\cdot }$, Lemmas \ref{boundedness_all ferro}$(i)$-$(iii)$, \ref{Interpolation estimate} and Theorem \ref{energy  norm error estimate ferro} show 
	\begin{align*}
T_1:&=\dual{{DN}(\Psi)(\Psi -\Psi_h), \boldsymbol{\chi}- \textrm{I}_{h}\boldsymbol{\chi}}\\&	
= A(\Psi -\Psi_h, \boldsymbol{\chi}- \textrm{I}_{h}\boldsymbol{\chi})+B_1(\Psi -\Psi_h, \boldsymbol{\chi}- \textrm{I}_{h}\boldsymbol{\chi}) +2B_2(\Psi,\Psi -\Psi_h, \boldsymbol{\chi}- \textrm{I}_{h}\boldsymbol{\chi})+3B_3(\Psi,\Psi,\Psi -\Psi_h, \boldsymbol{\chi}- \textrm{I}_{h}\boldsymbol{\chi})
\\&	 \lesssim (1+\ell^{-1})\vertiii{\Psi -\Psi_h}_1 \vertiii{ \boldsymbol{\chi}- \textrm{I}_{h}\boldsymbol{\chi}}_1 \lesssim (1+\ell^{-1})h^2 \vertiii{\boldsymbol{\chi}}_2.	\end{align*}
Lemmas \ref{boundedness_all ferro}$(ii)$,  \ref{Interpolation estimate} and Theorem \ref{energy  norm error estimate ferro} yield 
	\begin{align*}
&		T_2:=B_2({\Psi}_h-{\Psi}, {\Psi}_h- {\Psi}, \textrm{I}_{h}\boldsymbol{\chi} )\lesssim \ell^{-1}\vertii{{\Psi}_h- {\Psi}}^2  \vertiii{\textrm{I}_{h}\boldsymbol{\chi}}_1 \lesssim  \ell^{-1}h^2 \vertiii{\boldsymbol{\chi}}_2.
	\end{align*}
		Set $\e:= \Psi - \Psi_{h }$. Utilize  the third inequality of Lemma \ref{Properties of bilinear trilinear and quadrilinear forms}$(iii)$ for  $\boldsymbol{\eta}:=\Psi$, $\Theta_h:=\Psi_h$, $\Phi_h:={\rm I}_{h}\boldsymbol{\chi}$, and then apply Lemma \ref{Interpolation estimate}, Theorem \ref{energy  norm error estimate ferro} to estimate $T_3$.
	\begin{align*}
	T_3& :=2B_3(\Psi, \Psi, \Psi, {\rm I}_{h}\boldsymbol{\chi})  -3B_3(\Psi,\Psi, \Psi_h, {\rm I}_{h}\boldsymbol{\chi}) + B_3(\Psi_h,\Psi_h, \Psi_h , {\rm I}_{h}\boldsymbol{\chi})\\& \lesssim \ell^{-1} \vertiii{\e}^2_1(\vertiii{\e}_1+\vertiii{\Psi }_1)\vertiii{\textrm{I}_{h}\boldsymbol{\chi}}_2 \lesssim \ell^{-1} h^2 (h+1) \vertiii{\boldsymbol{\chi}}_2.
	\end{align*}
	The estimates of $T_1,$ $T_2$ and $T_3$ in \eqref{L2 error equn1}  and  \eqref{regularity of dual problem} yield 
	\begin{align*}
	\vertiii{ \Psi -\Psi_h}_0 \lesssim (1+\ell^{-1} ) \big(h^2(1+\ell^{-1} ) +\vertiii{\mathbf{g}-\mathbf{g}_h }_{-\frac{1}{2}, \partial \Omega}\big) ,
	\end{align*} 
	where the constants suppressed in $"\lesssim"$ depends on  $\abs{c}, \vertiii{\Psi}_2$, $\beta$,  $C_I,  C_3$  and the constants in Sobolev embedding results.  	
	This completes the proof. 
\end{proof}	
\begin{rem}
	Integration by parts in \eqref{L2 estimate2} leads to the boundary term $\int_{\partial \Omega} (\mathbf{g} -\mathbf{g}_h)\cdot \frac{\partial \boldsymbol{\chi}}{\partial \nu} \ds$, which gives a sub-optimal convergence rate in the $\mathbf{L}^2$ norm. An optimal convergence rate $\mathcal{O}(h^2)$ in $\mathbf{L}^2$ norm is obtained using Nitsche's method discussed below.\qed
\end{rem}
\subsection{Nitsche's method}\label{Nitsche's method}
Let $\mathcal{E}$ (resp. $\mathcal{E}(\Omega)$ or $\mathcal{E}(\partial \Omega)$) denote the set of all (resp. interior or boundary) edges in $\mathcal{T}.$
For Nitsche's method,  the finite element space $\X_h$ associated with the triangulation $\mathcal{T}$ of the convex polygonal domain $\Omega \subset \mathbb{R}^2$ into triangles 
%The finite element subspace of $\h^1(\Omega)$ 
%$\X_h:=X_h \times X_h \times X_h \times X_h$, where $$X_h:= \{v \in H^1(\Omega)|\,\, v|_T \in P_1(T) \text{ for all } T \in \mathcal{T} \}$$ 
is endowed with the mesh dependent norm defined by  
 $\vertiii{\Phi_h}_h:= \norm{\varphi_{1}}_h + \norm{\varphi_{2}}_h+\norm{\varphi_{3}}_h + \norm{\varphi_{4}}_h $ for all $\Phi_h=(\varphi_{1},\varphi_{2}, \varphi_{3},\varphi_{4}) \in \X_h$,  where for all $v \in X_h,$ 
\begin{align*}
\norm{v}^2_h:=\int_{\Omega} \abs{\nabla v}^2 \dx + \sum_{ E \in \mathcal{E}(\partial \Omega)} \frac{\sigma}{h_E}\int_{E}  v^2 \ds.
\end{align*}
Here $\sigma>0$ is the penalty parameter and $h_E$ denote length of an edge $E$. Let $\nu_T$ denotes the unit outward normal along $ \partial T$ of $T \in \mathcal{T}$. The jump $[\varphi]_E$ of piecewise $H^1$ function $\varphi,$ i.e, $\varphi \in H^1(\mathcal{T}):=\{ v \in L^2(\Omega)|\,\, v \in H^1(T) \text{ for all } T \in \mathcal{T}  \},$ across $E \in \mathcal{E}$ is defined by 
\begin{equation*} 
[\varphi]_E(x):=
\begin{cases} 
 v|_{T_+}(x)-v|_{T_-}(x)& \text{ for } x \in E= \partial T_+ \cap \partial T_- \in \mathcal{E}(\Omega), \\
v(x) & \text{ for } x \in E \in \mathcal{E}(\partial\Omega), 
\end{cases} 
\end{equation*} 
where for the interior edge $E= \partial T_+ \cap \partial T_- \in \mathcal{E}(\Omega)$ with unit normal $\nu_E$ of fixed orientation,  the adjacent triangles $T_{\pm } \in \mathcal{T} $ are in an order such that $\nu_E =\nu_{T_+}|_E=-\nu_{T_-}|_E$.
 The discrete formulation for Nitsche's method that corresponds to \eqref{continuous nonlinear ferro} seeks $ \Psi_{h}\!\in\!\X_h$ such that for all $ \Phi_{h} \in  \X_h,$
\begin{align}\label{discrete nonlinear problem ferro Nitsche's method}
N_h(\Psi_{h}; \Phi_h):= A_h(\Psi_{h}, \Phi_h)+B_1(\Psi_{h}, \Phi_h)+B_{2}( \Psi_{h},\Psi_{h},\Phi_h)+ B_3(\Psi_{h},\Psi_{h},\Psi_{h}, \Phi_h)-L_h(\Phi_h) = 0.
\end{align}
For $\Theta=(\theta_1,\theta_2,\theta_3,\theta_4)$ and $\Phi=(\varphi_1,\varphi_2,\varphi_3,\varphi_4) \in \X,$ $A_h(\Theta, \Phi):= \sum_{i=1}^{4}a_h(\theta_i, \varphi_i), \text{  and } L_h(\Phi) = \sum_{i=1}^{4}L^i_h(\varphi_i).$ Note that the forms $B_i, i=1,2,3$ are defined in Section \ref{Weak and finite element formulations}.
%\begin{align*}
%&A_h(\Theta, \Phi):= \sum_{i=1}^{4}a_h(\theta_i, \varphi_i), \text{  and } L_h(\Phi) = \sum_{i=1}^{4}L^i_h(\varphi_i).
%%+ a_h(\theta_2, \varphi_2) +a_h(\theta_3, \varphi_3) +a_h(\theta_4, \varphi_4)  \text{  and }\\&
%%L_h(\Phi) = L^1_h(\varphi_1)+ L^2_h(\varphi_2)+ L^3_h(\varphi_3)+L^4_h(\varphi_4).
%\end{align*}
For 
$\theta, \varphi \in H^1(\Omega)$, $\mathbf{g}=(g_1, g_2,g_3, g_4)$, 
\begin{align*}
&a_h(\theta,\varphi):= \int_\Omega \nabla \theta \cdot \nabla \varphi \dx-  \int_{\partial\Omega}\frac{\partial \theta}{\partial \nu} \varphi \ds  -\int_{\partial\Omega}\theta\frac{\partial \varphi}{\partial \nu} \ds + \sum_{E \in  \mathcal{E}(\partial \Omega)} \frac{\sigma}{h_E}\int_{ E} \theta \varphi \ds ,	
\notag\\&
\text{  and }\,\,
L^i_h(\varphi):=
-  \int_{\partial\Omega}g_i \frac{\partial \varphi}{\partial \nu}\ds + \sum_{E \in  \mathcal{E}(\partial \Omega)} \frac{\sigma}{h_E} \int_{ E}g_i \varphi \ds \text{ for } 1\leq i \leq 4,
\end{align*}
where $\nu$ denotes the outward unit normal associated to $\partial \Omega.$ For all $\Theta_h, \Phi_h \in \X_h$,  define the discrete bilinear form in this case as
\begin{align*}%\label{bilinear form for Nitsche}
\dual{{DN}_h(\Psi)\Theta_h, \Phi_h}: = A_h(\Theta_h,\Phi_h)+B_1(\Theta_h,\Phi_h)+2B_2(\Psi,\Theta_h,\Phi_h)+3B_3(\Psi,\Psi,\Theta_h,\Phi_h)
\end{align*}
and the perturbed bilinear form  as  $$\dual{{ DN}_h(\textrm{I}_h\Psi)\Theta_h, \Phi_h} := A_h(\Theta_h,\Phi_h)+B_1(\Theta_h,\Phi_h) +2B_2(\textrm{I}_h\Psi,\Theta_h,\Phi_h)+3B_3(\textrm{I}_h\Psi,\textrm{I}_h\Psi,\Theta_h,\Phi_h).$$
\begin{thm}[Existence, uniqueness and error estimates]\label{error estimates in Nitsche method for ferronematics}
	Let $\Psi$ be a regular solution of the non-linear system
	\eqref{continuous nonlinear ferro} such that \eqref{H2 bound for minimizers} holds.  For a given fixed $\ell>0,$ a sufficiently large $\sigma$ and a sufficiently small discretization parameter chosen as $h=O(\ell^{1+\varsigma} )$ for any $\varsigma > 0$, there exists a unique solution $\Psi_h$ of the discrete non-linear problem \eqref{discrete nonlinear problem ferro Nitsche's method} that approximates $\Psi $ such that 
	\begin{align*}
	(i)\,	\vertiii{\Psi-\Psi_h}_h\lesssim h, \,\text{  and } \,\,(ii)\:	\vertiii{\Psi-\Psi_h}_{0}\lesssim h^2(1+ (1+\ell^{-1} )^2).
	\end{align*}
\end{thm}
The proof of Theorem \ref{error estimates in Nitsche method for ferronematics}$(i)$ follows similar methodology of Theorem \ref{energy  norm error estimate ferro} with the choice of the non-linear map $\mu_h:\X_h \rightarrow \X_h$  defined by: for $\Theta_h,\Phi_h \in \X_h$,
	\begin{align*}%\label{2.5.1}
	\dual{{DN}_h({{\rm I}_h\Psi}) \mu_h(\Theta_h),\Phi_h }
	&= 2B_2( \textrm{I}_h\Psi,\Theta_h,\Phi_h)+3B_3(\textrm{I}_h\Psi, \textrm{I}_h\Psi,\Theta_h,\Phi_h) - B_2(\Theta_h, \Theta_h,\Phi_h)\notag\\&\quad- B_3(\Theta_h,\Theta_h, \Theta_h,\Phi_h) +L_h(\Phi_h).
	\end{align*}
\begin{proof}[\textbf{Proof of Theorem  \ref{error estimates in Nitsche method for ferronematics}$(ii)$}]
Set $G:= \textrm{I}_h\Psi -\Psi_h \in \mathbf{L}^2(\Omega).$	Multiply \eqref{Strong form of the dual problem} by $\Phi_h=\textrm{I}_h\Psi -\Psi_h$, use integration by parts, and then add and subtract an intermediate term as
		\begin{align} \label{L2 error estimate FerroNitsche1}
&\vertiii{\textrm{I}_h\Psi -\Psi_h}_0^2=	\dual{{DN}_h(\Psi)\textrm{I}_h\Psi -\Psi_h, \boldsymbol{\chi}}\notag\\&=	\dual{{DN}_h(\Psi)\textrm{I}_h\Psi -\Psi, \boldsymbol{\chi}}+	\dual{{DN}_h(\Psi)\Psi -\Psi_h, \boldsymbol{\chi}-\textrm{I}_h\boldsymbol{\chi}}+	\dual{{DN}_h(\Psi)\Psi -\Psi_h, \textrm{I}_h\boldsymbol{\chi}}=:T_1+T_2+T_3.
	\end{align}
Set $B_{L_1}(\cdot, \cdot):=B_1(\cdot, \cdot) +2B_2(\Psi,\cdot, \cdot)+3B_3(\Psi,\Psi,\cdot, \cdot),$ then  the terms 
\begin{align*}
&T_1:= A_h(\textrm{I}_h\Psi -\Psi, \boldsymbol{\chi})+B_{L_1}(\textrm{I}_h\Psi -\Psi, \boldsymbol{\chi}), \,\, T_2:= A_h(\Psi -\Psi_h,\boldsymbol{\chi}-\textrm{I}_h\boldsymbol{\chi})+B_{L_1}(\Psi -\Psi_h, \boldsymbol{\chi}-\textrm{I}_h\boldsymbol{\chi}).%, 
%\\& \text{and } T_3:= A_h(\Psi -\Psi_h, \textrm{I}_h\boldsymbol{\chi})+B_{L_1}(\Psi -\Psi_h, \textrm{I}_h\boldsymbol{\chi}).
\end{align*}
		A use of definition of $A_h(\cdot, \cdot)$, $\boldsymbol{\chi} = 0$ on $\partial \Omega$,  integration by parts and cancellation of terms yields $$A_h({\rm I}_h\Psi -\Psi, \boldsymbol{\chi})=  \int_\Omega \nabla ({\rm I}_h\Psi -\Psi) \cdot \nabla \boldsymbol{\chi} \dx -   \int_{\partial\Omega} \frac{\partial \boldsymbol{\chi}}{\partial \nu}\cdot ({\rm I}_h\Psi -\Psi )\ds = - \int_\Omega({\rm I}_h\Psi -\Psi) \cdot \Delta \boldsymbol{\chi} \dx .$$
%	\begin{align*}
%	&A_h({\rm I}_h\Psi -\Psi, \boldsymbol{\chi})=  \int_\Omega \nabla ({\rm I}_h\Psi -\Psi) \cdot \nabla \boldsymbol{\chi} \dx -   \int_{\partial\Omega} \frac{\partial \boldsymbol{\chi}}{\partial \nu}\cdot ({\rm I}_h\Psi -\Psi )\ds = - \int_\Omega({\rm I}_h\Psi -\Psi) \cdot \Delta \boldsymbol{\chi} \dx .
%	\end{align*}
This plus H\"older's inequality and  Lemma \ref{Interpolation estimate} show $$A_h({\rm I}_h\Psi -\Psi, \boldsymbol{\chi}) \leq \vertiii{{\rm I}_h\Psi -\Psi}_0 \vertiii{\boldsymbol{\chi}}_2 \lesssim h^2  \vertiii{\boldsymbol{\chi}}_2.$$ 
	The remaining $B_i, i=1,2,3$ terms in $T_1$ are estimated using Lemma \ref{boundedness_all ferro}$(i)$-$(iii)$,  and  Lemma \ref{Interpolation estimate}. 
	%The term $T_2$ (except $A_h(\cdot,\cdot)$ part) is estimated following the similar lines used for $T_1$ in Theorem \ref{L2_error_estimate}. %We have $A_h(\cdot,\cdot)$ term here instead of  $A(\cdot,\cdot)$ (in the expansion of $T_1$  in Theorem \ref{L2_error_estimate}) and is  computed utilizing 
The fact that $\boldsymbol{\chi}-{\rm I}_h\boldsymbol{\chi}=0$ on $\partial\Omega,$	Lemma \ref{Interpolation estimate} and 
	Theorem \ref{error estimates in Nitsche method for ferronematics}$(i)$ leads to the estimate of $A_h(\cdot,\cdot)$ term in $T_2$ as $A_h(\Psi-\Psi_h,\boldsymbol{\chi}-{\rm I}_h\boldsymbol{\chi}) \lesssim h\vertiii{\boldsymbol{\chi}}_2 \vertiii{\Psi-\Psi_h}_h \lesssim h^2\vertiii{\boldsymbol{\chi}}_2 .$    The term $T_3$ with $\eqref{discrete nonlinear problem ferro Nitsche's method}$ and $N_h(\Psi, \textrm{I}_h\boldsymbol{\chi})=0$ reads
	\begin{align*}
	&T_3:=	\dual{{DN}_h(\Psi)\Psi -\Psi_h, \textrm{I}_h\boldsymbol{\chi}}=	\dual{{ DN}_h(\Psi)\Psi -\Psi_h, \textrm{I}_h\boldsymbol{\chi}} + {N}_h(\Psi_h, \textrm{I}_h\boldsymbol{\chi}) -{N}_h(\Psi,\textrm{I}_h \boldsymbol{\chi})	\\&\quad =B_2({\Psi}_h-{\Psi}, {\Psi}_h- {\Psi}, \textrm{I}_{h}\boldsymbol{\chi} )+ (2B_3(\Psi,\Psi, \Psi, {\rm I}_{h}\boldsymbol{\chi})
		% + B_3(\Psi ,\Psi , \Psi, {\rm I}_{h}\boldsymbol{\chi}-\boldsymbol{\chi}))
		-3B_3(\Psi, \Psi, \Psi_h, {\rm I}_{h}\boldsymbol{\chi})  +B_3(\Psi_h,\Psi_h, \Psi_h , {\rm I}_{h}\boldsymbol{\chi}) ).
	\end{align*}
The term $T_3$ and the rest of The terms in $T_2$  are  directly comparable to the terms of Theorem \ref{L2_error_estimate}. The rest of the proof utilizes similar  methodology of Theorem \ref{L2_error_estimate} and hence is skipped here.
% Then the regularity result $\vertiii{\boldsymbol{\chi}}_2 \lesssim \vertiii{ \Psi -\Psi_h}_0 $, 
% %leads to $\vertiii{ \textrm{I}_h\Psi -\Psi_h}_0  \lesssim h^2 (1+\ell^{-1} )^2.$
%%		\begin{align*}
%%	\vertiii{ \textrm{I}_h\Psi -\Psi_h}_0  \lesssim h^2 (1+\ell^{-1} )^2.
%%	\end{align*} 
%%This plus 
%a triangle inequality and Lemma \ref{Interpolation estimate} show $\vertiii{ \Psi -\Psi_h}_0 \leq \vertiii{\Psi- \textrm{I}_h\Psi }_0+\vertiii{ \textrm{I}_h\Psi -\Psi_h}_0  \lesssim h^2(1+ (1+\ell^{-1} )^2).$
%	This concludes the proof.
\end{proof}	
\begin{rem}
	The boundary term in $T_1$ of Theorem  \ref{error estimates in Nitsche method for ferronematics}$(ii)$ appearing due to the integration by parts gets cancelled with the  boundary term in  $A_h({\rm I}_h\Psi -\Psi, \boldsymbol{\chi})$, whereas  similar type of boundary term in \eqref{L2 error equn1} of Theorem \ref{L2_error_estimate} leads to the sub-optimal convergence rate $\mathcal{O}(h^{\frac{3}{2}})$ in $\mathbf{L}^2$ norm. \qed
\end{rem}
%================================Nitsche ends=================================================
%=============================================================================================
\section{Numerical experiments}\label{Ferronematic numerical experiments}
This section reports on numerical experiments 
for the benchmark problem \cite{Ferronematics_2D} for dilute ferronematic suspensions, on  a re-scaled two-dimensional square domain $\Omega=(0,1)\times (0,1)$ with a uniform refinement strategy. Numerical solutions approximate the regular solutions of \eqref{continuous nonlinear ferro} for a fixed value of the parameters $\ell$ and $c.$ 
The discrete solution landscapes of \eqref{continuous nonlinear ferro}, for various parameter ($\ell, c$) values, the associated computational errors and convergence rates are explored for conforming FEM. Let $e_i$ and $h_i$ be the error and the mesh parameter at $i$-th level, respectively. The $i$-th level experimental order of convergence is defined by $\displaystyle \alpha_i:=\log(e_{i}/e_{i+1})/\log(h_{i}/h_{i+1})$ for $ i=1, \ldots , n-1$ and $n $ is the final iteration considered in numerical experiments.  Newton's method is applied to approximate the solutions of \eqref{continuous nonlinear ferro}. For detailed construction of the initial conditions/profiles, we refer to \cite{Hanetal2021Ferro2D,MultistabilityApalachong,DGFEM}. 
\begin{table}[H]
	\centering
	\begin{tabular}{c c c c c} 
		\hline
		Solution       &  $x=0$ &  $x=1$  & $y=0$&  $y=1$  \\ [0.5ex] 
		\hline
		$Q_{11}$     & -1  & -1 & 1  & 1 \\
		$Q_{12}$     & 0 & 0   & 0 & 0   \\
		$M_1$  &0 & 0 &-1  & 1   \\
		$M_2$     & 1 &-1 &0  &0\\
		\hline
	\end{tabular}
	\caption{ Tangential boundary conditions for solution components $Q_{11}, Q_{12}, M_1, M_2.$
	}
	\label{table:Tangential boundary conditions ferro}
\end{table}\noindent The Dirichlet tangent boundary conditions \cite{Ferronematics_2D,MultistabilityApalachong,Slavinec_2015, Tsakonas} are detailed in Table \ref{table:Tangential boundary conditions ferro}. The natural mismatch in the tangent boundary conditions of the director $\mathbf{n}$ and magnetization vector leads to the corner defects. 
We construct a Lipschitz continuous boundary condition $\mathbf{g}$ using the tangential boundary condition in Table \ref{table:Tangential boundary conditions ferro} and trapezoidal shape functions, \cite{MultistabilityApalachong} $\textit{T}_d:[0,1]\rightarrow {\mathbb{R}}$  defined as
\begin{equation*} 
\mathbf{g}=
\begin{cases} 
(\textit{T}_{d}(x),0, -\textit{T}_{d}(x),0 ) & \text{on}\,\,\,\, y=0, \\
(\textit{T}_{d}(x),0, \textit{T}_{d}(x),0)  & \text{on}\,\,\,\,  y=1, \\
(- \textit{T}_d(y),0, 0,\textit{T}_d(y) ) & \text{on}\,\,\,\, x=0 , \\
(- \textit{T}_d(y),0,0, -\textit{T}_d(y))  & \text{on}\,\,\,\, x=1, 
\end{cases} \,\, \text{  and } \,\,
\textit{T}_d(t)=
\begin{cases} 
t/d, & 0 \leq t \leq d, \\
1, &  d \leq  t \leq 1- d,  \\
(1-t)/d, & 1- d \leq t \leq 1,
\end{cases}
\end{equation*}  
where the parameter $d=3 \sqrt{\ell}$, is the size of mismatch region. 
\noindent For small choices of the parameter $\ell,$ the numerical results are divided into three categories according to the positive, negative and zero value of the coupling parameter $c$. 
\begin{figure}[H]
	\centering
	\begin{minipage}[t]{0.31\linewidth}
		\subfloat[$\Qvec_{D1}$ and $\Mvec$ profile ]{\includegraphics[height=1.8cm, width=4.5cm]{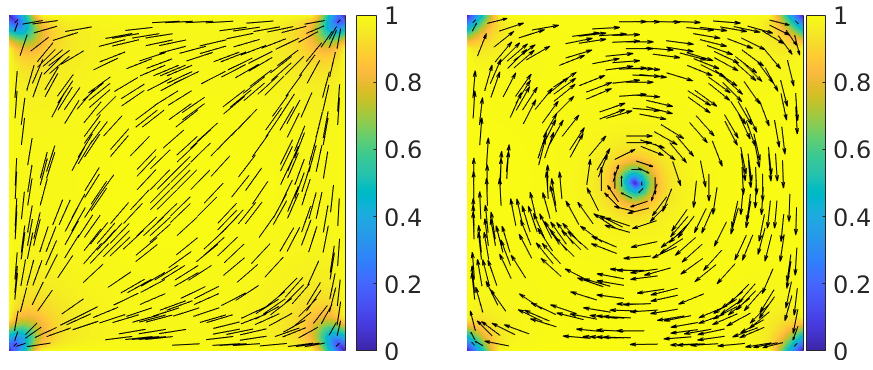}}\label{UncoupledD1}
		\vspace{-0.2 cm}\\ 
		\subfloat[$\Qvec_{R4}$ and $\Mvec$ profile]{\includegraphics[height=1.8cm, width=4.5cm]{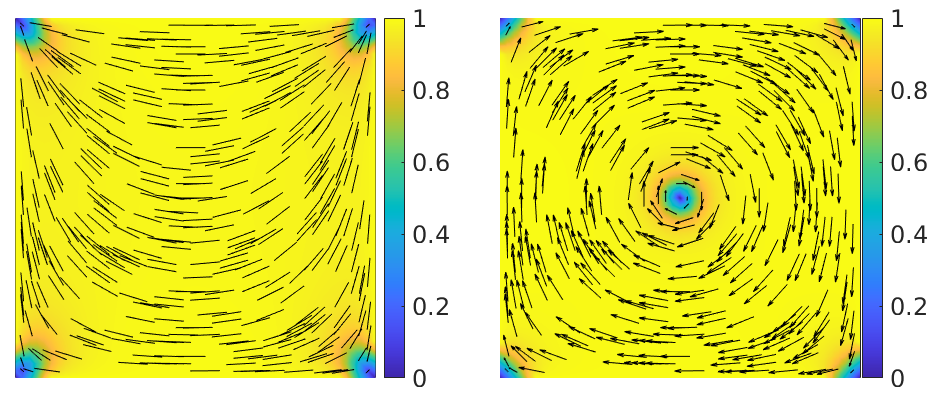}}\label{UncoupledR4}
	\end{minipage}
	\begin{minipage}[t]{0.31\textwidth}				
		\subfloat[]{\includegraphics[width=4.9cm,height=4.3cm]{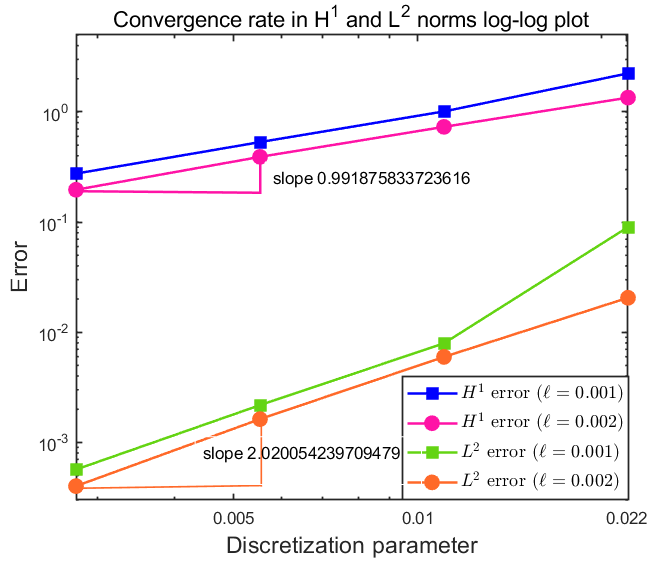}} 
	\end{minipage}
	\hspace{0.3 cm}
	\begin{minipage}[t]{0.31\linewidth}
		\subfloat[]{\includegraphics[width=4.9cm,height=4.3cm]{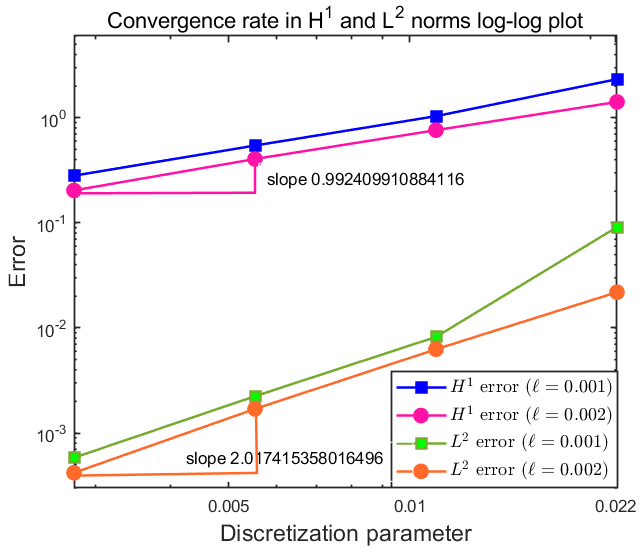}}
	\end{minipage}	
	\caption{Nematic $\Qvec$ and magnetic $\Mvec$  configurations: (a) D1 diagonal nematic and uncoupled magnetic profile and (b) R4 rotated nematic and  uncoupled magnetic profile for the parameter values  $\ell =0.001, c=0; $ Energy and  $\mathbf{L}^2$ norm error versus discretization parameter $h$ plots for the discrete solutions (c) $\Psi_h=(\Qvec_{D1},\Mvec)$, (d) $\Psi_h=(\Qvec_{R4},\Mvec)$ for two sets of parameter values $\ell=0.001, c=0$ and  $\ell=0.002, c=0$.}
	\label{Uncoupled system}
\end{figure}
\medskip

\noindent \textbf{Case I: the coupling parameter $\mathbf{c=0}$} 

\medskip

\noindent The two-dimensional planar bistable nematic device \cite{MultistabilityApalachong} (uncoupled system i.e., $c=0$) exhibits two sets of equilibrium configurations- 1) diagonally stable: the nematic directors roughly align along one of the square diagonals and there are two classes of diagonal solutions: D1 and D2, one for each square diagonal; 2) rotated states: here, the nematic director rotates by $\pi$ radians between a pair of opposite parallel edges, and there are $4$ classes of rotated solutions labelled by R1, R2, R3 and R4 respectively, related to each other by $\frac{\pi}{2}$ radians. The diagonal and rotated solutions are distinguished by the locations of the splay vertices; a splay vertex being a vertex such that the nematic director splays around the vertex and a bend vertex being such that the nematic director bends around the vertex in question. Each diagonal solution has a pair of diagonally opposite
splay vertices and each rotated solution has a pair of adjacent splay vertices, connected by a square edge. In Figure \ref{Uncoupled system}(a) and Figure \ref{Uncoupled system}(b), the discrete solutions,  $\Psi_h= (\Qvec_{D1}, \mathbf{M})$ and $\Psi_h= (\Qvec_{R4}, \mathbf{M})$, are plotted for $\ell=0.001, c=0$. Here $\Qvec_{D1}$ (resp. $\Qvec_{R4}$) is the D1 diagonal (resp. R4 rotated) solution with defects at vertices, and the corresponding nematic director $\mathbf{n}=(\cos\theta,\sin\theta)$ where $\theta = \frac{1}{2} \textrm{atan}\frac{Q_{D1,12}}{Q_{D1,11}}$ and $Q_{D1, 11} = \frac{|\Qvec_{D1}|}{\sqrt{2}} \cos 2 \theta, Q_{D1, 12} = \frac{|\Qvec_{D1}|}{\sqrt{2}}\sin 2\theta$ are the two independent components of $\Qvec_{D1}$. Analogous remarks apply to $\Qvec_{R4}$. $\Mvec$ labels the uncoupled magnetic profile with a $+1$-degree vortex at the square center consistent with topologically non-trivial boundary conditions.  The magnetization vector $\Mvec$ has a direction whereas the nematic director field, $\mathbf{n}$, is plotted without a direction since $\mathbf{n}$ and $-\mathbf{n}$ are physically equivalent.  Figure \ref{Uncoupled system}(c) (resp. Figure \ref{Uncoupled system}(d)) demonstrate the convergence history of the discrete solutions, computed using piecewise polynomials of degree $1$,  associated with D1 ( resp. R4) nematic solutions, in energy and $\mathbf{L}^2$ norms for the parameter values $\ell=0.001, c=0$ and $\ell=0.002, c=0$.  The order $1 $ convergence in energy norm and order $2 $ convergence in $\mathbf{L}^2$- norm are obtained for both sets of parameter values. The color bars
for nematic and magnetic profiles plot the values of $ s= \sqrt{Q_{11}^2+ Q_{12}^2}$  and $\abs{\mathbf{M}}=\sqrt{M_{1}^2+ M_{2}^2},$ respectively. The lines and arrows depict $\mathbf{n}$ and $\mathbf{M}$ respectively. Note that all subsequent discrete solution profiles,  $\Psi_h$,  have the nematic director field plot on the left and and magnetization
vector plot on the right.
\begin{figure}[H]
	\begin{minipage}[t]{0.31\linewidth}
		\subfloat[$\Qvec_{D1}$ and $\Mvec_{D1}$ profile]{\includegraphics[width=\textwidth]{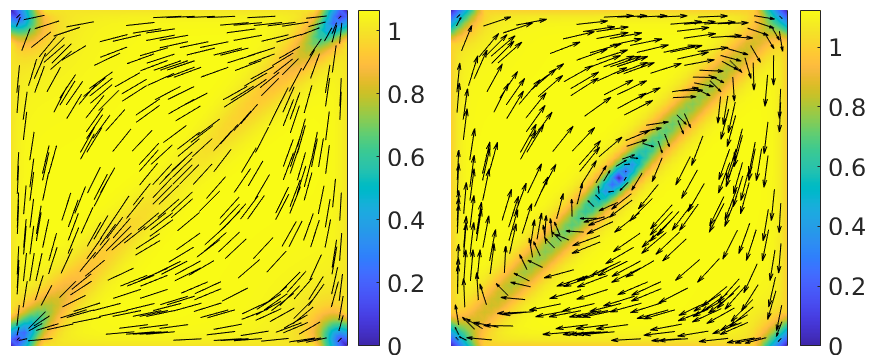}}
		\\
		\subfloat[$\Qvec_{D2}$ and $\Mvec_{D2}$ profile]{\includegraphics[width=\textwidth]{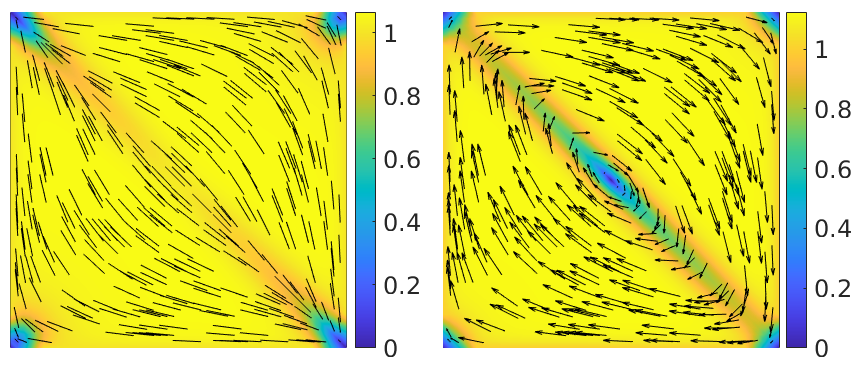}}	\end{minipage}\hspace{0.4 cm}
	\begin{minipage}[t]{0.31\linewidth}
		\subfloat[$\Qvec_{R1}$ and $\Mvec_{R1}$ profile]{\includegraphics[width=\textwidth]{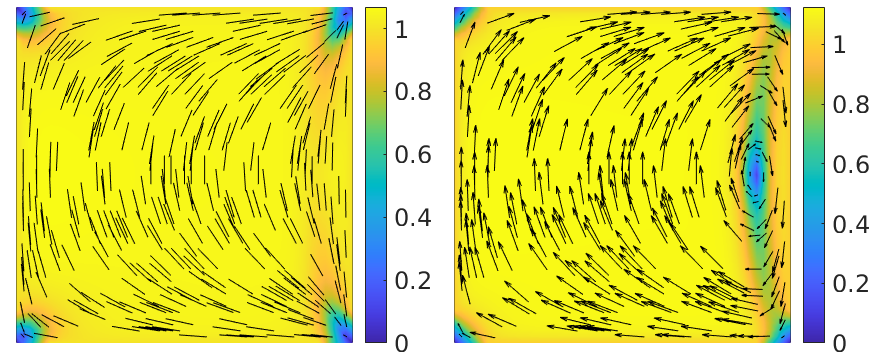}}
		\\			
		\subfloat[$\Qvec_{R2}$ and $\Mvec_{R2}$ profile]{\includegraphics[width=\textwidth]{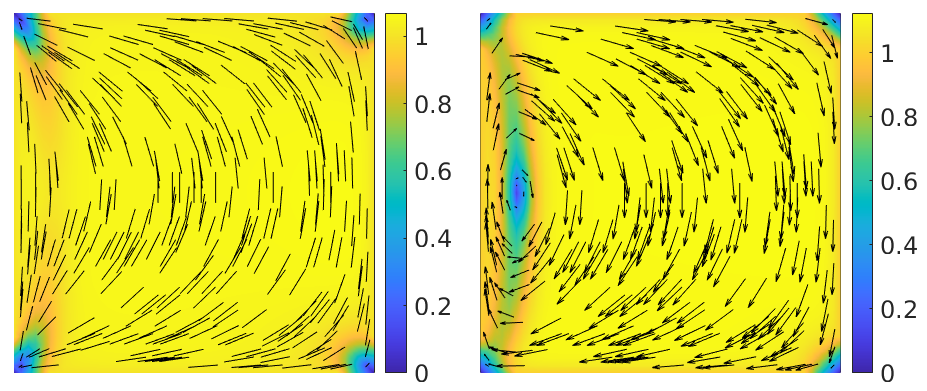}} 
	\end{minipage}\hspace{0.4 cm}
	\begin{minipage}[t]{0.31\linewidth}
		%	\centering
		\subfloat[$\Qvec_{R3}$ and $\Mvec_{R3}$ profile]{\includegraphics[width=\textwidth]{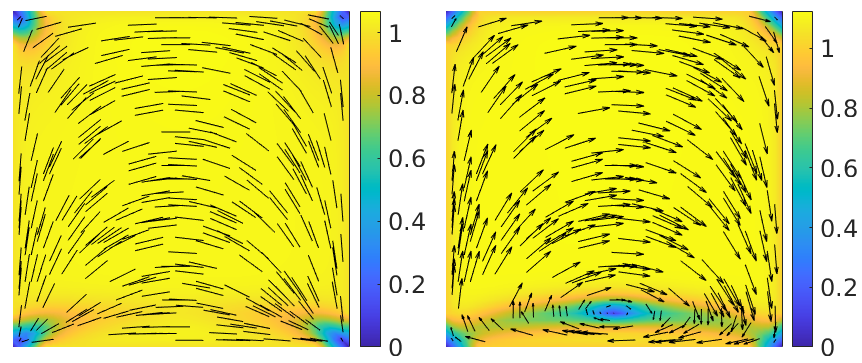}}
		\\
		\subfloat[$\Qvec_{R4}$ and $\Mvec_{R4}$ profile]{\includegraphics[width=\textwidth]{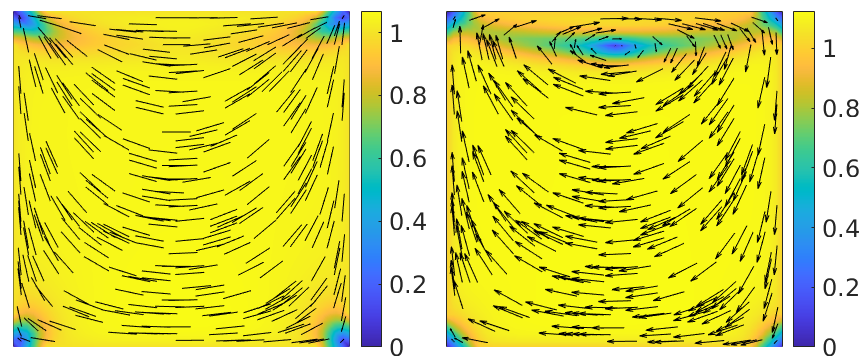}}	\end{minipage}
	\caption{Nematic $\Qvec$ and magnetic $\Mvec$ configurations for $\ell =0.001$ and $c=0.25.$ Left column: solution profiles $\Psi_h=(\Qvec_{D1},\Mvec_{D1})$ (top) and $\Psi_h=(\Qvec_{D2},\Mvec_{D2})$ (bottom) corresponding to diagonal D1 and D2 nematic stable solutions, respectively; Middle column: solution profiles $\Psi_h=(\Qvec_{R1},\Mvec_{R1})$ (top) and $\Psi_h=(\Qvec_{R2},\Mvec_{R2})$ (bottom) corresponding to rotated R1 and R2 nematic stable solutions, respectively; Right column: solution profiles $\Psi_h=(\Qvec_{R3},\Mvec_{R3})$ (top) and $\Psi_h=(\Qvec_{R4},\Mvec_{R4})$ (bottom) corresponding to rotated R3 and R4 nematic stable solutions, respectively.}
	\label{Coulped_systeM_a_positive}
\end{figure}
\noindent \textbf{Case II: the coupling parameter $\mathbf{c>0}$}

\medskip

\noindent Figure~\ref{Coulped_systeM_a_positive} plots the numerically computed stable solution profiles for parameter values $\ell=0.001,  c=0.25$, represented by $\Psi_h=(\Qvec_{D1},\Mvec_{D1}),$ $\Psi_h=(\Qvec_{D2},\Mvec_{D2}),$ $\Psi_h=(\Qvec_{R1},\Mvec_{R1}),$ $\Psi_h=(\Qvec_{R2},\Mvec_{R2}),$ $\Psi_h=(\Qvec_{R3},\Mvec_{R3}),$ $\Psi_h=(\Qvec_{R4},\Mvec_{R4}),$ corresponding to D1, D2 diagonal and the R1, R2, R3, R4 rotated stable nematic equillibria, respectively. 
For small $\ell$, and $c>0,$ the coupling energy favors the co-alignment of $\mathbf{n}$ and $\Mvec$, i.e., $\mathbf{n}\cdot \Mvec= \pm 1.$ In this case, the nematic profiles (both diagonal and rotated) do not exhibit any interior vortices whereas the magnetic profiles develop an interior line of reduced $|\mathbf{M}|$, analogous to a domain wall, smeared out along the square diagonals/ near one of the square edges. The magnetic  profiles,  $\Mvec_{D1}, \Mvec_{D2}$ (resp. $\Mvec_{R1}, \Mvec_{R2}$, $\Mvec_{R3}, \Mvec_{R4}$) exhibit $\pi$-walls \cite{Hanetal2021Ferro2D} along the square diagonals  $y=x$ for D1 and along $y=-x+1$ for D2 (resp. along the square edges $x=1$ for R1, $x=0$ for R2, $y=0$ for R3, $y=1$ for R4) nematic solutions. The domain walls are created to ensure the compatibility between the angle constraint in \eqref{angle constraint for c>0}, the condition necessary to be minimizer for $c>0$, and the tangent boundary conditions. Recall that for $\mathbf{n}=(\cos \theta, \sin \theta)$, $Q_{11}=s\cos 2\theta$ and $Q_{12}=s\sin 2\theta.$ For a stable stationary point $\Psi=(s\cos 2\theta,s\sin 2\theta, \abs{\Mvec}\cos \varphi,\abs{\Mvec}\sin \varphi)$ with magnetization angle $\varphi$,  \eqref{angle constraint for c>0} implies that $$\theta \approx  \varphi +k\pi,\, k \in \mathbb{Z}$$ almost  everywhere in the domain interior , for sufficiently small values of $\ell$. 
\begin{figure}[H]
	\centering
	\subfloat[]{\includegraphics[width=4cm,height=3.3cm]{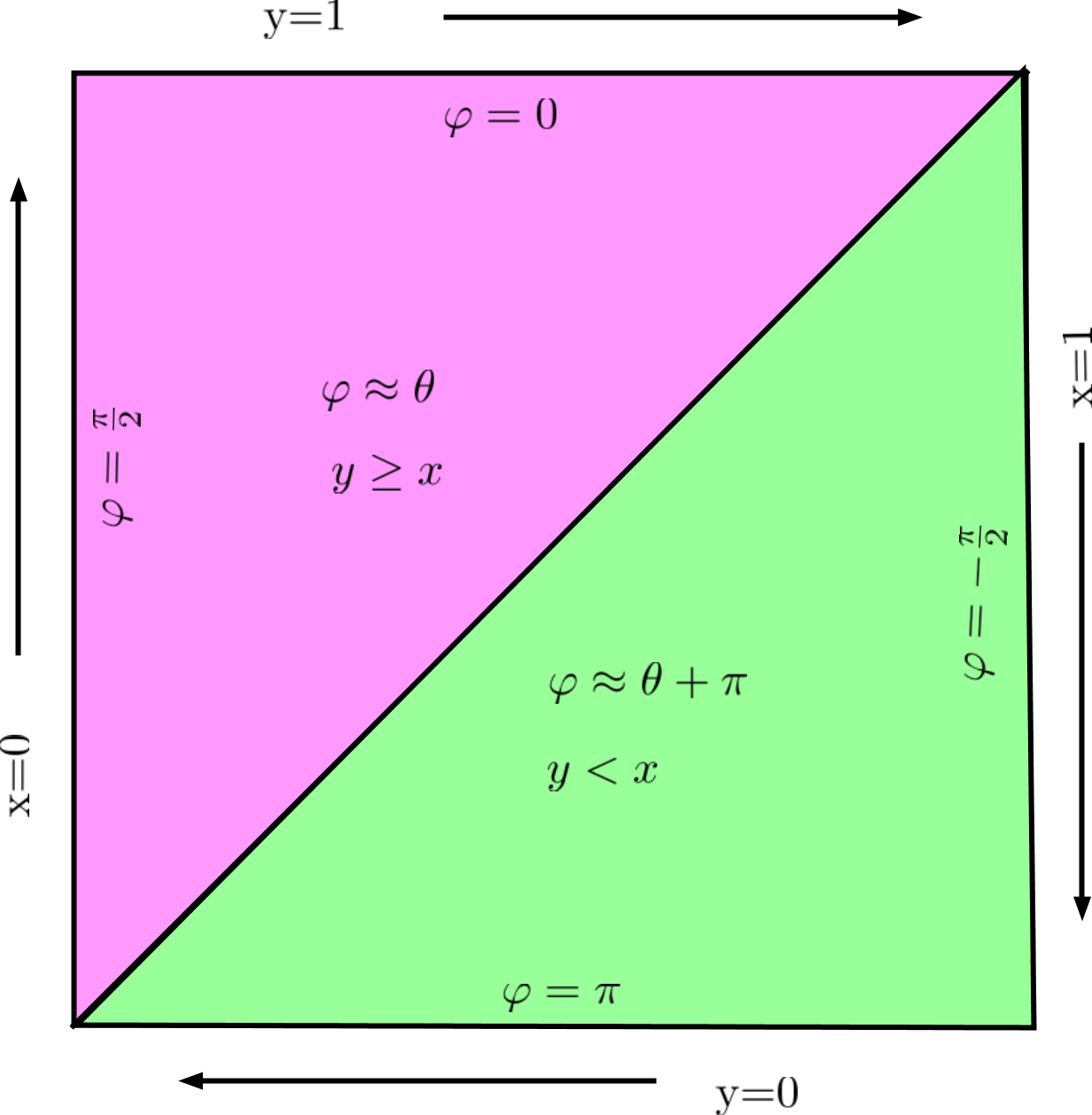}} \hspace{2 cm}
	\subfloat[]{\includegraphics[width=4cm,height=3.3cm]{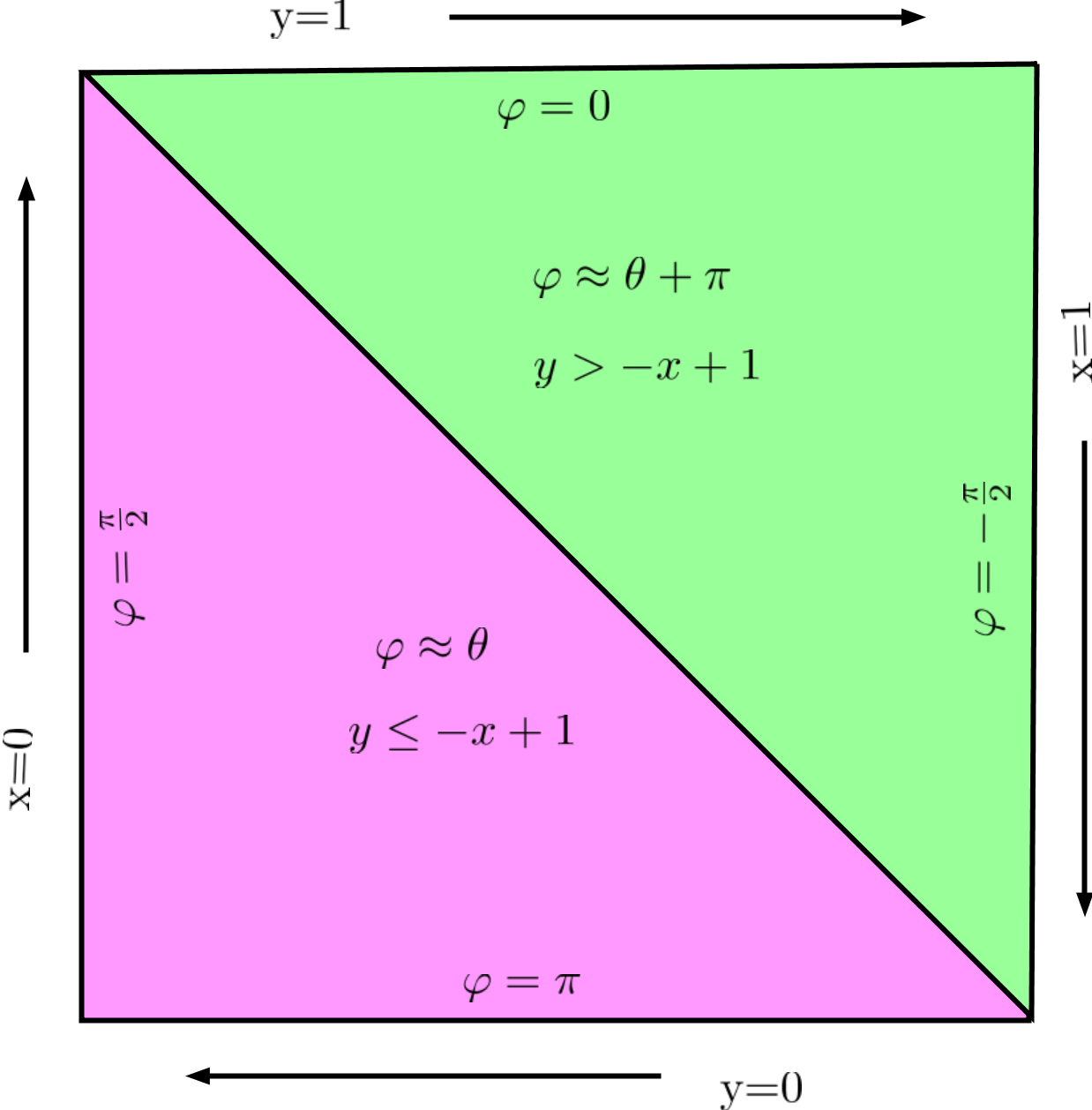}}
	\caption{Domain wall formation of magnetic profiles (a) $\Mvec_{D1}$ and (b) $\Mvec_{D2}$.}
	\label{Domain wall for D1 D2}
\end{figure}
\begin{figure}[H]
	\centering
	\subfloat[]{\includegraphics[width=6.3cm,height=5cm]{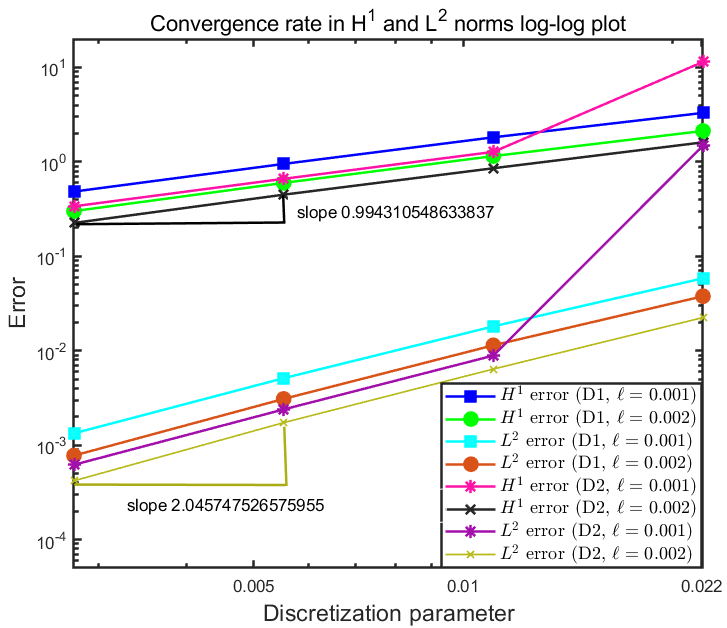}\label{D1}} \hspace{0.6 cm}
	\subfloat[]{\includegraphics[width=6.3cm,height=5cm]{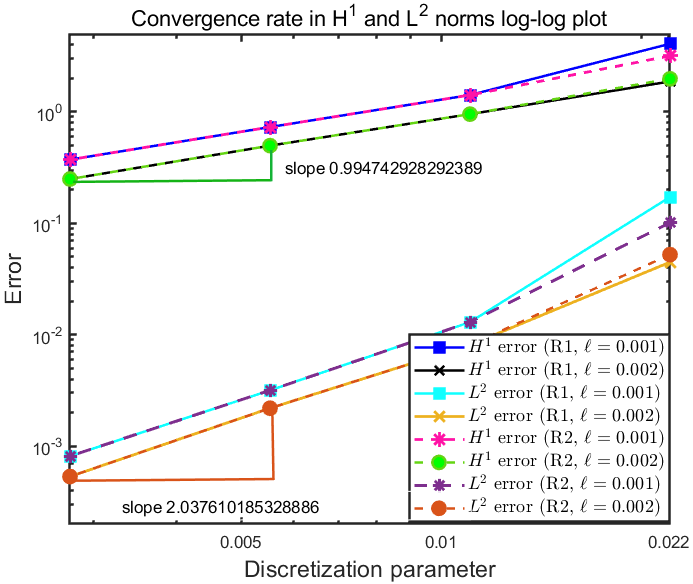}\label{R1}} \\
	\subfloat[]{\includegraphics[width=6.3cm,height=5cm]{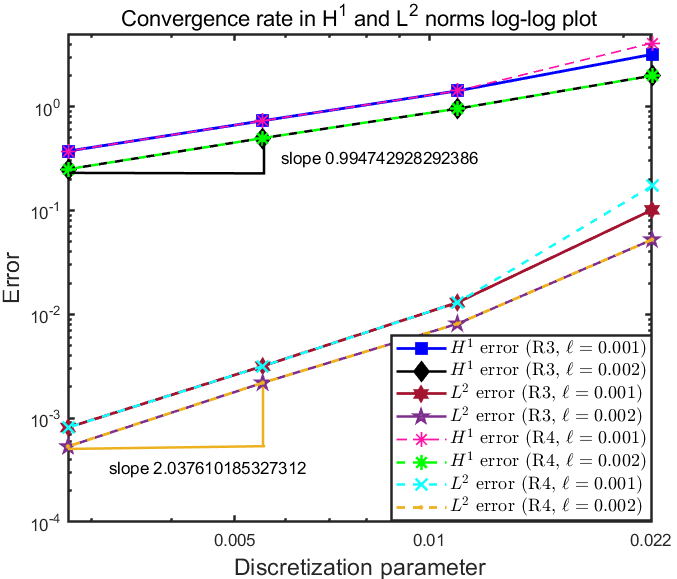}\label{R3}} 
	\caption{Energy and  $\mathbf{L}^2$ norm error versus discretization parameter $h$ plots for the discrete solutions (a) $\Psi_h=(\Qvec_{D1},\Mvec_{D1})$ and  $\Psi_h=(\Qvec_{D2},\Mvec_{D2})$, (b) $\Psi_h=(\Qvec_{R1},\Mvec_{R1})$ and $\Psi_h=(\Qvec_{R2},\Mvec_{R2})$, (c) $\Psi_h=(\Qvec_{R3},\Mvec_{R3})$ and  $\Psi_h=(\Qvec_{R4},\Mvec_{R4}),$ for the two sets of parameter values  $\ell =0.001 ,c=0.25,$ and $\ell =0.002, c=0.25,$}
	\label{convergence rates for c positive}
\end{figure}
%\\
%For a stable stationary point $\Psi=(s\cos 2\theta,s\sin 2\theta, \abs{\Mvec}\cos \varphi,\abs{\Mvec}\sin \varphi)$ with magnetization angle $\varphi$,  \eqref{angle constraint for c>0} requires that
% $$\theta= \varphi +k\pi,\, k\in \mathbb{Z}$$
%  in the domain interior, which is reflected  in the magnetic profiles of the discrete solutions as is illustrated in Figure \ref{Domain wall for D1 D2}. 
\noindent The tangent boundary condition for D1 ( resp. D2) nematic profile is coded in the boundary conditions: $\theta=0$ along $y=0,$  $\theta=0$ along $y=1,$  $\theta=\frac{\pi}{2}$ along $x=0,$  $\theta=\frac{\pi}{2}$ along $x=1$ (resp. $\theta=\pi$ along $y=0,$  $\theta=\pi$ along $y=1,$  $\theta=\frac{\pi}{2}$ along $x=0,$  $\theta=\frac{\pi}{2}$ along $x=1$). Figure \eqref{Domain wall for D1 D2} shows that $\theta \approx \varphi $ for  $y\geq x$ (resp.   $y\leq -x+1$) and  $\theta \approx \varphi +\pi $ for  $y< x$ (resp.   $y > -x+1$) in $\Mvec_{D1}$ (resp. $\Mvec_{D2}$) profile. The domain walls for $\Mvec_{R1}, \Mvec_{R2}$, $\Mvec_{R3}, \Mvec_{R4}$ can be interpreted similarly. Figure \ref{convergence rates for c positive} illustrates the numerical errors and orders of convergence,  computed using piecewise polynomials of degree $1$,  for the discrete solutions (a) $\Psi_h=(\Qvec_{D1},\Mvec_{D1})$ and $\Psi_h=(\Qvec_{D2},\Mvec_{D2})$, (b) $\Psi_h=(\Qvec_{R1},\Mvec_{R1})$ and $\Psi_h=(\Qvec_{R2},\Mvec_{R2})$, (c) $\Psi_h=(\Qvec_{R3},\Mvec_{R3})$ and  $\Psi_h=(\Qvec_{R4},\Mvec_{R4})$, respectively, in energy and $\mathbf{L}^2$ norms for the parameter values $\ell=0.001, c=0.25$ and $\ell=0.002, c=0.25$. The convergence rates obtained in energy and $\mathbf{L}^2$ norms are of  $\mathcal{O}(h)$ and $\mathcal{O}(h^2)$, respectively.
\begin{figure}[H]
	\centering
	\begin{minipage}[t]{0.31\linewidth}
		%	\centering
		\subfloat[$\Qvec_{D1}^1$ and $\Mvec^1_{D1}$ profile]{\includegraphics[height=1.9cm, width=4.6cm]{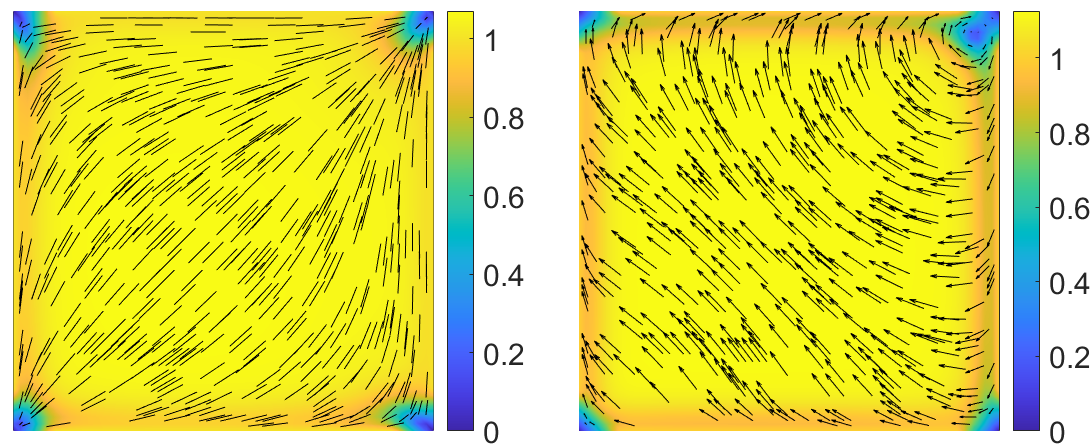}}
	\\
		\subfloat[$\Qvec_{D1}^2$ and $\Mvec^2_{D1}$ profile]{\includegraphics[height=1.9cm, width=4.6cm]{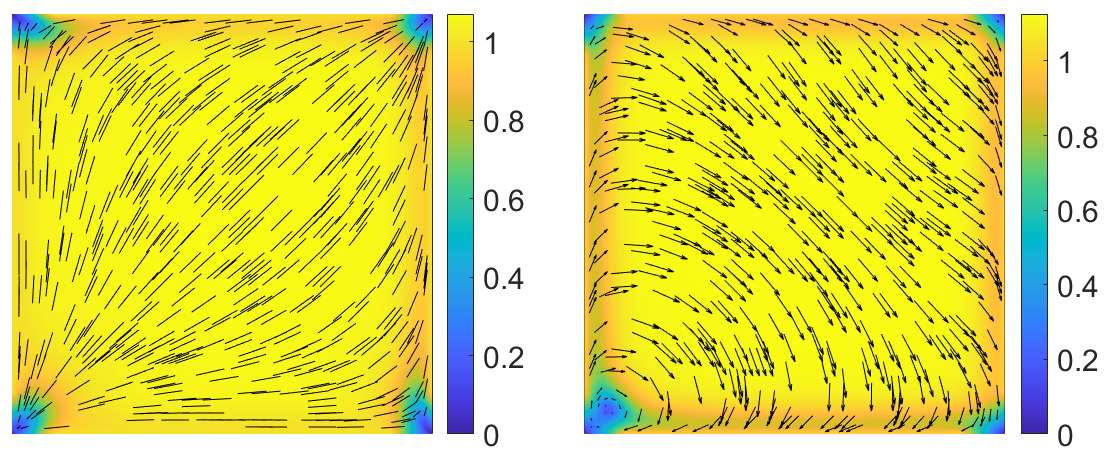}}
	\end{minipage}\hspace{0.4 cm}
	\begin{minipage}[t]{0.31\textwidth}
		\subfloat[$\Qvec^1_{D2}$ and $\Mvec^1_{D2}$ profile]{\includegraphics[height=1.9cm, width=4.6cm]{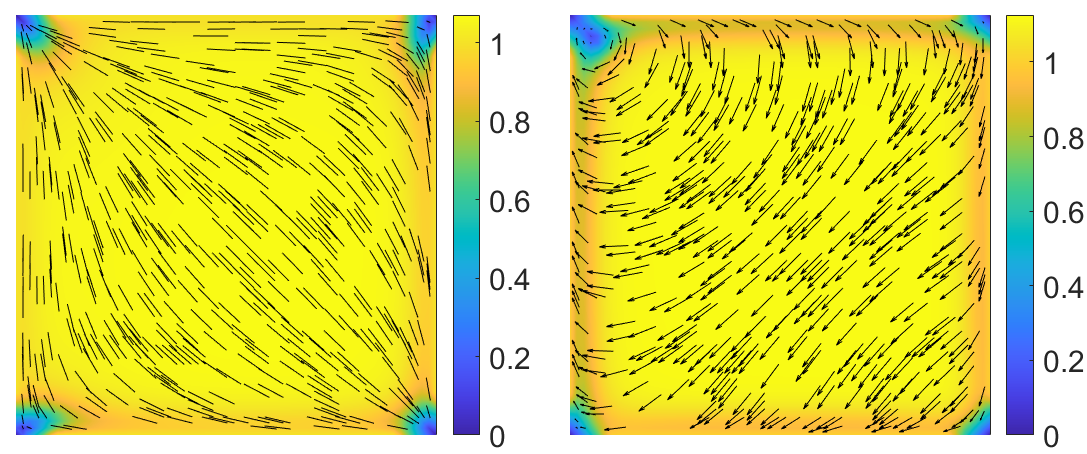}\label{Fig3cD2a}} 
		\\			
		\subfloat[$\Qvec^2_{D2}$ and $\Mvec^2_{D2}$ profile]{\includegraphics[height=1.9cm, width=4.6cm]{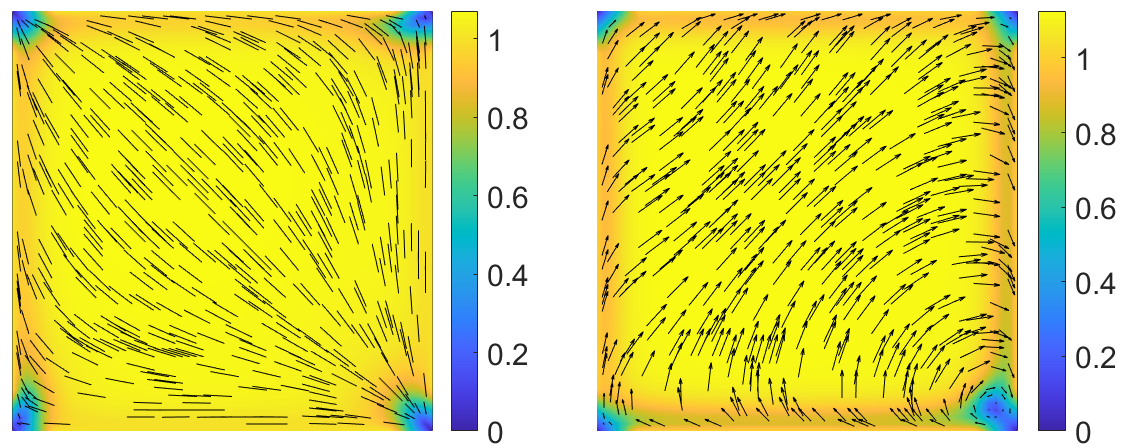}\label{Fig3cD2b}} 				
	\end{minipage} \hspace{0.2 cm}
	\begin{minipage}[t]{0.31\linewidth}
		\centering	
		\subfloat[$\Qvec^1_{R1}$ and $\Mvec^1_{R1}$ profile]{\includegraphics[height=1.9cm, width=4.6cm]{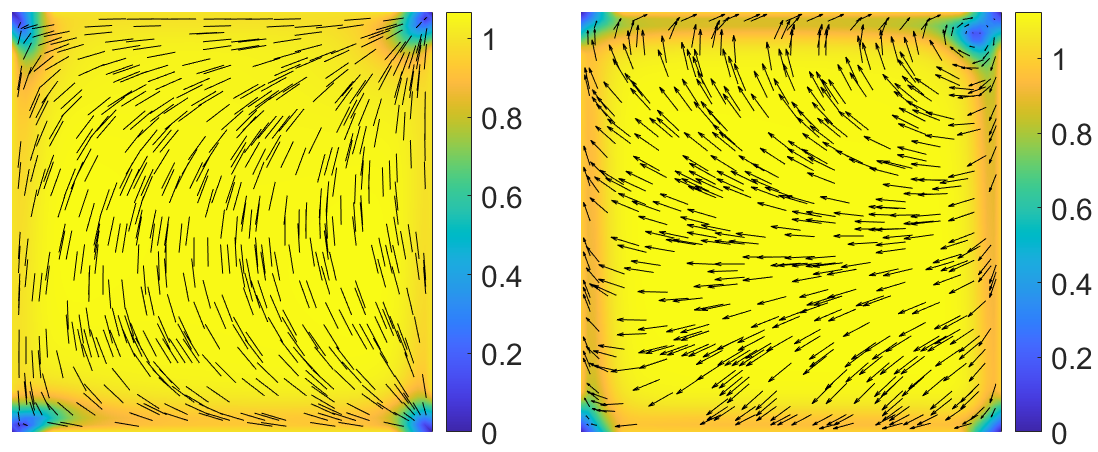}\label{Fig3dR1a}}
\\
		\subfloat[$\Qvec^2_{R1}$ and $\Mvec^2_{R1}$ profile]{\includegraphics[height=1.9cm, width=4.6cm]{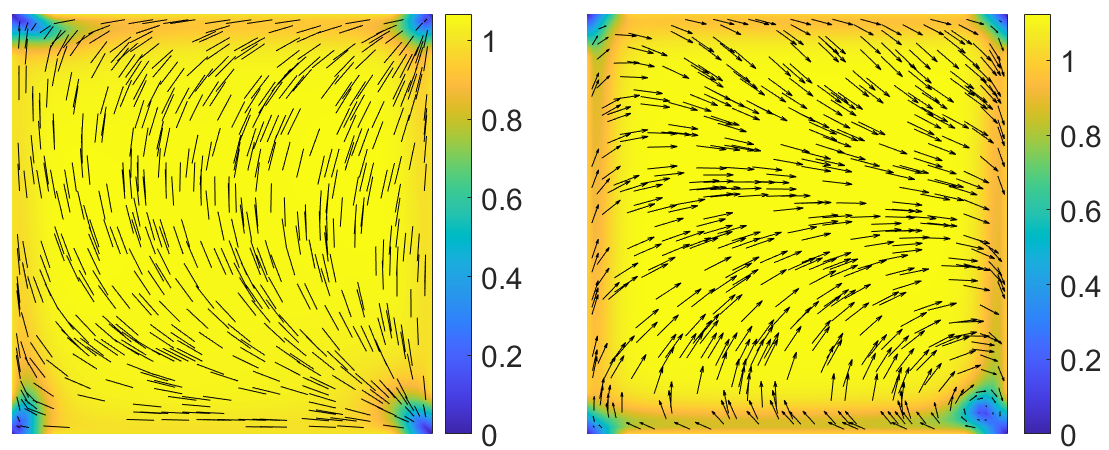}\label{Fig3dR2b}}	
		\vspace{0.3 cm}
	\end{minipage}
	\caption{Nematic $\Qvec$ and magnetic $\Mvec$ configurations for $\ell =0.001$ and $c=-0.25.$ Left column: two solution profiles $\Psi_h=(\Qvec^1_{D1},\Mvec^1_{D1})$ and $\Psi_h=(\Qvec^2_{D1},\Mvec^2_{D1})$ corresponding to diagonal D1 nematic stable solution; Middle column: two solution profiles $\Psi_h=(\Qvec^1_{D2},\Mvec^1_{D2})$ and $\Psi_h=(\Qvec^2_{D2},\Mvec^2_{D2})$ corresponding to diagonal D2 nematic stable solution; Right column: two solution profiles $\Psi_h=(\Qvec^1_{R1},\Mvec^1_{R1})$ and $\Psi_h=(\Qvec^2_{R1},\Mvec^2_{R2})$ corresponding to rotated R1 nematic stable solution.}
\label{Figure_c_negative_couplingPart1}
\end{figure}
\begin{figure}[H]
	\centering
	\begin{minipage}[t]{0.31\linewidth}
	%	\centering
	\subfloat[$\Qvec^1_{R2}$ and $\Mvec^1_{R2}$ profile]{\includegraphics[height=1.95cm, width=4.7cm]{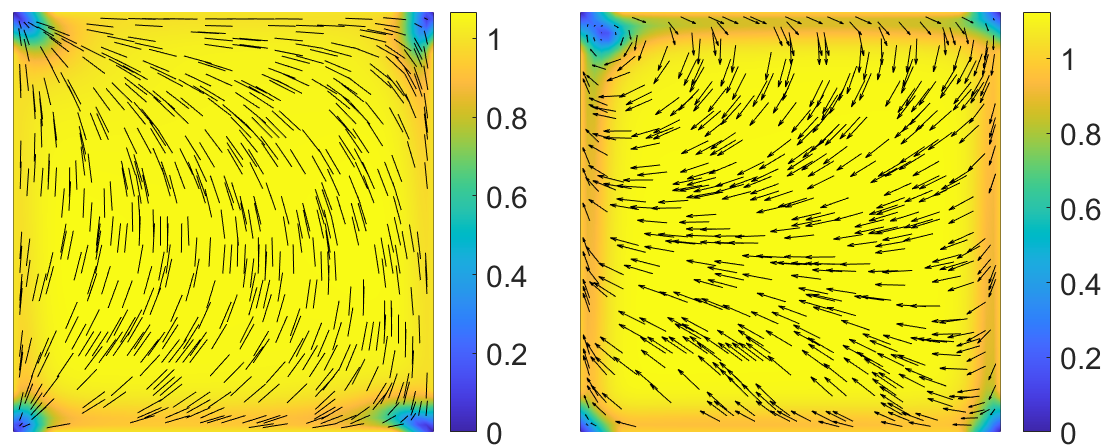}}
	\\
	\subfloat[$\Qvec^2_{R2}$ and $\Mvec^2_{R2}$ profile]{\includegraphics[height=1.95cm, width=4.6cm]{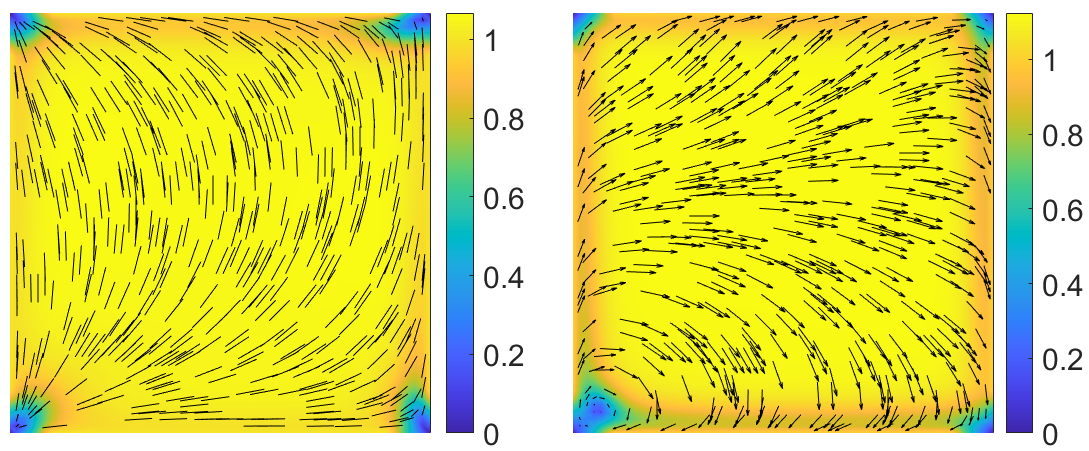}}
\end{minipage}\hspace{0.4 cm}
\begin{minipage}[t]{0.31\textwidth}
	\subfloat[$\Qvec^1_{R3}$ and $\Mvec^1_{R3}$ profile]{\includegraphics[height=1.9cm, width=4.6cm]{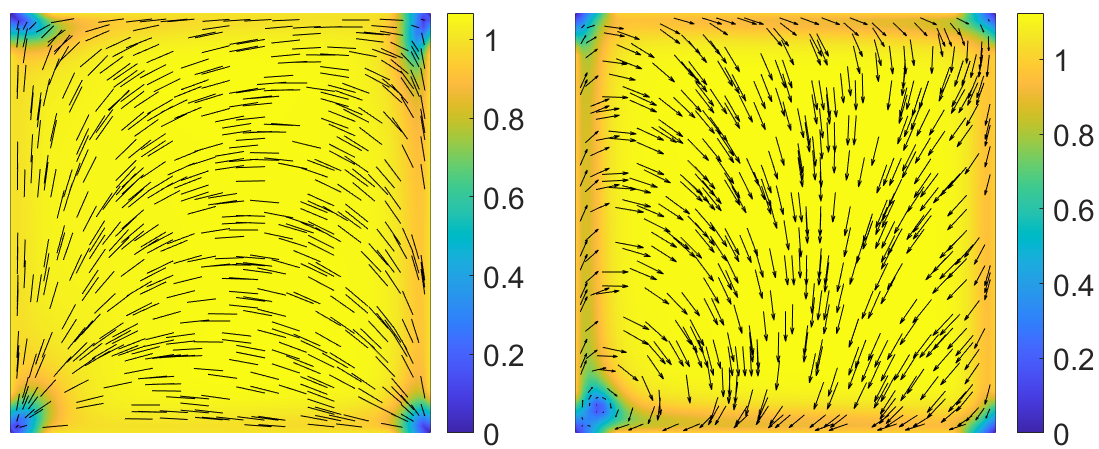}\label{Fig3cD2a1}} 
	\\			
	\subfloat[$\Qvec^2_{R3}$ and $\Mvec^2_{R3}$ profile]{\includegraphics[height=1.9cm, width=4.6cm]{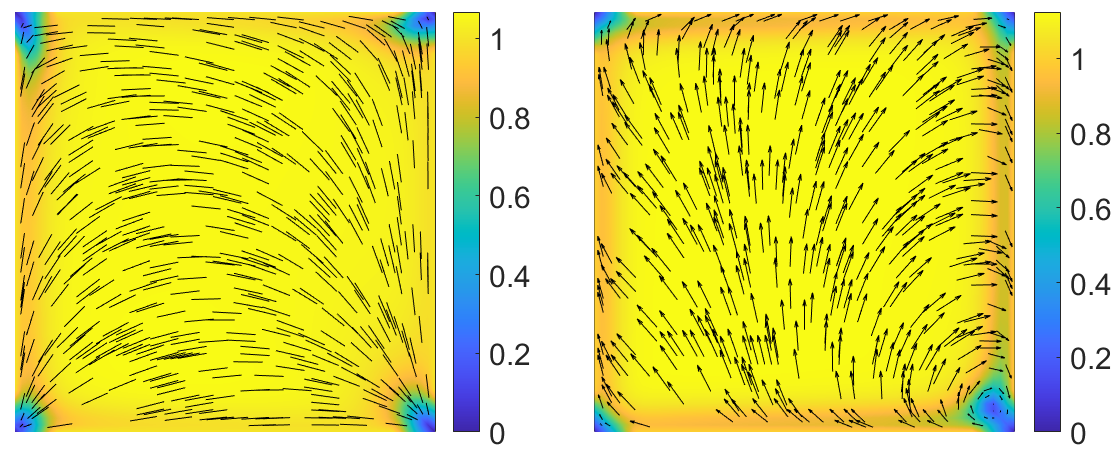}\label{Fig3cD2b1}} 				
\end{minipage}\hspace{0.4 cm}
\begin{minipage}[t]{0.31\linewidth}
	\centering	
	\subfloat[$\Qvec^1_{R4}$ and $\Mvec^1_{R4}$ profile]{\includegraphics[height=1.9cm, width=4.6cm]{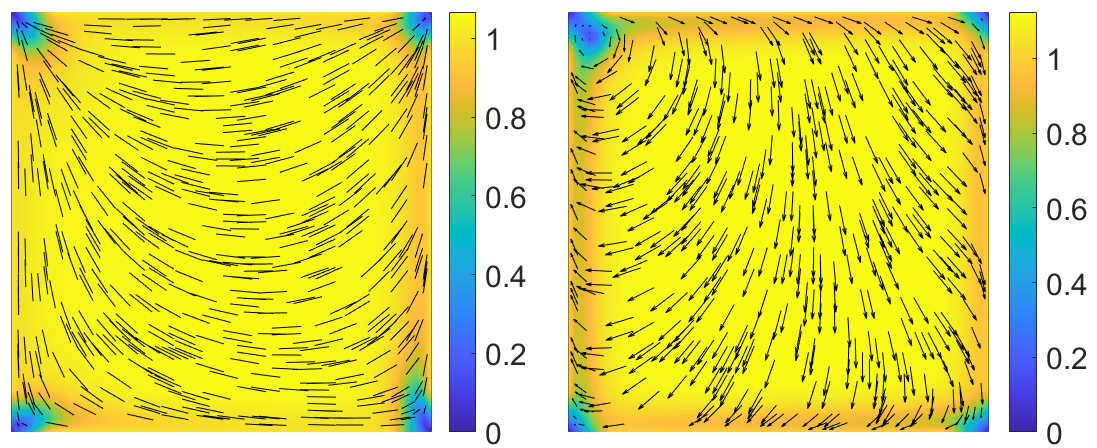}\label{Fig3dR1a1}}
	\\
	\subfloat[$\Qvec^2_{R4}$ and $\Mvec^2_{R4}$ profile]{\includegraphics[height=1.9cm, width=4.6cm]{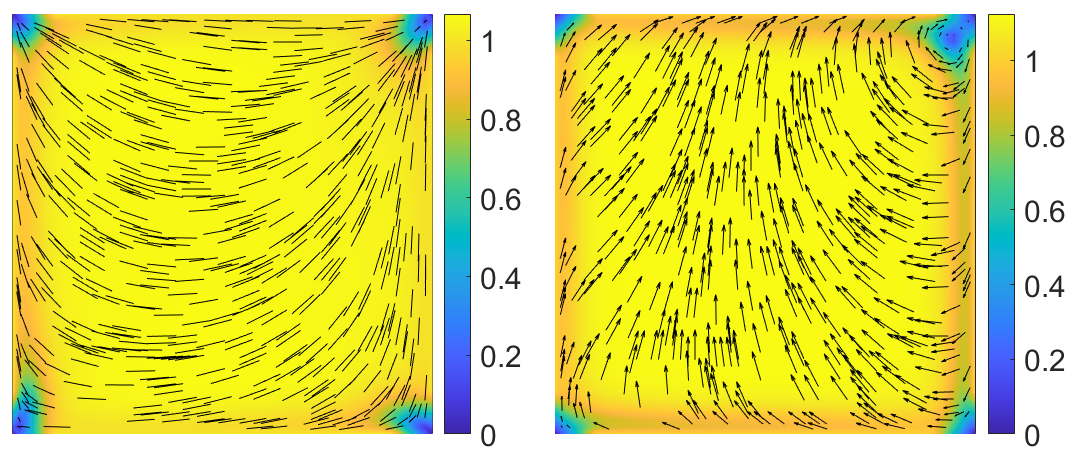}\label{Fig3dR2b1}}	
\end{minipage}
		\caption{Nematic $\Qvec$ and magnetic $\Mvec$ configurations for $\ell =0.001$ and $c=-0.25.$ Left column: two solution profiles $\Psi_h=(\Qvec^1_{R2},\Mvec^1_{R2})$ and $\Psi_h=(\Qvec^2_{R2},\Mvec^2_{R2})$ corresponding to rotated R2 nematic stable solution; Middle column: two solution profiles $\Psi_h=(Q^1_{R3},\Mvec^1_{R3})$ and $\Psi_h=(\Qvec^2_{R3},\Mvec^2_{R3})$ corresponding to rotated R3 nematic stable solution; Right column: two solution profiles $\Psi_h=(\Qvec^1_{R4},\Mvec^1_{R4})$ and $\Psi_h=(\Qvec^2_{R4},\Mvec^2_{R4})$ corresponding to rotated R4 nematic stable solution.}
	\label{Figure_c_negative_coupling}
\end{figure}
\noindent \textbf{Case III: the coupling parameter $c<0$}

\medskip

\noindent Now, we discuss the discrete solution profiles for negative coupling, which favours perpendicular alignment of $\mathbf{n}$ and $\Mvec$, i.e, $\mathbf{n} \cdot \Mvec =0,$  for the parameter values $\ell =0.001 $ and $c=-0.25.$ For diagonal nematic solutions (resp. rotated solutions), the symmetry between the diagonally opposite (resp. square edge) splay vertices is broken, so that there are $4$ distinct diagonal solutions. By similar reasoning, there are $8$ distinct rotated solutions, so that the number of stable admissible equilibria is doubled.
\begin{figure}[H]
	\centering
	%[width=5.1cm,height=5cm]
	\subfloat[]{\includegraphics[width=6.3cm,height=4.9cm]{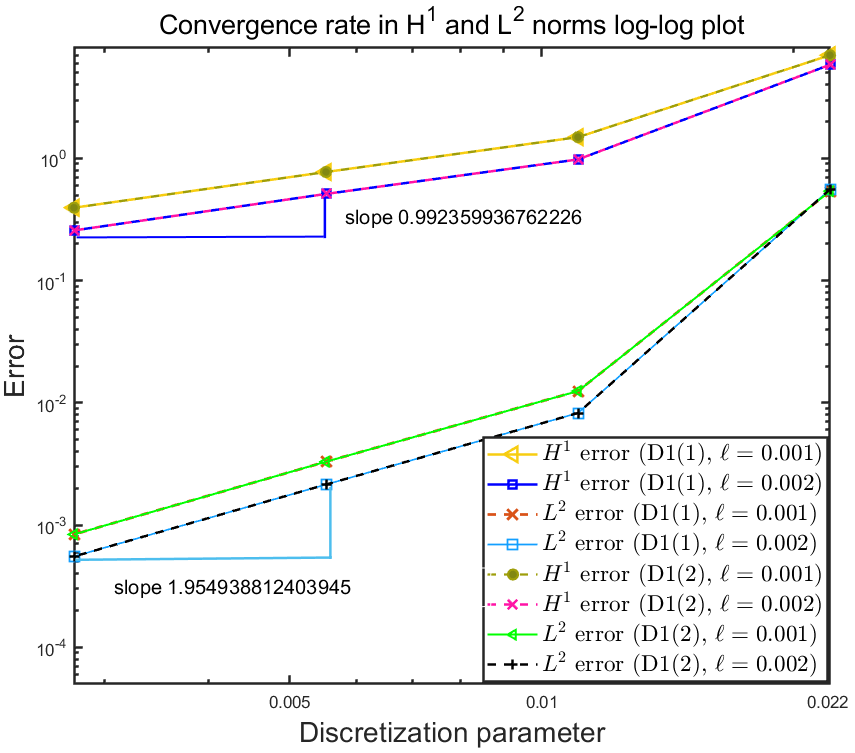}\label{D1cnegative}} \hspace{0.6 cm}
	\subfloat[]{\includegraphics[width=6.2cm,height=4.9cm]{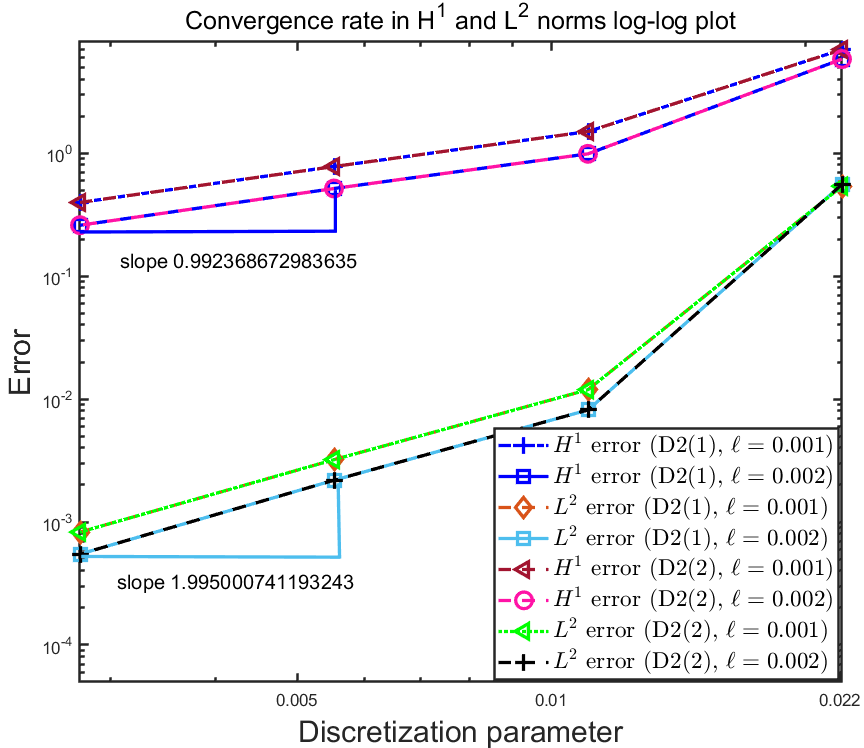}\label{D2cnegative}} 
	\caption{Energy and  $\mathbf{L}^2$ norm errors versus discretization parameter $h$ plots for the discrete solutions (a) $\Psi_h=(\Qvec_{D1}^1,\Mvec_{D1}^1)$ and $\Psi_h=(\Qvec_{D1}^2,\Mvec_{D1}^2)$, (b) $\Psi_h=(\Qvec_{D2}^1,\Mvec_{D2}^1)$ and  $\Psi_h=(\Qvec_{D2}^2,\Mvec_{D2}^2)$, for two sets of parameter values  $\ell =0.001 ,c=-0.25,$ and $\ell =0.002 ,c=-0.25.$}
	\label{convergence rates for c negative}
\end{figure}
\begin{figure}[H]
	\centering
	\subfloat[]{\includegraphics[width=6.3cm,height=4.9cm]{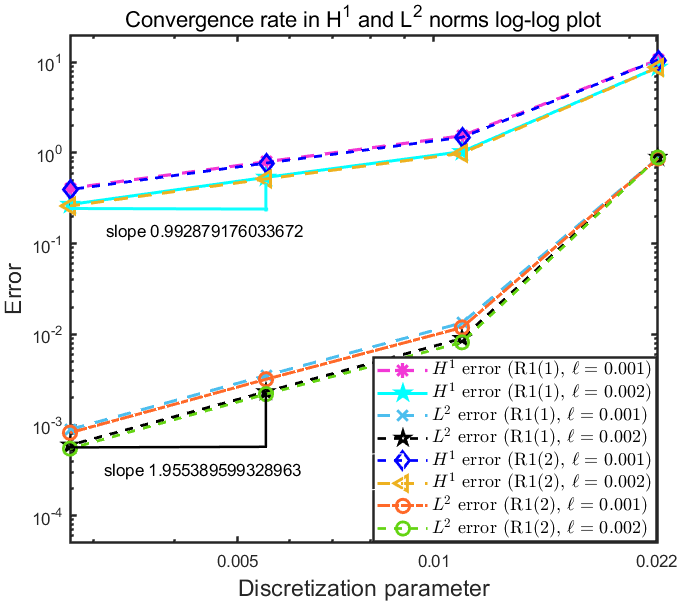}\label{R1cnegative}} \hspace{0.6 cm}		
	%	\\
	\subfloat[]{\includegraphics[width=6.3cm,height=4.9cm]{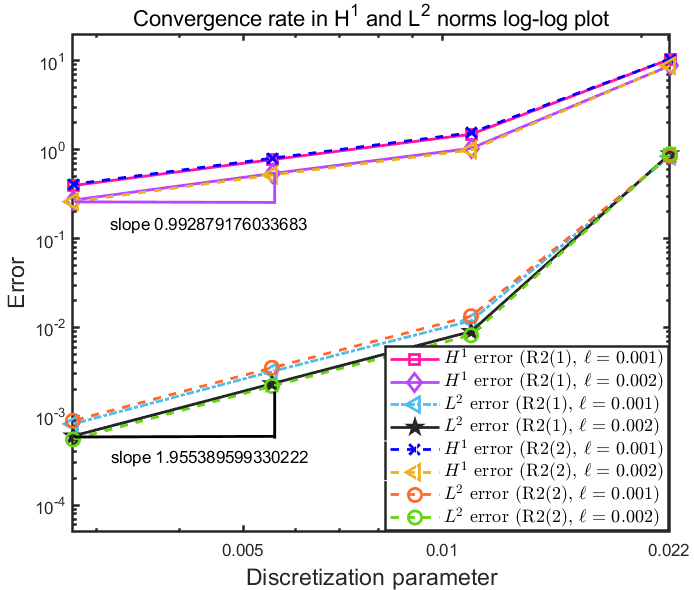}\label{R2cnegative}} 
	\caption{Energy and  $\mathbf{L}^2$ norm errors versus discretization parameter $h$ plots for the discrete solutions  (a) $\Psi_h=(\Qvec_{R1}^1,\Mvec_{R1}^1)$ and  $\Psi_h=(\Qvec_{R1}^2,\Mvec_{R1}^2)$, (b) $\Psi_h=(\Qvec_{R2}^1,\Mvec_{R2}^1)$ and  $\Psi_h=(\Qvec_{R2}^2,\Mvec_{R2}^2)$,  for two sets of parameter values  $\ell =0.001 ,c=-0.25,$ and $\ell =0.002 ,c=-0.25.$}
	\label{convergence rates for c negativeR1R2}
\end{figure}
\begin{figure}[H]
	\centering
	\subfloat[]{\includegraphics[width=6.2cm,height=4.9cm]{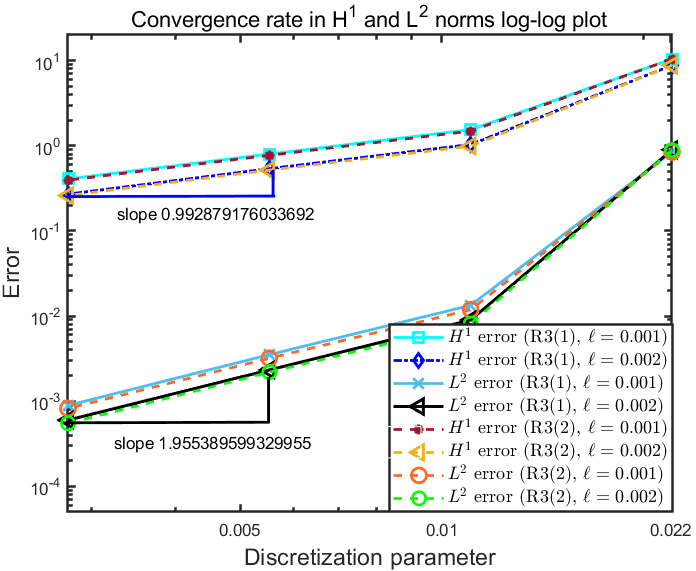}\label{R3cnegative}} \hspace{0.6 cm}
	\subfloat[]{\includegraphics[width=6.2cm,height=4.9cm]{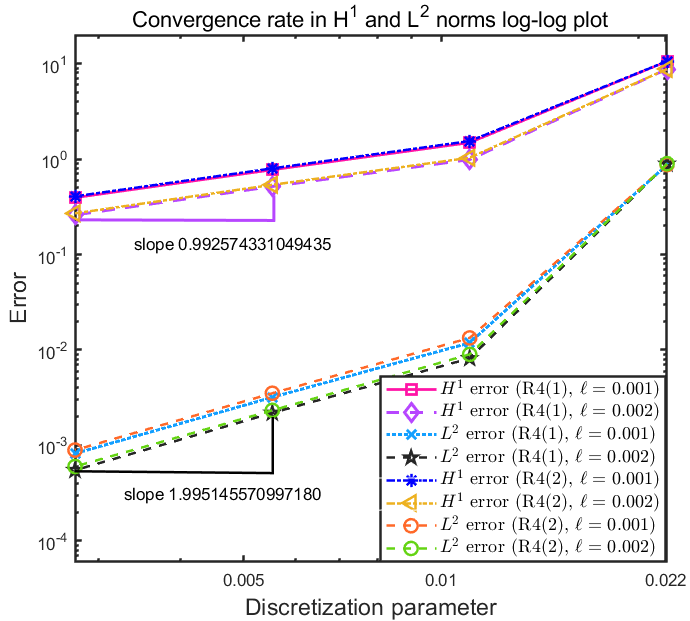}\label{R4cnegative}} 
	\caption{Energy and  $\mathbf{L}^2$ norm errors versus discretization parameter $h$ plots for the discrete solutions,  (a) $\Psi_h=(\Qvec_{R3}^1,\Mvec_{R3}^1)$ and  $\Psi_h=(\Qvec_{R3}^2,\Mvec_{R3}^2)$, (b) $\Psi_h=(\Qvec_{R4}^1,\Mvec_{R4}^1)$ and $\Psi_h=(\Qvec_{R4}^2,\Mvec_{R4}^2),$ for two sets of parameter values  $\ell =0.001 ,c=-0.25,$ and $\ell =0.002 ,c=-0.25.$}
	\label{convergence rates for c negativeR3R4}
\end{figure}
\noindent For instance, $\Psi_h=(\Qvec_{D1}^1,\Mvec_{D1}^1)$ and $\Psi_h=(\Qvec_{D1}^2,\Mvec_{D1}^2) $ in  Figure \ref{Figure_c_negative_couplingPart1}  are two distinct stable, numerically computed solutions corresponding to standard D1 diagonal nematic profile (for $c=0$). Similarly, there are ten pairs of distinct stable  solution  profiles,  corresponding to the standard D2, R1, R2, R3 and R4 profiles; see Figures \ref{Figure_c_negative_couplingPart1} and \ref{Figure_c_negative_coupling}. The numerical errors and orders of convergence of the discrete solutions associated with the diagonal (D1, D2), and rotated (R1, R2) and (R3, R4) nematic equilibria are plotted in  Figures \ref{convergence rates for c negative},  \ref{convergence rates for c negativeR1R2} and \ref{convergence rates for c negativeR3R4}, respectively, for two sets of parameter values $\ell=0.001,  c=-0.25$ and $\ell=0.002,  c=-0.25$.  The convergence rates in energy and $\mathbf{L}^2$ norms are noted to be of order, $O(h)$ and $O(h^2)$, respectively.
\medskip

\noindent \textbf{Parameter dependent plots}
\medskip

\noindent Figure \ref{convergence_rates_parameter_dependency} (resp. Figure \ref{convergence_rates_parameter_dependency_c_negative}) presents discretization parameter $h$ versus energy and $\mathbf{L}^2$ norm error plots, for various values of  $\ell$ and  positive coupling (resp. negative coupling) parameter, for the discrete solution corresponding to D1 diagonal nematic equillibria. We observe that both the energy and $\mathbf{L}^2$ norm errors are sensitive to the choice of the small parameter $\ell$, for both instances of positive and negative nemato-magnetic coupling. For instance, fix $h=10^{-2},$ the energy norm error for $\ell=0.01$ at $h=10^{-2}$ is smaller than the error for $\ell=0.006$ and similarly, the error increases as $\ell$ further decreases.
\begin{figure}[H]
	\centering
	\subfloat[]{\includegraphics[width=7.3cm,height=6cm]{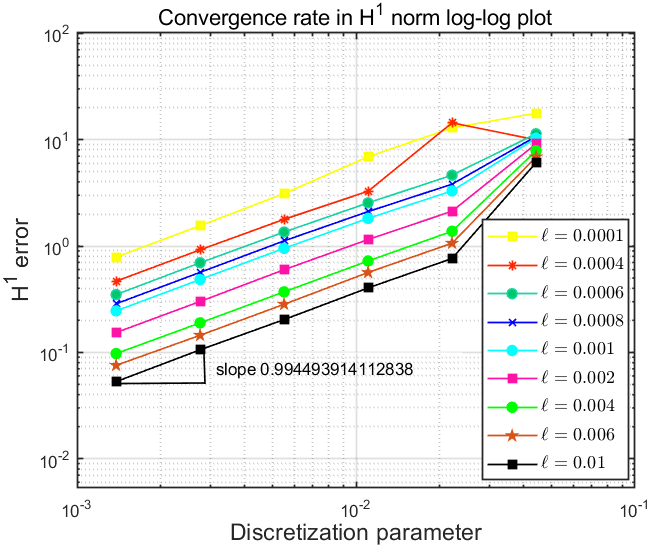}} 
		\hspace{0.2 cm}
	\label{convergence_rates_parameter_dependencyH1}
	\subfloat[]{\includegraphics[width=7.2cm,height=6cm]{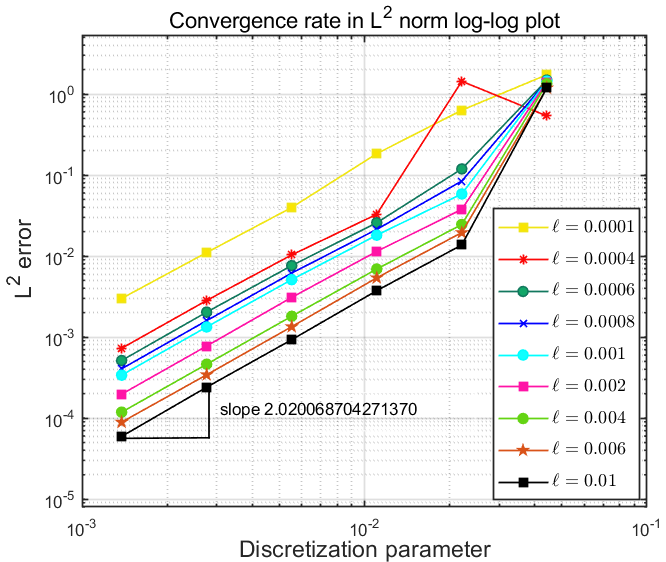}} 
	\caption{Convergence behavior plots of error in the (a) energy norm and (b)  $\mathbf{L}^2$ norm  versus the discretization parameter $h$ for $\Psi_h:=(\Qvec_{D1}, \Mvec_{D1})$ solution for various values of $\ell$ and $c=0.25.$}
		\label{convergence_rates_parameter_dependency}
\end{figure}
\begin{figure}[H]
	\centering
	\subfloat[]{\includegraphics[width=7.3cm,height=6cm]{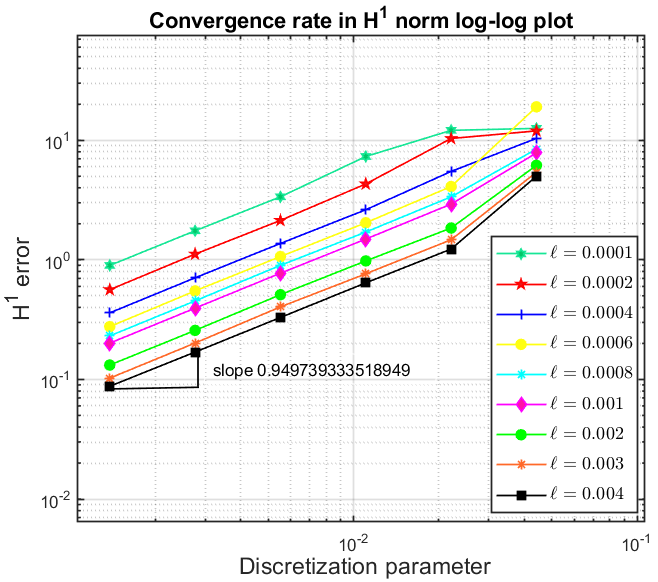}} 
	\hspace{0.2 cm}
	\label{convergence_rates_parameter_dependency_c_negative_H1error1}
	\subfloat[]{\includegraphics[width=7.2cm,height=6cm]{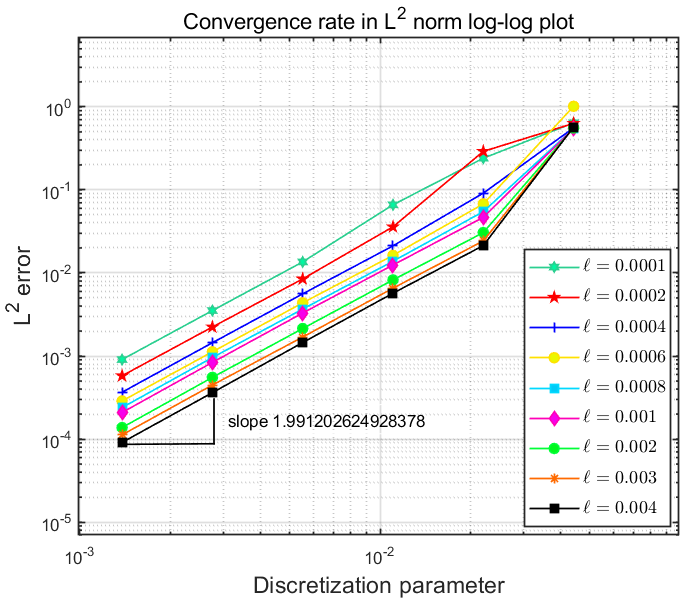}} 
	\caption{Convergence behavior plots of error in the (a) energy norm and (b)  $\mathbf{L}^2$ norm  versus the discretization parameter $h$ for $\Psi_h:=(\Qvec_{D1}^1, \Mvec_{D1}^1)$ solution for various values of $\ell$ and $c=-0.25.$}
	\label{convergence_rates_parameter_dependency_c_negative}
\end{figure}
	\begin{rem}
The Landau-de Gennes energy for nematic liquid crystal 
	 is defined  \cite{MultistabilityApalachong} as 
	\begin{align*} %\label{Landau-de Gennes energy functional nlc}
	F_{\text{nem}}(\bar{{\Psi}}_\epsilon) =\int_\Omega (\abs{\nabla \bar{{\Psi}}_\epsilon}^2 + 
	\epsilon^{-2}(\abs{\bar{{\Psi}}_\epsilon}^2  - 1)^2) \dx, 
	\end{align*}
	where  $\bar{{\Psi}}_\epsilon:=(u_1,u_2) = \bar{g} $ on $\partial \Omega$ and $\epsilon$ is a material-dependent parameter that depends on the elastic constant, domain size and temperature. 
	The Euler-Lagrange equations are a system of second order non-linear elliptic partial differential equations  that seeks $\bar{{\Psi}}_\epsilon\in X \times X   $ such that for all $\bar{\varphi}:=(\varphi_1, \varphi_2) \in V \times V $
	\begin{align*}
	&	 \int_{\Omega} \nabla u_1 \cdot \nabla \varphi_1 \dx+2\epsilon^{-2} \int_{\Omega} (u_1^2 + u_2^2-1)u_1 \varphi_{1} \dx=0,\notag\\ &
	\int_{\Omega} \nabla u_2 \cdot \nabla \varphi_2 \dx+ 2\epsilon^{-2}\int_{\Omega} (u_1^2 + u_2^2-1)u_2 \varphi_{2} \dx=0,
	\end{align*}
	which is \eqref{continuous nonlinear ferro} for ${\Psi}:=(u_1,u_2, 0,0)$ and ${\Phi}:=(\varphi_1,\varphi_2, 0,0)$ with the parameter values $\ell =\frac{\epsilon^2}{2}.$
	Note that the non-linearity in reduced Landau-de Gennes minimization problem is cubic. The quadratic  non-linear term $B_2(\cdot, \cdot, \cdot)$ in \eqref{continuous nonlinear ferro} is zero here. An {\it a priori} error analysis with $h-\epsilon$ dependency has been discussed  for this model for discontinuous Galerkin method in \cite{DGFEM}. The analysis for conforming finite element method is a special case of the problem considered in this paper. \qed
\end{rem}

\section{Conclusions}
We study the minimizers of a $\Qvec$-tensor -$\Mvec$ model for dilute ferronematic suspensions in 2-D framework. The energy functional has two parameters- a scaled elastic parameter $\ell$ and nemato-magnetic coupling parameter $c$. We analyze the asymptotic behavior of the global minimizers of $\tilde{\mathcal{E}}$, as $\ell \rightarrow 0$ and establish that $\vertiii{\Psi^{\ell}}_2 $ is bounded, independent of $\ell.$ This result plays a key role in the $h-\ell$ dependent finite element analysis for regular solutions of the corresponding Euler-Lagrange PDEs  such that \eqref{H2 bound for minimizers} holds. Whether the analysis holds for all regular solutions of the Euler-Lagrange PDEs is proposed as future work. The numerical results focus on the solution landscapes for $\ell = 0.001, c=0, \pm 0.25$  in a square domain, the convergence rates in energy and $\mathbf{L}^2$ norms, the convergence behavior of discrete solutions for various values of $\ell$. The numerical results in this manuscript can be extended to stable solutions, for other values of $c$ and $\ell$ as reported in \cite{Ferronematics_2D}. The convergence of minimizers in $\mathbf{L}^{\infty}$ norm i.e., $\vertiii{\Psi^{\ell}-\Psi_0}_{\infty}$ estimates, the analysis for three-dimensional geometries, the finite element analysis for polygonal domains with re-entrant corners and  Dirichlet boundary data with lesser regularity, {\it a posteriori} error analysis to investigate the effects of defects on numerical errors are interesting and challenging extensions of this work. Moreover, the asymptotic analysis of minimizers with topologically non-trivial boundary conditions and/or %relaxed boundary data regularity assumption,  $\text{deg}(\mathbf{g})=0,$ 
$\ell$-dependent Dirichlet boundary data $\mathbf{g}_{\ell}$, including star-shaped domains, are further interesting areas to be  investigated. 	    
	
\section*{Acknowledgements}
R. R. Maity gratefully acknowledges Professor Yiwei Wang for his illuminating suggestions in numerical computations, as well as Professor Giacomo Canevari for helpful discussions. R. R. Maity also acknowledges the support from institute Ph.D. fellowship. N. N. gratefully acknowledges SERB POWER Fellowship  SPF/2020/000019. A.M. acknowledges support from the University of Strathclyde Global Engagement Fund and an OCIAM Visiting Fellowship, Visiting Professorship from the University of Bath. A.M. acknowledges support from a Leverhulme International Academic Fellowship. A.M. also acknowledges support from the DST-UKIERI for the project on "Theoretical and experimental studies of suspensions of magnetic nanoparticles, their applications and generalizations".

\bibliographystyle{amsplain}
\bibliography{ReferencesLDG}

\providecommand{\bysame}{\leavevmode\hbox to3em{\hrulefill}\thinspace}
\providecommand{\MR}{\relax\ifhmode\unskip\space\fi MR }
\providecommand{\MRhref}[2]{%
  \href{http://www.ams.org/mathscinet-getitem?mr=#1}{#2}
}
\providecommand{\href}[2]{#2}
\begin{thebibliography}{10}

\bibitem{Bethuel}
F.~Bethuel, H.~Brezis, and F.~H\'{e}lein, \emph{\space {A}symptotics for the
  minimization of a {G}inzburg-{L}andau functional}, Calculus of Variations and
  Partial Differential Equations \textbf{1} (1993), no.~2, 123--148.

\bibitem{Brezis_Bethual_Book}
\bysame, \emph{Ginzburg-{L}andau vortices}, Modern Birkh\"{a}user Classics,
  Birkh\"{a}user/Springer, Cham, 2017, Reprint of the 1994 edition.

\bibitem{Ferronematics_2D}
K.~Bisht, Y.~Wang, V.~Banerjee, and A.~Majumdar, \emph{Tailored morphologies in
  two-dimensional ferronematic wells}, Physical Review E \textbf{101} (2020),
  022706.

\bibitem{Nochetto2020}
J.~P. Borthagaray, R.~H. Nochetto, and S.~W. Walker, \emph{A
  structure-preserving {FEM} for the uniaxially constrained {Q}-tensor model of
  nematic liquid crystals}, Numerische Mathematik \textbf{145} (2020), no.~4,
  837--881.

\bibitem{Brochard_DE_Gennes_1970}
F.~Brochard and P.~G. de~Gennes, \emph{Theory of magnetic suspensions in liquid
  crystals}, Journal De Physique \textbf{31} (1970), no.~7, 691--708.

\bibitem{Burylov_Raikher_1995}
S.~V. Burylov and Y.~L. Raikher, \emph{Macroscopic properties of ferronematics
  caused by orientational interactions on the particle surfaces. {I}. extended
  continuum model}, Molecular Crystals and Liquid Crystals Science and
  Technology. Section A. \textbf{258} (1995), no.~1, 107--122.

\bibitem{calderer2014}
M.~C. Calderer, A.~DeSimone, D.~Golovaty, and A.~Panchenko, \emph{An effective
  model for nematic liquid crystal composites with ferromagnetic inclusions},
  SIAM Journal on Applied Mathematics \textbf{74} (2014), no.~2, 237--262.

\bibitem{ciarlet}
P.~G. Ciarlet, \emph{The finite element method for elliptic problems}, Classics
  in Applied Mathematics, vol.~40, Society for Industrial and Applied
  Mathematics (SIAM), Philadelphia, PA, 2002.

\bibitem{Cirtoaje_Petrescu_Stan_Creanga_2016}
C.~Cîrtoaje, E.~Petrescu, C.~Stan, and D.~Creangă, \emph{Ferromagnetic
  nanoparticles suspensions in twisted nematic}, Physica E: Low-dimensional
  Systems and Nanostructures \textbf{79} (2016), 38 -- 43.

\bibitem{JDPEFAMJX}
J.~Dalby, Farrell~P. E., Majumdar A., and J.~Xia, \emph{One-dimensional
  ferronematics in a channel: order reconstruction, bifurcations and
  multistability}, https://arxiv.org/abs/2102.06347 (2021).

\bibitem{FiniteelementanalysisLDG}
T.~A. Davis and E.~C. Gartland, Jr., \emph{Finite element analysis of the
  {L}andau-de {G}ennes minimization problem for liquid crystals}, SIAM Journal
  on Numerical Analysis \textbf{35} (1998), no.~1, 336--362.

\bibitem{dg}
P.~de~Gennes and J.~Prost, \emph{The physics of liquid crystals}, International
  Series of Monogr, Clarendon Press, 1993.

\bibitem{Ern}
D.~A. Di~Pietro and A.~Ern, \emph{Mathematical aspects of discontinuous
  {G}alerkin methods}, Math\'{e}matiques \& Applications (Berlin) [Mathematics
  \& Applications], vol.~69, Springer, Heidelberg, 2012.

\bibitem{ErnGuermond}
A.~Ern and J-L. Guermond, \emph{Theory and practice of finite elements},
  Applied Mathematical Sciences, vol. 159, Springer-Verlag, New York, 2004.

\bibitem{Evance19}
L.~C. Evans, \emph{Partial differential equations}, second ed., Graduate
  Studies in Mathematics, vol.~19, American Mathematical Society, Providence,
  RI, 2010.

\bibitem{Canevari2015}
C.~Giacomo, \emph{Biaxiality in the asymptotic analysis of a 2{D} {L}andau--de
  {G}ennes model for liquid crystals}, ESAIM. Control, Optimisation and
  Calculus of Variations \textbf{21} (2015), no.~1, 101--137.

\bibitem{golovaty2015}
D.~Golovaty, J.~Alberto~Montero, and P.~Sternberg, \emph{Dimension reduction
  for the {L}andau-de {G}ennes model in planar nematic thin films}, Journal of
  Nonlinear Science \textbf{25} (2015), 1431 – 1451.

\bibitem{grisvard}
P.~Grisvard, \emph{Elliptic problems in nonsmooth domains}, Classics in Applied
  Mathematics, vol.~69, Society for Industrial and Applied Mathematics (SIAM),
  Philadelphia, PA, 2011.

\bibitem{Gunzburger1992}
M.~D. Gunzburger and S.~L. Hou, \emph{Treating inhomogeneous essential boundary
  conditions in finite element methods and the calculation of boundary
  stresses}, SIAM Journal on Numerical Analysis \textbf{29} (1992), no.~2,
  390--424.

\bibitem{Hanetal2021Ferro2D}
Y.~Han, J.~Harris, Majumdar A., and J.~Walton, \emph{Tailored morphologies in
  two-dimensional ferronematic wells}, Physical Review E (2021).

\bibitem{han_majumdar_zhang_siap}
Y.~Han, A.~Majumdar, and L.~Zhang, \emph{A reduced study for nematic equilibria
  on two-dimensional polygons}, SIAM Journal on Applied Mathematics \textbf{80}
  (2020), no.~4, 1678--1703.

\bibitem{keller}
H.~B. Keller, \emph{Approximation methods for nonlinear problems with
  application to two-point boundary value problems}, Mathematics of Computation
  \textbf{29} (1975), 464--474.

\bibitem{KesavaTopicsFunctinal}
S.~Kesavan, \emph{Topics in functional analysis and applications}, John Wiley
  \& Sons, Inc., New York, 1989.

\bibitem{scalia_review}
J.~P.F. Lagerwall and G.~Scalia, \emph{A new era for liquid crystal research:
  Applications of liquid crystals in soft matter nano-, bio- and
  microtechnology}, Current Applied Physics \textbf{12} (2012), no.~6,
  1387--1412.

\bibitem{Lions_Magenes}
J.-L. Lions and E.~Magenes, \emph{Non-homogeneous boundary value problems and
  applications. {V}ol. {I}}, Springer-Verlag, New York-Heidelberg, 1972,
  Translated from the French by P. Kenneth, Die Grundlehren der mathematischen
  Wissenschaften, Band 181.

\bibitem{MultistabilityApalachong}
C.~Luo, A.~Majumdar, and R.~Erban, \emph{Multistability in planar liquid
  crystal wells}, Physical Review E \textbf{85} (2012), 061702.

\bibitem{AposterioriRMAMNN}
R.~R. Maity, A.~Majumdar, and N.~Nataraj, \emph{Error analysis of {N}itsche’s
  and discontinuous {G}alerkin methods of a reduced {L}andau–de {G}ennes
  problem}, Computational Methods in Applied Mathematics \textbf{21} (2021),
  no.~1, 179 -- 209.

\bibitem{Majumdar2010}
A.~Majumdar and A.~Zarnescu, \emph{Landau-de {G}ennes theory of nematic liquid
  crystals: the {O}seen-{F}rank limit and beyond}, Archive for Rational
  Mechanics and Analysis \textbf{196} (2010), no.~1, 227--280.

\bibitem{Mertelj_Lisjak_2017}
A.~Mertelj and D.~Lisjak, \emph{Ferromagnetic nematic liquid crystals}, Liquid
  Crystals Reviews \textbf{5} (2017), no.~1, 1--33.

\bibitem{Mertelj_Lisjak_Drofenik_Copic_2013}
A.~Mertelj, D.~Lisjak, M.~Drofenik, and M.~Copič, \emph{Ferromagnetism in
  suspensions of magnetic platelets in liquid crystal}, Nature \textbf{504}
  (2013), no.~7479, 237—241.

\bibitem{RogerMoser2005}
R.~Moser, \emph{Partial regularity for harmonic maps and related problems},
  World Scientific Publishing Co. Pte. Ltd., Hackensack, NJ, 2005.

\bibitem{Nitsche1971}
J.~Nitsche, \emph{Über ein variationsprinzip zur lösung von
  dirichlet-problemen bei verwendung von teilräumen, die keinen
  randbedingungen unterworfen sind}, Abhandlungen aus dem Mathematischen
  Seminar der Universität Hamburg \textbf{36} (1971), no.~1, 9--15.

\bibitem{DGFEM}
R.~R.~Maity, A.~Majumdar, and N.~Nataraj, \emph{Discontinuous {G}alerkin finite
  element methods for the {L}andau--de {G}ennes minimization problem of liquid
  crystals}, IMA Journal of Numerical Analysis \textbf{41} (2021), no.~2,
  1130--1163.

\bibitem{Slavinec_2015}
M.~Slavinec, E.~Klemenčič, M.~Ambrožič, and M.~Krasna, \emph{Impact of
  nanoparticles on nematic ordering in square wells}, Advances in Condensed
  Matter Physics \textbf{2015} (2015), 1--11.

\bibitem{Tsakonas}
C.~Tsakonas, A.~J. Davidson, C.~V. Brown, and N.~J. Mottram, \emph{Multistable
  alignment states in nematic liquid crystal filled wells}, Applied Physics
  Letters \textbf{90} (2007), Article 111913.

\bibitem{LeiZhang}
W.~Wang, L.~Zhang, and P.~Zhang, \emph{Modeling and computation of liquid
  crystals}, Acta Numerica (2022).

\bibitem{canevari_majumdar_spicer}
Y.~Wang, G.~Canevari, and A.~Majumdar, \emph{Order reconstruction for nematics
  on squares with isotropic inclusions: a {L}andau--de {G}ennes study}, SIAM
  Journal on Applied Mathematics \textbf{79} (2019), no.~4, 1314--1340.

\end{thebibliography}

\end{document}